\documentclass[10pt]{amsart}

 \usepackage[usenames, dvipsnames]{xcolor}
\usepackage{amsmath, latexsym, amsfonts, amssymb, amsthm,MnSymbol,stmaryrd}
\usepackage{graphicx,hyperref,dsfont,epsfig,caption,wrapfig,subfig,tikz}
\usetikzlibrary{patterns}
\usetikzlibrary{arrows}
\usetikzlibrary{calc}
\usepackage{xparse}

\usepackage[mathscr]{eucal}
\hypersetup{
    linktoc=page,
    linkcolor=red,          
    citecolor=blue,        
    filecolor=blue,      
    urlcolor=cyan,
    colorlinks=true           
}

\newcommand{\nocontentsline}[3]{}
\newcommand{\tocless}[2]{\bgroup\let\addcontentsline=\nocontentsline#1{#2}\egroup}
\numberwithin{equation}{section}

\newtheorem{theorem}{Theorem}[section]
\newtheorem{lemma}[theorem]{Lemma}
\newtheorem{proposition}[theorem]{Proposition}
\newtheorem{corollary}[theorem]{Corollary}

\newtheorem{remark}[theorem]{Remark}
\newtheorem{definition}[theorem]{Definition}
\newtheorem{assumption}[theorem]{Assumption}

\renewcommand{\tilde}{\widetilde}          
\DeclareMathSymbol{\leqslant}{\mathalpha}{AMSa}{"36} 
\DeclareMathSymbol{\geqslant}{\mathalpha}{AMSa}{"3E} 
\renewcommand{\leq}{\;\leqslant\;}                   
\renewcommand{\geq}{\;\geqslant\;}                   
\newcommand{\dd}{\text{\rm d}}             

\newcommand{\C}{\mathbb{C}}

\newcommand{\D}{\mathbb{D}}
\newcommand{\R}{\mathbb{R}}
\newcommand{\Z}{\mathbb{Z}}
\newcommand{\N}{\mathbb{N}}

\newcommand{\E}{\mathds{E}}
\renewcommand{\P}{\mathds{P}}

\newcommand{\cjd}{\rangle}
\newcommand{\cjg}{\langle}

\newcommand{\hf}{\frac{_1}{^2}}

\newcommand{\pl}{\partial}
\newcommand{\bbar}{\overline}
\newcommand{\mc}{\mathcal}
\newcommand{\la}{\lambda}

\def\eps{\varepsilon}

\def\T{\mathbb{T}}

\def\bi{\begin{itemize}}
\def\ei{\end{itemize}}

\def\bnum{\begin{enumerate}}
\def\enum{\end{enumerate}}
\def\<#1{\langle #1 \rangle}

\newcommand{\cS}{\mathcal{S}}

\newcommand{\bT}{\mathbf{T}}

\newcommand{\bL}{\mathbf{L}}
\newcommand{\bv}{\mathbf{v}}
\newcommand{\del}{\partial}
\newcommand{\bH}{\mathbf{H}}

\renewcommand{\d}{\mathrm{d}}
\newcommand{\A}{\mathbb{A}}
\newcommand{\bvarphi}{\boldsymbol{\varphi}}

\newcommand{\cC}{\mathcal{C}}
\renewcommand{\Re}{\mathrm{Re}}
\newcommand{\cH}{\mathcal{H}}

\newcommand{\bA}{\mathbf{A}}

\newcommand{\bzeta}{\boldsymbol{\zeta}}
\newcommand{\bs}{\boldsymbol}

\author{Guillaume Baverez}
\address{Beijing International Center for Mathematical Research, Peking University}
\email{guillaume.baverez@bicmr.pku.edu.cn}

\author{Colin Guillarmou}
\address{Universit\'e Paris-Saclay, CNRS,  Laboratoire de math\'ematiques d'Orsay, 91405, Orsay, France.}
\email{colin.guillarmou@math.u-psud.fr}

\author{Antti Kupiainen}
\address{University of Helsinki, Department of Mathematics and Statistics}
\email{antti.kupiainen@helsinki.fi}

\author{R\'emi Rhodes}
\address{Aix Marseille Univ, CNRS, Centrale Marseille, I2M, Marseille, France
and Institut Universitaire de France (IUF)}
\email{remi.rhodes@univ-amu.fr}

\title{Semigroup of annuli in Liouville CFT}

\begin{document}

 \begin{abstract} 
In conformal field theory, the semigroup of annuli with boundary parametrisation plays a special role, in that it generates the whole algebra of local conformal symmetries, the so-called Virasoro algebra.  The subgroup of elements $\A_f=\D\setminus f(\D^\circ)$ 
for contracting biholomorphisms $f:\D\to f(\D)\subset \D^\circ$ with $f(0)=0$ is called the holomorphic semigroup of annuli.
In this article, we construct a differentiable representation of the holomorphic semigroup on the space of bounded operators on the Hilbert space $\mc{H}$ of Liouville Conformal Field Theory. We   show  that it generates under differentiation the positive Virasoro elements ${\bf L}_n,\tilde{{\bf L}}_n$ for $n\geq 0$. We also construct a projective representation of the semigroup of annuli in the space of bounded operators on $\mc{H}$ in terms of Segal amplitudes and show that all Virasoro elements ${\bf L}_n,\tilde{{\bf L}}_n$ for $n\in \Z$ are generated by differentiation of these annuli amplitudes. Finally, we use this to show that the Segal amplitudes for Liouville theory are differentiable with respect to their boundary parametrisations, and the differential is computed in terms of Virasoro generators. 
This paper will serve, in a forthcoming work, as a fundamental tool in the construction of conformal blocks as globally defined holomorphic sections of a holomorphic line bundle on Teichm\"uller space and satisfying the Ward identities. 
  \end{abstract}

\maketitle

\section{Introduction}

\subsection{Liouville  conformal field theory and its conformal blocks}
The Liouville conformal field theory (CFT in short) is a $2$-dimensional CFT whose path integral representation in physics is given by 
\[ \int_{{\rm Maps}(\Sigma,\R)}F(\phi)e^{-S_{\Sigma,g}(\phi)}D\phi\]
where $(\Sigma,g)$ is a Riemannian surface with scalar curvature denoted $K_g$ (and volume form ${\rm dv}_g$) and the action is 
\[ S_{\Sigma,g}(\phi)=\frac{1}{4\pi}\int_{\Sigma}(|d\phi|^2_g+QK_g\phi+\mu e^{\gamma\phi}){\rm dv}_g\]  
for some parameters $\gamma\in (0,2)$, $Q:=\gamma/2+2/\gamma$ and $\mu>0$. Here $F$ denotes a functional on the set of maps on $\Sigma$ and $D\phi$ stands for the ``uniform measure'' (which  is not mathematically defined). 
This theory was introduced by Polyakov \cite{Polyakov81} and has been studied extensively in physics, see for example \cite{DornOtto94,Zamolodchikov96,Teschner_revisited}. 
The natural observables $F(\phi)$ for this theory are products  
\[ F(\phi)=\prod_{j=1}^m V_{\alpha_j}(x_j,\phi)\]
of the functionals $V_{\alpha_j}(x_j,\phi):=e^{\alpha_j\phi(x_j)}$ called \emph{primary fields}, with  real $\alpha_j<Q$ and pairwise distinct points $x_j\in \Sigma$ on the surface. The (formal) path integral representing these observables 
\begin{equation}\label{correlIntro}
\cjg \prod_{j=1}^m V_{\alpha_j}(x_j)\cjd_{\Sigma,g}=\int_{{\rm Maps}(\Sigma,\R)}\prod_{j=1}^mV_{\alpha_j}(x_j,\phi)e^{-S_{\Sigma,g}(\phi)}D\phi
\end{equation}
is called a \emph{correlation function}; when $m=0$, the path integral of the function $F=1$ is called the \emph{partition function}. 
In physics, it is postulated that the 
correlation functions $\cjg \prod_{j=1}^mV_{\alpha_j}(x_j)\cjd_{\Sigma,g}$ and partition functions on any surface $(\Sigma,g)$ of genus $g(\Sigma)$ can be expressed entirely in terms of: 
\begin{enumerate}
\item The $3$-point function on the Riemann sphere $\hat{\C}\simeq \mathbb{S}^2$, called the \emph{structure constant}, 
\[C(\alpha_1,\alpha_2,\alpha_3):=\cjg V_{\alpha_1}(0)V_{\alpha_2}(1)V_{\alpha_3}(\infty)\cjd_{\mathbb{S}^2,g_{\mathbb{S}^2}},\] $g_{\mathbb{S}^2}$ being the round metric. 
\item  Holomorphic functions $\mc{F}_{\bs{\alpha},{\bf p}}(\Sigma,[g])$ in the moduli parameters, called \emph{conformal blocks}, where $\bs{\alpha}=(\alpha_1,\dots,\alpha_m)$ and ${\bf p}=(p_1,\dots,p_{3g(\Sigma)-3+m})\in \R_+^{3g(\Sigma)-3+m}$ are extra parameters associated to the spectrum of a certain operator and $[g]$ denotes the conformal class of $g$, which is viewed as a parameter on the  moduli space of surfaces of genus $g(\Sigma)$ with $m$ marked points. These functions are associated with the algebra of symmetries of the model, called the \emph{Virasoro algebra}. They are characterized   as solutions of certain equations called \emph{Ward identities}.
\end{enumerate}
The factorization, called \emph{conformal bootstrap}, has the form 
\begin{equation}\label{bootstrap1} 
\cjg \prod_{j=1}^mV_{\alpha_j}(x_j)\cjd_{\Sigma,g}= C(g)\int_{\R^{3g(\Sigma)-3+m}_+}\bs{\rho}(\bs{\alpha},{\bf p})|\mc{F}_{\bs{\alpha},{\bf p}}(\Sigma,[g])|^2\d {\bf p}
\end{equation}
where $\bs{\rho}(\bs{\alpha},{\bf p})$ is a product of $2g(\Sigma)-2+m$ structure constants with parameters picked from the collection of $\bs{\alpha}$ and ${\bf p}$, $C(g)$ is an explicit constant depending on the choice of conformal representative $g$ in $[g]$.

The picture drawn above is expected from physics but it is still incomplete mathematically.
Recently, a mathematical construction of the Liouville conformal field theory has been achieved in \cite{DKRV16,DRV16_tori,Guillarmou2019} on all closed surfaces using probabilistic methods, and a proof of the bootstrap factorization formula \eqref{bootstrap1} was shown in \cite{GKRV20_bootstrap}  for the $4$-point function on the sphere and for general surfaces 
in \cite{GKRV21_Segal}. However, in these works, the conformal blocks  $\mc{F}_{\bs{\alpha},{\bf p}}(\Sigma,[g])=\mc{F}_{\bs{\alpha},{\bf p}}(\Sigma,[g],\mc{C})$ depend not only on the surface $(\Sigma,[g])$ but also  on a family $\mc{C}=(\mc{C}_1,\dots,\mc{C}_{3g(\Sigma)-3+m})$ of $3g(\Sigma)-3+m$ simple parametrised curves cutting the surface into geometric elementary pieces, namely pairs of pants, annuli with one marked point and disks with two marked points; see Figure \ref{Genus2}.
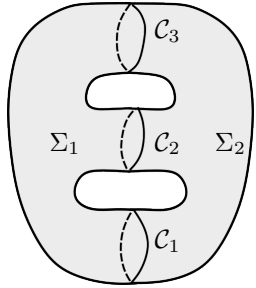
\begin{figure}
 \begin{tikzpicture}
 \node[inner sep=0pt] (pant) at (0,0)
{\includegraphics[width=0.2\textwidth]{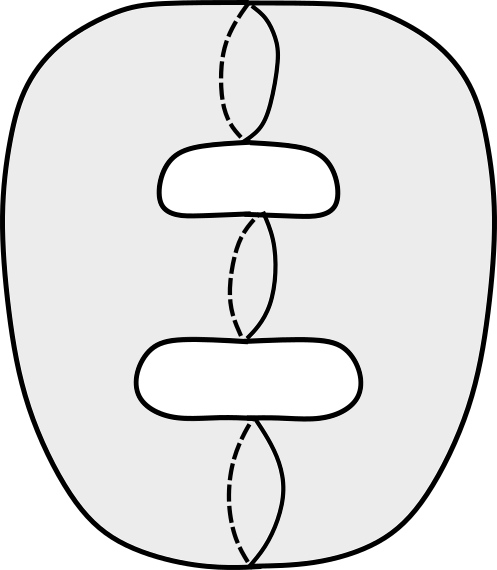}};
 \draw (-1.2,0) node[right,black]{$\Sigma_1$} ;
  \draw (1,0) node[right,black]{$\Sigma_2$} ; 
   \draw (0.2,0) node[right,black]{$\mc{C}_2$} ;
      \draw (0.2,1.5) node[right,black]{$\mc{C}_3$} ;
        \draw (0.2,-1.3) node[right,black]{$\mc{C}_1$} ;
\end{tikzpicture}
 \caption{A genus $2$ surface cut into two pairs of pants $\Sigma_1,\Sigma_2$ along $3$ curves 
        $\mc{C}=(\mc{C}_1,\mc{C}_2,\mc{C}_3)$.}\label{Genus2}
\end{figure}
We call these elementary surfaces geometric building blocks. 
The 
Riemann surface $(\Sigma,[g])$, considered up to biholomorphism, can be parametrised by $3g(\Sigma)-3+m$ complex parameters 
(the moduli space $\mc{M}_{g(\Sigma),m}$) while the cutting curves $\mc{C}$ live in an infinite dimensional space, namely a loop space. In \cite{GKRV21_Segal}, a choice of family $\mc{C}$ of cutting curves on a fixed Riemann surface $(\Sigma_0,[g_0])$ induces a local holomorphic coordinate chart, called plumbing coordinates, on the moduli space $\mc{M}_{g(\Sigma),m}$ near $(\Sigma_0,[g_0])$. From this approach,  $\mc{M}_{g(\Sigma),m}$ can be covered by local holomorphic coordinate charts and it was proved that near $(\Sigma_0,[g_0])$ the conformal blocks can be reduced to local holomorphic functions in these local coordinates patches. 
However, it is important to understand the global picture of the factorization formula \eqref{bootstrap1} and how these local 
functions on moduli space relate  to each other  when going from one coordinate patch to another one: it is postulated in physics that the conformal blocks should be global functions (or sections of line bundles) on the   universal cover of $\mc{M}_{g(\Sigma),m}$, called  Teichm\"uller space $\mc{T}_{g(\Sigma),m}$.
The global construction of the conformal blocks should then lead to a projective unitary representation of the mapping class group in the space of conformal blocks and the Verlinde conjecture in this setting is that this representation is equivalent to the quantisation of Teichm\"uller space, see Teschner \cite{Teschner03,Teschner04,Teschner_Teich}.

Our goal now is to perform the global construction of the conformal blocks and to show that they satisfy the Ward identities. This is done through a precise analysis of the dependence of $\mc{F}_{\bs{\alpha},{\bf p}}(\Sigma,[g],\mc{C})$ on the curves $\mc{C}$ 
(with each $\mc{C}_j$ in a fixed free homotopy class). However, showing differentiability with respect to $\mc{C}$ and computing the variation of the conformal block turns out to be involved, mainly because of the irregularity of the random variables used to define the correlation functions and the conformal blocks. The present paper provides the main tool for this analysis: 
the representation of the \emph{Segal semigroup of annuli} into a space of bounded operators on a well-chosen Hilbert space 
and, more importantly, the differentiability of this representation and the computation of the differential in terms of the Virasoro algebra. 
In the forthcoming companion paper \cite{BGKR2}, we address the global construction of the conformal blocks using the tools developed here, and the representation of the mapping class  group in the space of conformal blocks.

\medskip

Let us now recall this formalism in the probabilistic setting. The mathematical construction of this conformal field theory in \cite{DKRV16,DRV16_tori,Guillarmou2019} 
on all closed surfaces uses probabilistic methods, in particular the Gaussian Multiplicative Chaos, which allows one to define the random measure 
\[ M_\gamma(dx)= e^{\gamma X_g(x)}{\rm dv}_g(x)\]
where $X_g$ is the Gaussian Free Field, a random variable with distributional values and covariance given by the Green's function of the Laplacian $\Delta_g$ on $\Sigma$. 
The correlation functions \eqref{correlIntro} are defined as  expectations of certain functionals of the Gaussian Free Field and are finite (and non-trivial) when the weights satisfy the Seiberg bounds
\begin{equation}\label{Seiberg_bounds_intro} 
\sum_{j=1}^m\alpha_j-Q\chi(\Sigma)>0, \quad \alpha_j<Q,
\end{equation}
$\chi(\Sigma)$ being the Euler characteristic of $\Sigma$.

The proof of the  bootstrap factorization  formulas \eqref{bootstrap1} in 
\cite{GKRV20_bootstrap,GKRV21_Segal} starts from the probabilistic construction and combines it with geometric 
and analytic input. In particular, it was shown that the probabilistic approach fits perfectly well with the axioms 
introduced by Segal. 
In \cite{Segal87}, Segal introduced a  geometric  approach to axiomatize CFT, based on the existence of a Hilbert space $\mc{H}$ and on a correspondence between Riemannian surfaces with parametrised boundary (and possibly marked points with weights) 
 and operators acting on tensor products of $\mc{H}$. In this correspondence, we consider 
an  oriented Riemannian surface $(\Sigma,g)$, $m$ disjoint marked points ${\bf x}=(x_1,\dots,x_m)$ and real analytic parametrisations $\zeta_j:\T\to \pl_j\Sigma$ of the boundary circles (here $\T\subset \C$ denotes the unit circle) collected in $\bs{\zeta}=(\zeta_1,\dots,\zeta_b)$. 
We denote by $S_g:=(\Sigma,g,{\bf x},\bs{\zeta})$ and let 
$\bs{\alpha}=(\alpha_1,\dots,\alpha_m)$ be associated weights satisfying \eqref{Seiberg_bounds_intro}.
To the collection  $(S_g,\bs{\alpha})$ we associate
an element 
$\mc{A}_{\Sigma,g,{\bf x},\bs{\zeta},\bs{\alpha}}=\mc{A}_{S_g,\bs{\alpha}}\in \mc{H}^{\otimes b^+}\otimes (\mc{H}^*)^{\otimes b^-}$, called \emph{amplitude}, that can  alternatively be viewed as a Hilbert-Schmidt operator 
\[\mc{A}_{S_g,\bs{\alpha}}: \mc{H}^{\otimes b^-}\to \mc{H}^{\otimes b^+}\]
where $b^+$ (resp. $b^-$) is the number of \emph{outgoing} (resp. \emph{incoming}) 
boundary circles $\pl_j\Sigma$, where outgoing (resp. incoming) means that $(\dot{\zeta}_j,\nu)$ is positively (resp. negatively)
oriented in $\Sigma$ if $\nu$ denotes the inward pointing vector at $\pl \Sigma$ and $\dot{\zeta}_j:=\pl_{\theta}\zeta_j(e^{{\rm i}\theta})$ is the vector tangent to the curve $\theta\mapsto \zeta_j(e^{{\rm i}\theta})$. When there is no marked points, we simply write $\mc{A}_{S_g}$ with $S_g=(\Sigma,g,\bs{\zeta})$.
In Segal's formalism, these amplitudes must satisfy certain covariance properties under conformal scaling of the metric and diffeomorphism actions, and they must satisfy a gluing property. Roughly speaking, if we write 
$\bs{\zeta}=(\zeta_1,\dots,\zeta_b)$ for the collection of boundary parametrisation and if we assume there are no marked points to simplify the notation,  a surface $(\Sigma^1,g^1,\bs{\zeta}^1)$ can be glued to another one $(\Sigma^2,g^2,\bs{\zeta}^2)$ by identifying an incoming boundary circle $\pl_k\Sigma^1$ to an outgoing one $\pl_j\Sigma^2$ via the parametrisation, i.e. setting $\zeta^2_j(e^{{\rm i}\theta})=\zeta^1_k(e^{{\rm i}\theta})$, and assuming that the metric $g^1$ and $g^2$ glue smoothly. 
The resulting surface $\Sigma$ has an induced collection of boundary parametrisations $\bs{\zeta}^1$ and $\bs{\zeta}^2$, and an induced metric $g$. The gluing axiom  states that 
the amplitude $\mc{A}_{\Sigma,g,\bs{\zeta}}$ of the glued surface can be written as a certain composition of the amplitude 
$\mc{A}_{\Sigma^2,g^2,\bs{\zeta}^2}$ and $\mc{A}_{\Sigma^1,g^1,\bs{\zeta}^1}$, namely by pairing the $k$-th 
$\mc{H}^*$ component of $\mc{A}_{\Sigma^1,g^1,\bs{\zeta}^1}\in  \mc{H}^{\otimes b_1^+}\otimes (\mc{H}^*)^{\otimes b_1^-}$ with the $j$-th $\mc{H}$ component of  $\mc{A}_{\Sigma^2,g^2,\bs{\zeta}^2}\in  \mc{H}^{\otimes b_2^+}\otimes (\mc{H}^*)^{\otimes b_2^-}$.
For example, if $b_1^-=1$ and $b^+_2=1$, then $\mc{A}_{\Sigma,g,\bs{\zeta}}$ is simply the composition of the operators 
\[ \mc{A}_{\Sigma,g,\bs{\zeta}}=\mc{A}_{\Sigma^1,g^1,\bs{\zeta}^1}\circ \mc{A}_{\Sigma^2,g^2,\bs{\zeta}^2}: \mc{H}^{\otimes b^-_2}\to \mc{H}^{\otimes b^+_1}.\]

The probabilistic construction for the Liouville CFT of the Segal amplitudes is  carried out  in \cite{GKRV21_Segal} 
and it is shown that they satisfy the gluing axioms and the required conformal covariance. The Hilbert space of the theory is
\begin{equation}\label{def_Hilbert}
\mc{H}=L^2(\R \times \Omega_\T,\mu_0), \quad \Omega_\T:=(\R^2)^{\N}, \quad \mu_0:=\d c\otimes \prod_{n=1}^\infty e^{-\frac{x_n^2}{2}-\frac{y_n^2}{2}} \frac{\d x_n\d y_n}{2\pi}
\end{equation}
and the amplitude $\mc{A}_{S_g,\bs{\alpha}}$  is defined as a 
conditional expectation in terms of the Gaussian Free Field and the Gaussian multiplicative chaos measure. 
In this introduction,  we do not give the explicit formula but refer to Definition \ref{def:amp} instead.
If $\D:=\{z\in \C\,|\, |z|\leq 1\}$, the annulus $\A_{e^{-t}}:=\D\setminus e^{-t}\D^\circ$ equipped with the flat metric $g_{\A}=|\dd z|^2/|z|^2$ and 
the parametrisations $\bs{\zeta}=(\zeta_1,\zeta_2)$ given by
$\zeta_2(e^{{\rm i}\theta})=e^{-t+{\rm i}\theta}$ for the interior boundary (incoming) and $\zeta_1(e^{{\rm i}\theta})=e^{{\rm i}\theta}$ for the exterior one (outgoing), generates a contraction semigroup (which is not Hilbert-Schmidt)
\[ e^{-t{\bf H}}:=\frac{e^{-t\frac{  c_{\rm L}}{12}}}{\sqrt{2}\pi}\mc{A}_{\A_{e^{-t}},g_{\A},\bs{\zeta}}:\mc{H}\to \mc{H},  \quad {\rm c}_L=1+6Q^2>25\]
where the generator ${\bf H}$, called the \emph{Hamiltonian}, is an unbounded self-adjoint operator. This operator has continuous spectrum 
and can be diagonalized by a continuous family of eigenstates 
\[  \Psi_{Q+ip,\nu,\tilde{\nu}}, \quad p\in \R_+, \nu, \tilde{\nu}\in \mc{T}\]
where $\mc{T}$ denotes the set of Young diagrams, i.e. the finite decreasing sequences $\nu(1)\geq \dots \geq \nu(k)$ of positive integers. These states belong to the weighted spaces $\bigcap_{\eps>0}e^{-\eps |c|}\mc{H}$ and provide a spectral 
resolution of the identity, analogous to the plane wave eigenbasis $e^{ipc}$ for the Laplacian $-\pl_p^2$ on $L^2(\R_+)$. 
Moreover the eigenstates have the form 
\[\Psi_{Q+ip,\nu,\tilde{\nu}}={\bf L}_{-\nu(k)}\dots {\bf L}_{-\nu(1)}\tilde{{\bf L}}_{-\tilde{\nu}(k')}\dots \tilde{{\bf L}}_{-\tilde{\nu}(1)}\Psi_{Q+ip}\]
where $({\bf L}_n)_{n\in \Z}$  and $(\tilde{{\bf L}}_n)_{n\in \Z}$ are two commuting  Lie algebras of operators  acting on $\mc{H}$, called Virasoro algebra, satisfying 
\[ [{\bf L}_n,{\bf L}_{m}]=(n-m){\bf L}_{n+m} +\frac{c_{\rm L}}{12}(n^3-n)\delta_{n,-m},\]
and similarly for $\tilde{{\bf L}}_n$.  The state $\Psi_{Q+ip}$ (no Young diagrams) is called a primary state, it is killed by 
${\bf L}_n$ and $\tilde{{\bf L}}_n$ if $n>0$ and satisfies ${\bf H}\Psi_{Q+ip}=\frac{1}{2}(Q^2+p^2)\Psi_{Q+ip}$.
This diagonalization and the proof of Segal's axioms in the probabilistic setting are the heart of the proof of the bootstrap formulas 
\eqref{bootstrap1}   in \cite{GKRV20_bootstrap,GKRV21_Segal}. Both in the local picture   \cite{GKRV21_Segal} or the global picture in the forthcoming work \cite{BGKR2}, the conformal blocks are constructed from 
the Segal amplitudes of the geometric building blocks decomposing the surfaces, 
namely from the  evaluation of these amplitudes on the basis elements $\Psi_{Q+ip,\nu,\emptyset}$ for $\nu \in \mc{T}$. 
The variation of the conformal blocks under deformation of the geometric  building-block  decomposition relies on the analysis of the variation of the Segal amplitude of a surface  under deformation of its boundary parametrisation. 
On the other hand, since any boundary component of a surface $\Sigma$  possesses an annular neighborhood (see Figure \ref{var_annulus}), studying the variation of the boundary of a surface boils down to studying the variations of the boundary parametrisation of an annulus amplitude.
 \begin{figure}[h] 
 \begin{tikzpicture}
 \node[inner sep=0pt] (pant) at (0,0)
{\includegraphics[width=0.4\textwidth]{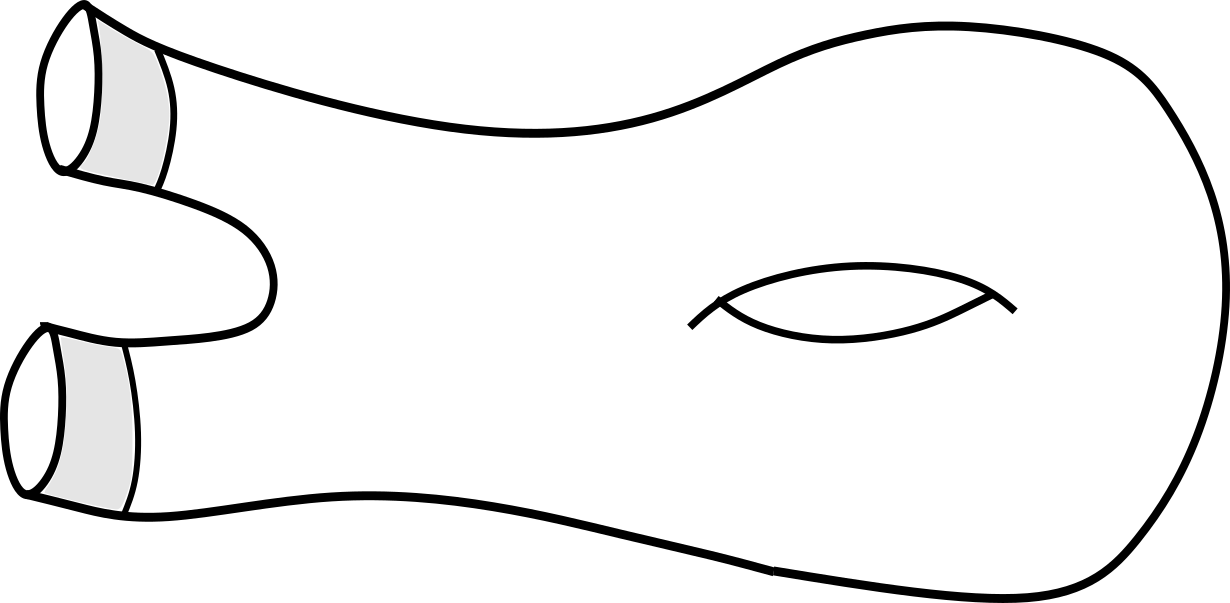}};
\end{tikzpicture}
\caption{Varying boundaries by gluing annuli}\label{var_annulus}
\end{figure}
This is the main motivation of the present work, which focuses on Segal amplitudes of annuli. This family of  annulus  amplitudes in turn generates a representation of the whole Virasoro algebra and a 
remarkable quantisation of the local conformal symmetries.

Let us now discuss the annuli amplitudes and the link  with the Virasoro algebra. In \cite{BGKRV}, the Virasoro operators ${\bf L}_n$ and $\tilde{{\bf L}}_n$ appear as generators of Markov 
semigroups on $\mc{H}$ constructed from the following probabilistic representation: first consider the random variable in $H^{-s}(\T)$ for all $s>0$
\[ \varphi(\theta):=\sum_{n=1}^\infty \frac{x_n+iy_n}{2\sqrt{n}}e^{in\theta}+\sum_{n=1}^\infty \frac{x_n-iy_n}{2\sqrt{n}}e^{-in\theta}\]
where $x_n,y_n$ are viewed as i.i.d. Gaussians (sampled with the Gaussian measure appearing in \eqref{def_Hilbert}), and identify 
$\mc{H}$ with $L^2(H^{-s}(\T))$ using the push-forward of $\mu_0$ by the map $(c,(x_n,y_n)_n)\in \R\times \Omega_\T 
\mapsto c+\varphi\in H^{-s}(\T)$.
Then consider
the meromorphic vector field  ${\rm v}=-(z\pl_z+\eps z^{n+1})\pl_z$ on $\C\setminus \{0\}$ with $\eps\in \C$ small and $n\geq 0$, 
and  let $f_t$ be its flow at time $t$, defined by $\partial_tf_t(z)=v(f_t(z))$ with $f_0(z)=z$.  Define the family of operators, for $t\geq 0$,
\begin{equation}\label{def_T_t}
 \begin{array}{l}
 {\bf T}_t: \mc{H}\to \mc{H}, \\
 {\bf T}_tF(c+\varphi)= |f_t'(0)|^{\frac{Q^2}{2}}\E_\varphi \Big[F\Big(c+X\circ f_t+Q\log\frac{|f_t'|}{|f_t|}\Big)|_{\T})e^{-\mu e^{\gamma c}\int_{\D\setminus f_t(\D)}\frac{e^{\gamma X}}{|x|^{\gamma Q}}\dd x}\Big],
 \end{array}
 \end{equation}
where $X=X_\D +P\varphi$ for $X_\D$ the Dirichlet Gaussian Free Field on $\D$ (a Gaussian variable 
with covariance $\E[X_\D(x)X_\D(x'))]=\log\frac{|1-x\bar{x}'|}{|x-x'|}$) and $P\varphi$ the harmonic extension of $\varphi$ to the disk $\D$, and $\E_\varphi$ means expectation conditionally on $\varphi$. Here, we use the fact that $f_t(\D)\subset \D^\circ$ for $t>0$, and define the random measure $e^{\gamma X}\dd x$ using the Gaussian multiplicative chaos method of \cite{Kahane85}. These operators ${\bf T}_t$ are shown to generate a semigroup $e^{-t\bH_{\rm v}}$ with generator 
\[ \bH_{\rm v}= {\bf L}_0+\tilde{{\bf L}}_0+\eps {\bf L}_n+\bar{\eps}\tilde{{\bf L}}_n.\]
The elements ${\bf L}_{-n}$ for $n>0$ are then constructed by taking the adjoints of ${\bf L}_n$ and similarly for $\tilde{{\bf L}}_{-n}$.

In the present work, we pursue this study and relate these semigroups to the Segal amplitudes of the annuli 
$\A_{f_t}:=\D\setminus f_t(\D)$ and, more generally, we construct a quantisation of the semigroup $\mc{S}$ of 
  biholomorphic contractions $f:\D\to f(\D)\subset \D^\circ$ with $f(0)=0$ and a representation (see \eqref{defS} for the precise definition of $\cS$)
\begin{equation}\label{def_repTf}  
f\in \mc{S} \mapsto {\bf T}_f \in \mc{L}(\mc{H})
\end{equation}
defined by replacing $f_t$ by $f$ in \eqref{def_T_t} (see Definition \ref{defTf}). This semigroup is called \emph{Segal's semigroup of holomorphic annuli}, in reference to the paper  \cite[Section 2]{Segal87} in which Segal makes
 the observation that this semigroup plays a fundamental role in CFT. It was also independently studied by Neretin \cite{Ner87,Ner1990}. For practical reasons, for $\eps>0$ we will define $\mc{S}_\eps$  as the subset of elements $f\in \mc{S}$ that admits a holomorphic extension to $(1+\eps)\D^\circ$ and are smooth on $(1+\eps)\T$, and we identify the tangent space to $\mc{S}_\eps$ to be the space of holomorphic vector fields ${\rm v}=v(z)\pl_z=\sum_{n\in \N}v_nz^{n+1}\pl_z$ in $(1+\eps)\D$. It can be equipped with a Fr\'echet topology using seminorms \eqref{seminorm}. 

A first important intermediate result of the paper is a representation of Segal's holomorphic annuli and its 
differentiability.
\begin{theorem}[Representation of Segal semigroup of annuli]\label{th:main}
The map ${\bf T}: f\mapsto {\bf T}_f$ is a representation of the semigroup $\mc{S}$ in the sense that for $f_1,f_2\in \mc{S}$
\[ {\bf T}_{f_1\circ f_2}={\bf T}_{f_1}\circ {\bf T}_{f_2}.\]
Moreover for $\eps>0$, the map $f\mapsto {\bf T}_f$ is differentiable as a map 
\[ {\bf T}: \mc{S}_\eps \to \mc{L}(\mc{D}(\mc{Q}))\]
where $\mc{D}(\mc{Q})\subset \mc{H}$ is the domain of the quadratic form $\mc{Q}(F)=\cjg {\bf H}F,F\cjd_{\mc{H}}$ associated to the Hamiltonian 
${\bf H}$. Moreover, the differential is given by: for a tangent vector ${\rm v}=v(z)\pl_z=\sum_{n\geq 0}v_nz^{n+1}\pl_z$ to $\mc{S}_\eps$ 
and $F\in \mc{D}(\mc{Q})$
\[ D_{\rm v} {\bf T}_{f}F=-{\bf T}_f{\bf H}_{\rm w}F \]
where ${\rm w}=w(z)\pl_z$ with $w(z)=v(z)/f'(z)=-\sum_{n=0}^\infty w_nz^{n+1}$ and ${\bf H}_{\rm w}$ is the operator 
\[ {\bf H}_{\rm w}=\sum_{n\geq 0}w_n{\bf L}_n+\bbar{w}_n\tilde{{\bf L}}_n: \mc{D}(\mc{Q})\to \mc{D}'(\mc{Q})\]
with $\mc{D}'(\mc{Q})$ the dual space to $\mc{D}(\mc{Q})$.
\end{theorem}
We refer to Theorem \ref{th:differentiabiliteTf} for more details. 

The next question we address is the link between ${\bf T}_f$ and the Liouville amplitude of the annulus $\A_f:=\D\setminus f(\D^\circ)$ for $f\in \mc{S}$. 
This is important to show the differentiability of Segal amplitudes with respect to the boundary parametrisation, and the fact that 
$f\mapsto {\bf T}_f$ is a semigroup is extremely helpful for this task.
For $f\in \mc{S}_\eps$ with $\eps>0$, we prove in Corollary \ref{cor:annulus_propagator} that  the following identity holds 
\[ {\bf T}_fF(\tilde{\varphi}_1)=\frac{e^{-\frac{c_\mathrm{L}}{12}W(f,g_f)} }{ \sqrt{2}\pi }
\int_{H^{-s}(\T)} \mc{A}_{\mathbb{A}_f,g_f,\bzeta}(\tilde{\varphi}_1,\tilde{\varphi}_2)F(\tilde{\varphi}_2) \,\dd\mu_0(\tilde{\varphi}_2)
\]
where $g_f=e^{\omega}g_\A$ is any smooth metric conformal to $g_\A=|\dd z|^2/|z|^2$ on $\A_f$ such that 
$g_f=|\dd z|^2/|z|^2$ near the exterior boundary $\T$ of $\A_f$ and $g_f=f_*(|\dd z|^2/|z|^2)$ near the interior boundary $f(\T)$ of $\A_f$, and 
$W(f,g_f)$ is an explicit deterministic term defined in \eqref{def_of_W}.

The semigroup $\mc{S}$ is said ``of holomorphic annuli" because the outer boundary is prescribed to be the identity map and the inner boundary is parametrised by a holomorphic function in the disc. This is further reflected in the fact that the derivative involves only the operators ${\bf L}_n$, $\tilde{{\bf L}}_n$ with $n\geq 0$. To obtain the whole Virasoro algebra as variations of amplitudes, 
we have to extend this analysis to general annuli.

For an annulus $A$ equipped with a complex structure $J$ and  with parametrised analytic boundary $\bs{\zeta}^A=(\zeta_1,\zeta_2)$ with $\zeta_1$ outgoing and $\zeta_2$ incoming, there is a unique flat metric $g^A=e^{\omega}|\dd z|^2$ compatible with the complex structure of $\C$ with $\zeta_j^*(g^A)=\dd\theta^2$ on $\T$ for $j=1,2$. The set of such complex annuli with parametrised boundary forms a semigroup under gluing where the product $(A_1,\bs{\zeta}^{A_1})\cdot(A_2,\bs{\zeta}^{A_2})$ is the annulus obtained by gluing 
the outgoing component of $A_2$ with the incoming one of $A_1$.
We show in Section \ref{projective_rep} (see Proposition \ref{prop:proj_rep} for a more precise statement) that the map 
\[ (A,\zeta^A)\mapsto \mc{A}_{A,g^A,\bs{\zeta}^A}\in \mc{L}(\mc{H})\]
defines a projective representation of the full semigroup of annuli with parametrised boundary. 
 
Finally, in Theorem \ref{derivative_meromorphic_vf},  we prove that   the annuli $\A_f$, with boundary given by the circle $\T$ and the curve $f(\T)$ 
if $f:\T \to f(\T)\subset \D^\circ$ is analytic, have amplitudes that are differentiable at $f_0:z\mapsto rz$ for $0<r<1$ as a function of  $f$ and their  differentials  generate all ${\bf L}_n$ for $n\in \Z$. As a quite direct corollary (using Segal's gluing property of amplitudes), we obtain our main result, which is the differentiability of Segal's amplitudes with respect to the boundary parametrisations.
If $S_g=(\Sigma,g,{\bf x},\bs{\zeta})$ is a Riemannian surface with admissible metric, marked points, analytic parametrisations $\bs{\zeta}=(\zeta_1,\dots,\zeta_b)$ of the $b$ boundary components, we can glue a disk $\D=\mc{D}_j$ at each $\pl_j\Sigma$ using the parametrisation $\zeta_j:\T\to \pl_j\Sigma$, producing 
a closed Riemann surface $(\hat{\Sigma},\hat{J})$ ($\hat{J}$ denotes its complex structure) 
with $b$ new marked points given by the centers of the $\mc{D}_j$, called filling of $\Sigma$. For $\bs{f}=(f_1\dots,f_b)$ a family of biholomorphisms in the space ${\rm Hol}_\eps(\A)$ of holomorphic functions defined on the annulus $\{z\in \C \,|\, |z|\in [\frac{1}{1+\eps},1+\eps]\}$ and smooth up to the boundary (for $\eps>0$ sufficiently small), with $f_j$ close to the identity, we consider the surface 
$\Sigma^{\bs{f}}\subset \hat{\Sigma}$ (diffeomorphic to $\Sigma$) with boundary components $\zeta_j\circ f_j(\T)$ 
obtained from deforming $\zeta_j$ to $\zeta_j\circ f_j$; the new parametrisation is denoted $\bs{\zeta}^{\bs{f}}=(\zeta_1\circ f_1,\dots,\zeta_{b}\circ f_b)$. We say that a metric $g^{\bs{f}}$ is admissible on $\Sigma^{\bs{f}}$ if it is compatible with $\hat{J}$ and $(\zeta_j\circ f_j)^*g^{\bs{f}}=|\dd z|^2/|z|^2$ near $\T$, we write $\hat{g}$ for the extension of $g$ by $|\dd z|^2/|z|^2$ on each pointed disk $\mc{D}_j\setminus\{0\}$, called filling metric. We define 
\[S_g\cdot \bs{f}:=(\Sigma^{\bs{f}},g^{\bs{f}},{\bf x},\bs{\zeta}^{\bs{f}})\] 
and view the family $\bs{f}$ of biholomorphisms as acting on the Riemannian surface $S_g$, although only the Riemann surface $S\cdot\bs{f}:=(\Sigma^{\bs{f}},\hat{J},{\bf x},\bs{\zeta}^{\bs{f}})$ is canonically defined if $S:=(\Sigma,J,{\bf x},\bs{\zeta})$. The action of $\bs{f}$ on the Riemannian surface $S_g$ involves to make a choice of admissible metric $g^{\bs{f}}$ for each $\bs{f}$.
The space ${\rm Hol}_\eps(\A)$ is equipped with a Fr\'echet topology using the seminorms \eqref{seminorm}, just as $\mc{S}_\eps$ above. Finally, the Liouville action $ S_{\rm L}^0$ is given by (recall $K_g$ denotes the scalar curvature of $g$)
\begin{equation}\label{SL0intro}
S_{\rm L}^0(\Sigma,g_0,g):=\frac{1}{96\pi}\int_{\Sigma}(|d\omega|_{g_0}^2+2K_{g_0}\omega) {\rm dv}_{g_0},
\end{equation}
if $g_0,g$ are two conformally related metrics, i.e. $g=e^{\omega}g_0$, for some smooth $\omega:\Sigma\to \R$. 
We prove:

\begin{theorem}[Differential of Segal amplitudes]\label{intro_diff_amplitudes}
Let $S_g:=(\Sigma,g,{\bf x},\bs{\zeta})$ be a Riemannian surface  with $b$ incoming analytic parametrised boundaries, $\bs{\alpha}$ some weights attached to ${\bf x}$ satisfying \eqref{Seiberg_bounds_intro}, and let 
$(\hat{\Sigma},\hat{J})$ be the filling of $\Sigma$ and $\hat{g}$ the filling metric.
There is a neighborhood $\mc{U}$ of  ${\rm Id}=({\rm Id},\dots,{\rm Id})$ in
 ${\rm Hol}_\eps(\A)^b$ such that, if $\bs{f}\in \mc{U}\mapsto g^{\bs{f}}$ is a $C^1$ family of admissible metrics on $\Sigma^{\bs{f}}$ with $g^{\bs{f}}=g$ outside a small neighborhood of $\pl \Sigma$, then 
\[  \mc{A}_{S_g\cdot,\bs{\alpha}}: \bs{f} \in \mc{U}\mapsto \mc{A}_{S_g\cdot \bs{f},\bs{\alpha}}\in \mc{L}(\mc{D}(\mc{Q})^{\otimes b},\C)\]
is differentiable, with $S_g\cdot \bs{f}=(\Sigma^{\bs{f}},g^{\bs{f}},\bf{x},\bs{\zeta}^{\bs{f}})$. The  differential at $\bs{f}={\rm Id}$ given by: for any ${\bf u}\in \mc{D}(\mc{Q})^{\otimes b}$,
\[\begin{split} 
D_{\bv} \mc{A}_{S_g\cdot,\bs{\alpha}}({\rm Id}){\bf u}=&-\mc{A}_{S_g,\bs{\alpha}} \big( {\bf H}_{{\rm v}_1}\otimes {\rm Id}\otimes \dots \otimes {\rm Id}\big){\bf u}-\dots-\mc{A}_{S_g,\bs{\alpha}}\big( {\rm Id}\otimes\dots \otimes {\rm Id}\otimes {\bf H}_{{\rm v}_b}\big){\bf u}\\
& -c_{\rm L}\Big(\sum_{j=1}^b\frac{{\rm Re}(v_{j0})}{12}+
D_{\bv}\mc{S}_{S_g\cdot }({\rm Id})\Big)\mc{A}_{S_g,\bs{\alpha}}{\bf u}
\end{split}\]
for $\bv=({\rm v}_1,\dots,{\rm v}_b)\in {\rm Hol}_\eps(\A)^b$ and ${\rm v}_j=\sum_{n\in \Z}v_{jn}z^{n+1}\pl_z$,
where ${\bf H}_{{\rm w}}:=\sum_{n\in \Z}w_n{\bf L}_n+ \bar{w}_n\tilde{{\bf L}}_{n}$ if ${\rm w}$ is the holomorphic vector field 
${\rm w}=-\sum_{n\in \Z}w_nz^{n+1}\pl_z$ defined near $\T$, and $\mc{S}_{S_g\cdot }:\mc{U}\to \R$ is the map  defined by $\mc{S}_{S_g\cdot \bs{f}}:=S_{\rm L}^0(\Sigma^{\bs{f}},g^{\bs{f}},\hat{g})$.
\end{theorem}

This theorem builds on Theorem \ref{th:main} and
will be fundamental in the global construction of conformal blocks on Teichm\"uller space 
in \cite{BGKR2} and the proof that they satisfy the Ward identities postulated in physics.

\vskip 2mm
\noindent\textbf{Data availability statement.}  No data attached to this research.

\vskip 2mm
\noindent\textbf{ Conflict of interest.}  The authors declare no conflicts of interest.

\vskip 2mm
\noindent\textbf{Acknowledgements.}  We thank D. Radnell,  E. Schippers, J. Teschner and V. Vargas for useful discussions. G. Baverez is supported by the National Natural Science Foundation of China (Grant No. 12526204).   A. Kupiainen acknowledges the support of the ERC Advanced Grant 741487 and of Academy of Finland. R. Rhodes is partially supported by the Institut Universitaire de France (IUF). R. Rhodes  acknowledges the support of the ANR-21-CE40-0003.

\section{Background}

First, we recall some material from \cite{Guillarmou2019,GKRV20_bootstrap,GKRV21_Segal} that will be used to construct the Liouville amplitudes.
Along the paper, we typically use the variable $z$ for the complex coordinate in $\C$ and $x$ for the Euclidean coordinate in $\R^2$, keeping in mind that $z$ is identified with $x$ via $x=({\rm Re}(z),{\rm Im}(z))$. The Euclidean metric is denoted by $|\dd z|^2$.

\subsection{Geometric background} We begin by reviewing the geometric background on compact Riemann surfaces, 
Green's functions, determinants of Laplacians, that will be needed for the construction of Liouville conformal field theory.\\

\noindent \textbf{Closed Riemann surface.}  
A closed Riemann surface $(\Sigma,J)$ is a smooth oriented compact surface $\Sigma$ with no  boundary, equipped with a complex structure $J$, i.e. $J \in {\rm End}(T\Sigma)$ with $J^2=-{\rm Id}$. Alternatively, this amounts to having a 
set of charts $\omega_j:U_j\to \D\subset \C$ such that $\omega_j\circ \omega_{k}^{-1}$ are biholomorphic where defined, where $\D=\{z\in \C\,| \, |z|<1\}$ is the unit disk. 
For each $k$, one has $J=\omega_k^*J_\C$ with $J_\C$ the canonical complex structure given by $J_\C\pl_{x_1}=\pl_{x_2}, J_\C\pl_{x_2}=-\pl_{x_1}$ in $\C$ if $z=x_1+{\rm i}x_2$.
The orientation on $\Sigma$ is a non-vanishing $2$-form $w_\Sigma\in C^\infty(\Sigma;\Lambda^2T^*\Sigma)$, and by convention we  require it to be compatible with the complex structure, i.e. $(\omega_j^{-1})^*w_\Sigma=e^{f_j} dx_1\wedge dx_2$ in $\D$ 
for some function $f_j$.\\

\noindent \textbf{Compact Riemann surface with parametrised analytic boundary.}  A compact Riemann surface $(\Sigma,J)$ with real analytic boundary $\partial\Sigma=\sqcup_{j=1}^{b}\pl_j \Sigma$ (here $\pl_j\Sigma$ are the connected components, which are analytic circles) is a 
compact oriented surface with smooth boundary with a family of charts 
 $\omega_j:V_j\to \omega_j(V_j)\subset \C$ for $j=1,\dots,j_0$ where $\cup_j V_j$ is an open cover of $\Sigma$ and $\omega_k \circ \omega_j^{-1}$ are biholomorphic where  defined, and $\omega_j(V_j\cap\pl \Sigma)$ is a real analytic curve if $V_j\cap\pl\Sigma\not=\emptyset$. 
Using the Riemann uniformisation theorem, we can moreover assume that for $j\in [1,b]$, $V_j$ are neighborhoods of $\pl_j \Sigma$ with 
$\omega_j(V_j)=\mathbb{A}_\delta$ where $\mathbb{A}_{\delta}=\{z\in \C\,|\, |z|\in [\delta,1]\}$ for some $\delta\in (0,1)$ with $\omega_j(\pl_j\Sigma)=\{|z|=1\}$, while for all other $j$, $V_j$ are open sets not intersecting $\pl \Sigma$ satisfying $\omega_j(V_j)=\D\subset \C$. The charts induce a complex structure $J$, as in the closed case, and we shall often write $(\Sigma,J)$ for the Riemann surface with boundary.
 The boundary circles $\pl_j\Sigma$ all inherit an orientation from the orientation of $\Sigma$: indeed, take 
the $1$-form $-\iota_{\pl_j\Sigma}^*(i_{\nu}\omega_\Sigma)$  where $i_\nu$ is the interior product with a non-vanishing  inward-pointing vector field $\nu$ to $\Sigma$ and $\iota_{\pl_j\Sigma}:\pl_j\Sigma\to \Sigma$ is the natural inclusion. 
In the chart given by the annulus $\mathbb{A}_\delta$, the orientation is then given by $\dd\theta$ on the unit circle parametrised by $(e^{{\rm i}\theta})_{\theta\in [0,2\pi]}$. We equip the Riemann surfaces with analytic boundary with a set 
 $\bs{\zeta}=(\zeta_1,\dots,\zeta_b)$ of analytic parametrisations of the boundary circles $\pl_1\Sigma,\dots,\pl_b\Sigma$
 \[\zeta_j:\T \to \pl_j\Sigma.\] 
We denote by $S=(\Sigma,J,\bs{\zeta})$ the Riemann surface $(\Sigma,J)$ with parametrised boundary  $\bzeta$.
We say that the boundary $\pl_j\Sigma$ is \textbf{outgoing} if the orientation $(\zeta_j)_*(\dd\theta)$ is the orientation of $\pl_j\Sigma$ induced by that of $\Sigma$ as described above,  otherwise the parametrised boundary $\pl_j\Sigma$ is called \textbf{incoming}. An analytic parametrisation $\zeta_j:\T\to \pl_j\Sigma$ also induces a holomorphic chart by holomorphically extending $\zeta_j$ in an annular neighborhood of $\T\subset \C$.\\

 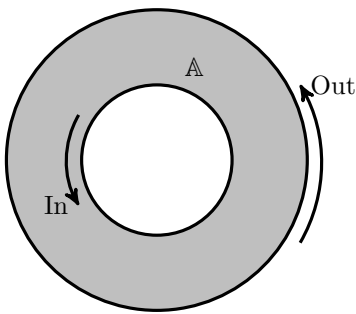
\begin{figure}[h] \label{pic2}
\begin{tikzpicture}[scale=1]
\draw[fill=lightgray,draw=black,very thick] (0,0) circle (2) ;
\draw[fill=white,draw=black,very thick] (0,0) circle (1) ;
\draw[->,>=stealth',draw=black,very thick] (150:1.2cm) arc[radius=1.2, start angle=150, end angle=210] node[left,black]{In};
\draw[->,>=stealth',draw=black,very thick] (-30:2.2cm) arc (-30:30:2.1cm) node[right,black]{Out};
\coordinate (Z1) at (1.5,0) ;
\coordinate (Z2) at (-1,0) ;
\coordinate (Z3) at (-0.8,1.3) ;
\coordinate (Z4) at (-0.7,-0.9) ;
 \draw (0.5,1) node[above]{$ \mathbb{A}$} ;
\end{tikzpicture}
\caption{In/out boundaries on an annulus $\mathbb{A}$ }
\end{figure}

\noindent \textbf{Marked points.} An ordered set of disjoint points ${\bf x}=(x_1,\dots,x_m)\in (\Sigma^\circ)^m$ in the interior $\Sigma^\circ$ of a Riemann surface (closed or with analytic boundary) $\Sigma$ is called a set of marked points. When considered as a subset of $\Sigma$, this set of points is denoted by $\{{\bf x}\}=\{x_1,\dots,x_m\}\subset \Sigma$.\\

\noindent \textbf{Metrics.} A Riemannian metric $g$ on a Riemann surface $(\Sigma,J)$, closed or with parametrised analytic boundary, 
is said to be compatible with its complex structure if  $g:=\omega_j^*(e^{f_j}|\dd z|^2)$ for each holomorphic chart 
$\omega_j:V_j\to \D$, for some smooth functions $f_j$. Any smooth metric $g$ on an oriented compact surface $\Sigma$
induces a complex structure $J$ compatible with $g$ (by constructing local holomorphic charts, or via the Hodge star operator).
If in addition, for each $j$ one has $g=\omega_j^*(|\dd z|^2/|z|^2)$ on an annular neighborhood $V_j$ of the boundary circle $\pl_j\Sigma$, where $\omega_j$ is the holomorphic extension of $\zeta_j^{-1}$ to $V_j$, we say that the metric is \textbf{admissible}.\\

We can then define the following notion of Riemann(ian) surfaces:
 \begin{definition}
Let $(\Sigma,J)$ be either a closed Riemann surface, or a compact Riemann surface with parametrised analytic boundary and parametrisation $\bzeta=(\zeta_1,\dots,\zeta_b)$ and let 
$g(\Sigma)$ be its genus. Let  ${\bf x}=(x_1,\dots,x_m)$ be a possibly empty collection of marked points. Let 
$b^+$ be the number of outgoing parametrised boundary circles and $b^-$ the number of incoming parametrised ones, 
where $b^++b^-=b$. We say that:
\begin{enumerate}
\item $S=(\Sigma,J,{\bf x},\bzeta)$ is a Riemann surface of type  $(g(\Sigma),m,b^+,b^-)$
\item $S_g=(\Sigma,g,{\bf x},\bzeta)$ is a Riemannian surface of type  $(g(\Sigma),m,b^+,b^-)$ if $g$ is compatible with $J$
\item $S_g=(\Sigma,g,{\bf x},\bzeta)$ is an admissible Riemannian surface of type  $(g(\Sigma),m,b^+,b^-)$ if $g$ is an admissible 
Riemannian metric. 
\end{enumerate}
If $(m=0,b\not=0)$ we simply write $S=(\Sigma,J,\bzeta)$, if $(m\not=0,b=0)$ we write $S=(\Sigma,J,{\bf x})$, while if $(m,b)=(0,0)$ we write $S=(\Sigma,J)$.
\end{definition}

\noindent \textbf{Gluing surfaces.}  
On a Riemann surface $S=(\Sigma,J,{\bf x},\bzeta)$ (not necessarily connected)  
with parametrised analytic boundary and marked points, if $\zeta_j :\T\to \pl_j\Sigma$ is an outgoing parametrised boundary component 
and $\zeta_k:\T\to \pl_k\Sigma$ an incoming one, we can define a new Riemann surface $\Sigma^{\#}$ by gluing/identifying $\pl_j\Sigma$ with $\pl_k\Sigma$ as follows: we identify $\pl_j\Sigma$ with $\pl_k\Sigma$ by setting $\zeta_j(e^{{\rm i}\theta})=\zeta_k(e^{{\rm i}\theta})$. This produces a complex structure $J^\#$ on $\Sigma^{\#}$ (see \cite[Section 3.1.3.]{GKRV21_Segal} for details), which depends on $\zeta_j,\zeta_k$ and 
the parametrisation $\bzeta^\#$ obtained by removing $\zeta_j,\zeta_k$ from $\bzeta$ induces an analytic parametrisation of $\pl \Sigma\#$. The marked points ${\bf x}$ can still be viewed as marked point in $\Sigma^\#$.   
If $g$ is an admissible metric on $\Sigma$, then it descends as a smooth admissible metric on $\Sigma^\#$.

\subsection{Laplacians and their determinant, Green's functions}

Let $(\Sigma,J)$ be a closed Riemann surface and $g$ be a compatible Riemannian metric with associated volume form ${\rm dv}_g$. 
Integration of functions $f$ will be denoted $\int_{\Sigma} f(x)  {\rm dv}_g(x)$ and $L^2(\Sigma,{\rm dv}_g)$ is the space of square integrable functions with respect to ${\rm v}_g$. In this context, we consider the standard non-negative Laplacian $\Delta_{g}=\dd^*\dd$ where $\dd^*$ is the $L^2$ adjoint of the exterior derivative $\dd$. It is self-adjoint and has discrete spectrum $(\la_n)_{n\in \N}\subset [0,\infty)$ with $\la_0=0$ and $\la_n\to \infty$ as $n\to \infty$. In the case of a Riemann surface with analytic parametrised boundary $(\Sigma,J,\bzeta)$ 
and $g$   a compatible Riemannian metric, we denote $\Delta_{g,{\rm D}}=\dd^*\dd$ where $\dd^*$ is as above but we consider the Dirichlet self-adjoint extension, i.e. the domain of $\Delta_{g,{\rm D}}$ consists of $H^2(\Sigma)\cap H_0^1(\Sigma)$ functions. Its spectrum is a discrete set $(\la_{n,{\rm D}})_{n\in \N^*}\subset (0,\infty)$.\\ 

\noindent \textbf{Determinant of Laplacian.}
The determinant of the Laplacian $\det(\Delta_{S_g})$ on a compact Riemannian manifold $S_g=(\Sigma,g)$ was introduced by Ray-Singer \cite{Ray-Singer}, we refer to that article for more details. First, in the closed case, 
we define the determinant of $\Delta_{S_g}$ by 
\[ \det(\Delta_{S_g})=\exp(-\pl_s\zeta_{S_g}(s)|_{s=0})\]
where $\zeta_{S_g}(s):=\sum_{n=1}^\infty \la_n^{-s}$ is the spectral zeta function of $\Delta_{S_g}$, which admits a meromorphic continuation from ${\rm Re}(s)\gg 1$ to $s\in \C$ and is holomorphic at $s=0$. Here $(\la_n)_{n\in \N^*}$ are the non-zero eigenvalues of $\Delta_{S_g}$ counted with multiplicities. 
Second, if $S_g=(\Sigma,g)$ is compact with boundary, the determinant is defined in a similar way as $\det (\Delta_{S_g,{\rm D}}): = e^{- \zeta'_{g,{\rm D}}(0)}$ where $\zeta_{S_g,{\rm D}}$   is the spectral zeta function of $\Delta_{S_g,{\rm D}}$ with Dirichlet boundary conditions defined for ${\rm Re}(s)\gg 1$ by $\zeta_{S_g,{\rm D}}(s):=\sum_{n=1}^\infty \la_{n,{\rm D}}^{-s}$. The function $\zeta_{S_g,{\rm D}}(s)$ admits a meromorphic extension to $s\in \C$ that is holomorphic at $s=0$. We shall sometimes use the notation $\Delta_g$ and $\Delta_{g,{\rm D}}$ when the underlying surface is fixed and there is no possible confusion.\\  

\noindent \textbf{Green's functions.} The Green function $G_{S_g}$ of the Laplacian $\Delta_g$ on a closed Riemannian 
surface $S_g=(\Sigma,g)$ is the integral kernel of the resolvent operator $R_{S_g}:L^2(\Sigma)\to L^2(\Sigma)$ satisfying $\Delta_{g}R_{S_g}=2\pi ({\rm Id}-\Pi_0)$, $R_{S_g}^*=R_{S_g}$ and $R_{S_g}1=0$, where $\Pi_0$ is the orthogonal projection in $L^2(\Sigma,{\rm dv}_g)$ on $\ker \Delta_{g}$ (the constants). By integral kernel, we mean that for each $f\in L^2(\Sigma,{\rm dv}_g)$
\[ R_{S_g}f(x)=\int_{\Sigma} G_{S_g}(x,x')f(x'){\rm dv}_g(x').\] 
The Laplacian $\Delta_{g}$ has an orthonormal basis of real-valued eigenfunctions $(e_n)_{n\in \N}$ in $L^2(\Sigma,{\rm dv}_g)$ with associated eigenvalues $\la_j\geq 0$; we set $\la_0=0$ and $e_0=({\rm v}_g(\Sigma))^{-1/2}$ where ${\rm v}_g(\Sigma)=\int_\Sigma \dd {\rm v}_g$ is the volume.  The Green function then admits the following Mercer's representation in $L^2(\Sigma\times\Sigma, {\rm dv}_g \otimes {\rm dv}_g)$
\begin{equation}
G_{S_g}(x,x')=2\pi \sum_{n\geq 1}\frac{1}{\lambda_n}e_n(x)e_n(x').
\end{equation}
When the surface $\Sigma$ is fixed and we want to describe the dependence of the Green function in terms of $g$, we shall 
often write $G_g$ instead of $G_{S_g}$ to simplify notations.

Similarly, on a surface with smooth boundary $\Sigma$, we will consider the Green function with Dirichlet boundary conditions $G_{S_g,{\rm D}}$ associated to the Laplacian $\Delta_{g}$. In this case, the associated resolvent operator
\[ R_{S_g,{\rm D}}f(x)=\int_{\Sigma} G_{S_g,{\rm D}}(x,x')f(x'){\rm dv}_g(x')\] 
solves $\Delta_{g}R_{S_g,{\rm D}}=2\pi {\rm Id}$. We note that, contrary to the closed case, the Dirichlet Green's function is conformally invariant, 
i.e. $G_{S_g',{\rm D}}(x,x')=G_{S_g,{\rm D}}(x,x')$ if $g'=e^{\omega}g$ for some $\omega\in C^\infty(\Sigma)$. In particular $G_{S_g,{\rm D}}$ and $R_{S_g,{\rm D}}$ only depend on the complex structure $(\Sigma,J)$ but not on $g$, we thus simplify the notation to $G_{S,{\rm D}}$ and $R_{S,{\rm D}}$ in that case.\\

\noindent \textbf{Poisson operator.} Let $S=(\Sigma,J,\bs{\zeta})$ be a compact Riemann surface with  parametrised analytic boundary 
$\partial\Sigma=\bigcup_{j=1}^b\pl_j \Sigma$ having $b$ connected components.
 By  abuse of notation, we denote by $(\cdot,\cdot)_2$ the pairing between real-valued functions in $(H^{s}(\T))^{b}$ with real-distributions in $(H^{-s}(\T))^{b}$ for $s>0$, which is the extension of the following pairing  on $C^\infty(\T)^b$ functions
\[(u,f)_2= \sum_{j=1}^b \frac{1}{2\pi}\int_{0}^{2\pi}u_j\bar{f}_j\d \theta.\]
For  $\bs{\tilde \varphi}=(\tilde\varphi_1,\dots,\tilde\varphi_{b})\in (H^{s}(\T))^{b}$ with $s\in\R$, we denote by 
$P_{S}\tilde{\bs{\varphi}}$ the harmonic extension  of $\tilde{\bs{\varphi}}$ on $\Sigma$ defined by 
\begin{equation}\label{def:harmonic_extension}
\Delta_g P_{S}\tilde{\bs{\varphi}}=0 \textrm{ in } \Sigma^\circ, \quad  
(P_{S}\tilde{\bs{\varphi}})|_{\pl_j\Sigma}=\tilde\varphi_j\circ \zeta_j^{-1}  \textrm{ for }j=1,\dots,b.
\end{equation}
where $g$ is any Riemannian metric compatible with $J$. The boundary value is understood in the following weak sense: for all $u\in C^\infty(\T)$ real-valued and $\omega_j$ the holomorphic extension of $\zeta_j^{-1}$ to a neighborhood of $\pl_j\Sigma$
\[\lim_{r\to 1^-}\int_{0}^{2\pi}P_{S}\tilde{\bs{\varphi}}(\omega_j^{-1}(re^{{\rm i}\theta}))u(e^{{\rm i}\theta})d\theta =2\pi (\tilde\varphi_j,u)_{2}.\]
We call $P_{S}:H^s(\T)\to C^\infty(\Sigma^\circ)$ the Poisson operator of $S$.

We also consider the Poisson operator on a surface $S=(\Sigma,J,\bs{\zeta})$ with $i_0$ parametrised analytic embedded non-overlapping curves in $\Sigma^\circ$: $\xi_i:\T\to \mc{C}_i$ for $i=1,\dots,i_0$, gathered in $\bs{\xi}=(\xi_1,\dots,\xi_{i_0})$. The surface  $\Sigma\setminus \cup_i\mc{C}_i$ can be compactified into a smooth surface $\Sigma_{\mc{C}}$ with analytic boundary (possibly disconnected if some $\mc{C}_i$ are separating). The new boundaries of $\Sigma_{\mc{C}}$ are pairs of copies of $\mc{C}_i$, parametrised by $\xi_i$. For $\bs{\tilde \varphi}=(\tilde\varphi_1,\dots,\tilde\varphi_{i_0})\in (H^{s}(\T))^{i_0}$ with $s\in\R$, we let $P_{S,\bs{\xi}}\tilde{\bs{\varphi}}$ be  the harmonic extension of $\bs{\tilde \varphi}$ on $S=(\Sigma,J,\bzeta)$ with boundary value $0$ on $\partial\Sigma$, i.e.
\[\Delta_g P_{S,\bs{\xi}}\tilde{\bs{\varphi}}=0 \textrm{ in } \Sigma\setminus \bigcup_i\mc{C}_i, \quad  P_{S,\bs{\xi}}\tilde{\bs{\varphi}}|_{\pl \Sigma}=0,\quad    (P_{S,\bs{\xi}}\tilde{\bs{\varphi}})|_{\mc{C}_i}=\tilde{\varphi}_i\circ \xi_i^{-1} \textrm{ for } i=1,\dots,i_0,\]
where the restriction on $\mc{C}_i$ is understood in the weak sense as above, and $g$ is any compatible smooth metric with $J$ on $\Sigma$.\\

\noindent \textbf{Dirichlet-to-Neumann maps.}
Let $S=(\Sigma,J,\bs{\zeta})$ be a surface  with parametrised boundary, with $b$ boundary  components, and $g$ a metric compatible with $J$ on $\Sigma$. The DN map $\mathbf{D}_{S}: C^\infty(\T)^{b}\to C^\infty(\T)^{b}$ is defined by
\[\mathbf{D}_{S}\tilde{\bs{\varphi}}=(-\partial_{\nu } P_{S}\tilde{\bs{\varphi}}|_{\pl_j\Sigma}\circ\zeta_j)_{j=1,\dots,b}\]
where $\nu$ is the inward unit normal vector fields to $\pl_j\Sigma$ with respect to any admissible metric $g$ and $\tilde{\bs{\varphi}}\in   C^\infty(\T)^{b}$.
Since $P_S\tilde{\bs{\varphi}}$ does not depend on the conformal representative metric $g$ representing the complex structure of $\Sigma$, $\mathbf{D}_{S}$ only depends on the complex structure $J$ of $\Sigma$, on $\bzeta$ but not on $g$.
Note that, by  Green's  formula,
\begin{equation}\label{Greenformula}
\int_{\Sigma} |\dd P_{S}\tilde{\bs{\varphi}}|_g^2{\rm dv}_g = -\sum_{j=1}^b\int_{\pl_j\Sigma}
\pl_{\nu}(P_{S}\tilde{\bs{\varphi}})\tilde{\varphi}_j\circ \zeta_j^{-1} \dd\ell_g=
2 \pi ( \tilde{\bs{\varphi}},\mathbf{D}_{S}\tilde{\bs{\varphi}})_2
\end{equation}
where $\nu$ is the inward unit normal vector with respect to $g$ and $\dd\ell_g$ is the measure induced by the metric $g$ on $\pl\Sigma$.
By \eqref{Greenformula}, $\mathbf{D}_{S}$ is a non-negative symmetric operator with kernel $\ker {\bf D}_{S}=\R {\bf 1}$ where ${\bf 1}= (1, \dots, 1)$.
We also observe that $\mathbf{D}_{S}$ is conformally invariant in the following sense:
\begin{lemma}\label{DNmapInv}
Let $S=( \Sigma,J,\bzeta)$ and $S'=(\Sigma',J',\bzeta')$ be two Riemann surfaces with analytic parametrised boundary and $\psi: (\Sigma,J)\to (\Sigma',J')$ be a biholomorphism such that $\psi\circ \zeta_k=\zeta_k'$ for $k=1,\dots,b$; here $\bzeta=(\zeta_1,\dots,\zeta_b)$ and $\bzeta'=(\zeta_1',\dots,\zeta_b')$. Then the following holds true  
\[  \mathbf{D}_{S}=\mathbf{D}_{S'}.\]
\end{lemma} 
\begin{proof} We first observe that $(P_{S'} \tilde{\bs{\varphi}})\circ \psi=P_S\tilde{\bs{\varphi}}$, and then by a change of variables
\[\int_{\Sigma} |\dd P_S\tilde{\bs{\varphi}}|_{g}^2{\rm dv}_{g}=
\int_{\Sigma} |\dd P_{S'}\tilde{\bs{\varphi}}|_{g'}^2{\rm dv}_{g'}.\qedhere\]
\end{proof} 

For $S=(\Sigma,J,\bs{\zeta})$ a Riemann surface with parametrised analytic boundary and $\bs{\xi}=(\xi_1,\dots,\xi_{i_0})$ some analytic parametrised interior embedded non-overlapping curves $\xi_i:\T\to \mc{C}_i$ as above. 
The DN map  associated to $\bs{\xi}$ is the operator $\mathbf{D}_{S,\bs{\xi}}:C^\infty(\T)^{i_0}\to C^\infty(\T)^{i_0}$, given for $\tilde{\bs{\varphi}}\in C^\infty(\T)^{i_0}$ by $((\mathbf{D}_{S,\bs{\xi}}\tilde{\bs{\varphi}})_i)_{i=1,\dots,i_0}$ with the operator    given by     
\begin{equation}\label{defDSigmaC}
\mathbf{D}_{S,\bs{\xi}}\tilde{\bs{\varphi}}=-((\partial_{\nu_-} P_{S,\bs{\xi}}\tilde{\bs{\varphi}})|_{\mc{C}_i}\circ \xi_i+(\partial_{\nu_+} P_{S,\bs{\xi}}\tilde{\bs{\varphi}})|_{\mc{C}_i}\circ \xi_i)_{i=1,\dots,i_0}.
\end{equation}
Here $\partial_{\nu_\pm}$ denote the two inward unit normal derivatives with respect to any metric $g$ compatible with $J$ and such that $\xi_i^*g=\dd \theta^2$ for each $i=1,\dots,i_0$.

The operator $\mathbf{D}_{S,\bs{\xi}}$ is a symmetric unbounded operator on $L^2(\T)^{i_0}$, it is positive definite and invertible as a map $H^s(\T)^{i_0}\to H^{s-1}(\T)^{i_0}$ for all $s\in \R$. Its quadratic form is equal (by Green's formula), for smooth $\tilde{\bs{\varphi}}$, to
\[ (\mathbf{D}_{S,\bs{\xi}}\tilde{\bs{\varphi}},\tilde{\bs{\varphi}})_2=
\frac{1}{2\pi}\int_{\Sigma} |\dd P_{S,\bs{\xi}}\tilde{\bs{\varphi}}|_{g}^2{\rm dv}_{g}.
\] 
By \cite[Theorem 2.1]{carron}, $\mathbf{D}_{S,\bs{\xi}}$ is invertible and for $i=1,\cdots, i_0$ one has 
\begin{equation}\label{DNmapandGreen}
\begin{split}
(\mathbf{D}_{S,\bs{\xi}}^{-1}\tilde{\bs{\varphi}})_{i} (e^{{\rm i}\theta})=& \sum_{j=1}^{i_0}
\frac{1}{2\pi}\int_{\T}G_{S,{\rm D}}(\xi_i(e^{{\rm i}\theta}),\xi_j(e^{{\rm i}\theta'})) \tilde{\varphi}_j(e^{{\rm i}\theta'})  \dd\theta'
\end{split}
\end{equation}  
We denote by $\mathbf{D}:C^{-\infty}(\T)\to C^{-\infty}(\T)$ the Fourier multiplier by $|n|$, i.e.
\[ \forall \tilde{\varphi}=\sum_{n\in \Z}e^{in\theta}\varphi_n\in C^{-\infty}(\T), \,\, {\bf D}\tilde{\varphi}=\sum_{n\in \Z}|n|e^{in\theta}\varphi_n\]
and we extend this action to $C^{-\infty}(\T)^b$ for any $b\in \N$ as a diagonal action.
In particular it is an unbounded self-adjoint operator $\mathbf{D}$ on $L^2(\T)^{b}$ associated to  the quadratic form
 \[
\forall \tilde{\bs{\varphi}} \in C^\infty(\T;\R)^{b},\quad (\mathbf{D}\tilde{\bs{\varphi}},\tilde{\bs{\varphi}})_2:= 2\sum_{j=1}^{b}\sum_{n>0} n|\varphi_{jn}|^2 \textrm{ with }\tilde{\varphi}_j(e^{{\rm i}\theta})=c+\sum_{n\not=0}\varphi_{jn}e^{in\theta}.
\]
By \cite[Lemma 4.1]{GKRV21_Segal}, we have that  $(\mathbf{D}_{S}-\mathbf{D}):C^{-\infty}(\T)^b\to C^\infty(\T)^b$ 
and $(\mathbf{D}_{S,\bs{\xi}}-2\mathbf{D}):C^{-\infty}(\T)^b\to C^\infty(\T)^b$ are smoothing operators, and 
\begin{equation}\label{continuityDN-D}
\begin{split}
& 1)\,\, \tilde{\bs{\varphi}} \mapsto ((\mathbf{D}_S-\mathbf{D})\tilde{\bs{\varphi}},\tilde{\bs{\varphi}})_2 \textrm{ is continuous on }H^s(\T)^b\textrm{ for all }s\in \R,\\
& 2)\textrm{ the Fredholm determinant  }{\det}_{\rm Fr}(\mathbf{D}_{S,\bs{\xi}}(2\Pi_0+2{\bf D})^{-1}) \textrm{ is well-defined}.
\end{split}
\end{equation}
where $\Pi_0(\tilde\varphi_1,\dots,\tilde\varphi_{i_0}):=  (( \tilde\varphi_1,1)_2,\dots,(\tilde\varphi_{i_0},1 )_2)$ and ${\bf D}$ acts on 
$C^{-\infty}(\T)^{i_0}\to C^{-\infty}(\T)^{i_0}$ as above (diagonally).

\subsection{Gaussian Free Fields and regularisations}\label{SectionGFF} In the case of a connected closed Riemann surface $(\Sigma,J)$ with a compatible metric $g$, the Gaussian Free Field (GFF in short) of $S_g=(\Sigma,g)$ is defined as follows. Let $(a_n)_n$ be a sequence of i.i.d. real-valued Gaussians   $\mc{N}(0,1)$ with mean $0$ and variance $1$, defined on some probability space  $(\Omega,\mc{F},\mathbb{P})$ and $(e_n)_{n\in \N}$ be an orthonormal basis of eigenfunctions of $\Delta_g$ in $L^2(\Sigma,{\rm dv}_g)$ with eigenvalues $\la_n\geq 0$, and we take the convention that $\la_0=0$ and 
 $e_0(x)=1/\sqrt{{\rm v}_g(\Sigma)}$. 
The Gaussian Free Field with vanishing mean in the metric $g$ is defined by the random series
\begin{equation}\label{GFFong}
X_{S_g}:= \sqrt{2\pi}\sum_{n\geq 1}a_n\frac{e_n}{\sqrt{\la_n}} 
\end{equation}
 where the sum converges almost surely, for any $s<0$ in the Sobolev space  $H^{s}(\Sigma)$, defined by
\[ H^{s}(\Sigma):=\{f=\sum_{n\geq 0}f_ne_n\,|\, \|f\|_{s}^2:=|f_0|^2+\sum_{n\geq 1}\lambda_n^{s}|f_n|^2<+\infty\}.\] 
This Hilbert space is independent of $g$, only its norm depends on a choice of $g$.
The covariance is then the Green function when viewed as a distribution, which we write with a slight abuse of notation
\[\mathbb{E}[X_{S_g}(x)X_{S_g}(x')]= \,G_{S_g}(x,x').\]
When the surface $(\Sigma,J)$ is fixed and no confusion is possible, we simplify the notation and write $X_{g}$ instead of $X_{S_g}$.
We denote the Liouville field by $\phi_{S_g}:=c+X_{S_g}$ where $c\in\R$ is a constant that stands for the constant mode of the field, 
and we also use the shortcut notation $\phi_g$.  

In the case of a surface with parametrised analytic boundary $S=(\Sigma,J,\bzeta)$, 
the Dirichlet Gaussian free field (with covariance $G_{S,{\rm D}}$) will be denoted $X_{S,{\rm D}}$. It is defined similarly to the sum \eqref{GFFong} with the $(e_n)_{n\in \N^*}$  and $(\lambda_n)_{n\in \N^*}$ replaced by the normalised eigenfunctions $(e_{n,{\rm D}})_{j\in \N^*}$ and ordered eigenvalues $(\lambda_{n,{\rm D}})_{n\in \N^*}$ of the Laplacian with Dirichlet boundary conditions, the sum being convergent  almost surely  in the Sobolev space  $H^{s}(\Sigma)$ (for all $s\in (-1/2,0)$) defined by
\[ H^{s}_{\rm D}(\Sigma):=\{f=\sum_{n\geq 1}f_ne_{n,{\rm D}}\,|\, \|f\|_{s}^2:=\sum_{n\geq 1}\lambda_{n,{\rm D}}^{s}|f_n|^2<+\infty\}.\] 
In this context, we  will always consider the harmonic extension $P_S \tilde{\bs{\varphi}}$ of a boundary field $ \tilde{\bs{\varphi}}\in H^{s}(\T)^b$ 
defined by \eqref{def:harmonic_extension}.
In the case of the surface $S$ with boundary, the Liouville field will be denoted by  
\begin{equation}\label{defliouvillefield}
\phi_S:= X_{S,{\rm D}}+P_S \tilde{\bs{\varphi}}. 
\end{equation}
Hence $\phi_S$ depends on boundary data $\tilde{\bs{\varphi}}$ and on $S=(\Sigma,J,\bzeta)$ in this case but is not depending on a choice of compatible metric $g$.

As Gaussian Free Fields are distributions and therefore require regularisations with respect to a Riemannian metric $g$. We then introduce a regularisation procedure, which we call \textbf{$g$-regularisation}. Let $(\Sigma,J)$ be a closed Riemann surface  equipped with a compatible Riemannian metric $g$ and denote by $d_g$ the associated distance. 
For certain random distributions $h$ on $\Sigma$ and for $\eps>0$ small, we define a regularisation $h_{\eps}$ of $h$ by averaging on geodesic circles of radius $\eps>0$: let $x\in \Sigma$ and let $\mc{C}_g(x,\eps)$ be the geodesic circle of center $x$ and radius $\eps>0$, and let $(f^n_{x,\eps})_{n\in \N} \in C^\infty(\Sigma)$ be a sequence with $||f^n_{x,\eps}||_{L^1}=1$ 
which is given by $f_{x,\eps}^n=\theta^n(d_g(x,\cdot)/\eps)$ where $\theta^n(r)\in C_c^\infty((0,2))$ is non-negative, 
supported near $r=1$ and such that $f^n_{x,\eps}{\rm dv}_g$ 
converges weakly to the uniform probability measure 
$\mu_{x,\eps}$
on $\mc{C}_g(x,\eps)$ as $n\to \infty$. 
If the pairing $\langle h, f_{x,\eps}^n\rangle$ converges almost surely to a random variable $h_\eps(x)$ that has a modification which is continuous in the parameters $(x,\eps)$, we will say that $h$ admits a $g$-regularisation $(h_\eps)_\eps$. This is the case for the GFF $X_{S_g}$ on a closed surface or the Dirichlet GFF $X_{S,{\rm D}}$ on a surface with boundary, see \cite[Lemma 3.2]{Guillarmou2019}. We will denote by $X_{g,\eps}$, $X_{g,{\rm D},\eps}$ their respective $g$-regularisations
and $\phi_{g,\eps}$ the $g$-regularisation of the Liouville field $\phi_{S_g}$ (closed case) or $\phi_{S}$ (case with boundary).

\subsection{Construction of Liouville Conformal Field Theory} 
 In this section we recall briefly the construction of Liouville path integral and correlation functions on closed surfaces from the Gaussian free field. 
 
\noindent \textbf{Gaussian multiplicative chaos.} For $\gamma\in\R$ and $h$ a random distribution admitting a  $g$-regularisation $(h_\eps)_\eps$, we define the measure 
\begin{equation}\label{GMCg}
M^{g,\eps}_{\gamma}(h,\dd x):= \eps^{\frac{\gamma^2}{2}}e^{\gamma h_{ \eps}(x)}{\rm dv}_g(x).
\end{equation}
Of particular interest for us is the case when $h=X_g$ or $h=X_{S,{\rm D}}$ (and consequently $h=\phi_{g}$ and $h=\phi_{S}$ too). In that case, for $\gamma\in (0,2)$,  the random measures above converge  as $\eps\to 0$ in probability and weakly in the space of Radon measures towards    non-trivial  random measures respectively denoted by $M^g_\gamma(X_g,\dd x)$, $M^g_\gamma(X_{S,{\rm D}},\dd x)$, $M^g_\gamma(\phi_g,\dd x)$ and $M^g_\gamma(\phi_S,\dd x)$; this is a standard fact and the reader may consult \cite{Kahane85,rhodes2014_gmcReview,Guillarmou2019} for further details. Clearly, when $\Sigma$ has no boundary, we have the relation $M^g_\gamma(\phi_g,\dd x)=e^{\gamma c}M^g_\gamma(X_{g},\dd x)$ and, in the case when $\Sigma$ has a  boundary, then $M^g_\gamma(\phi_S,\dd x)= e^{\gamma P_S  \tilde{\bs{\varphi}} (x)} M^g_\gamma(X_{S,{\rm D}},\dd x)$.

Also, from  \cite[Lemma 3.2]{Guillarmou2019}  we recall that there exist $W_g,W_{g,{\rm D}}\in C^\infty(\Sigma)$ such that 
\begin{equation}\label{varYg}
\lim_{\eps \to 0}\E[X^2_{g,\eps}(x)]+\log\eps=W_g(x),\quad \lim_{\eps \to 0}\E[X^2_{g,{\rm D},\eps}(x)]+\log\eps=W_{g,{\rm D}}(x)
\end{equation}
uniformly over the compact subsets of $\Sigma$. Moreover, in the case of the Dirichlet GFF,  one has the following relation if $g'=e^{\omega}g$ for some $\omega \in C^\infty(\Sigma)$ (observe that the Dirichlet GFF is the same for $g$ and $g'$): 
\begin{equation}\label{varYgconformal}
W_{g',{\rm D}}(x)= W_{g,{\rm D}}(x)+\frac{\omega(x)}{2}
\end{equation}
which leads directly to the scaling relation 
\begin{equation}\label{scalingmeasure}
M^{g'}_\gamma(X_{S,{\rm D}},\dd x)= e^{\frac{\gamma}{2}Q \omega(x)} M^g_\gamma(X_{S,{\rm D}},\dd x).
\end{equation}

\noindent\textbf{Liouville CFT on closed Riemann surfaces.}
  Let  $S_g=(\Sigma,g)$ be a closed Riemannian surface of genus $g(\Sigma)$, $g$ being compatible with a complex structure $J$, and let  $\gamma\in(0,2)$, $\mu>0$ and $Q=\frac{\gamma}{2}+\frac{2}{\gamma}$. Let $K_g$ be the scalar curvature of $g$. Recall that the Liouville field is $\phi_g=c+X_g$ with $c\in\R$. For $F:  H^{s}(\Sigma)\to\R$ (with $s<0$) a bounded continuous functional, we set 
\begin{align}\label{def:pathintegralF}
\cjg F(\phi)\cjd_{S_g}:=& \big(\frac{{\rm v}_{g}(\Sigma)}{\det(\Delta_{g})}\big)^\hf  \int_\R  \E\Big[ F( \phi_g) \exp\Big( -\frac{Q}{4\pi}\int_{\Sigma}K_{g}\phi_g\,{\rm dv}_{g} - \mu   M^g_\gamma(\phi_g,\Sigma)  \Big) \Big]\,\dd c . 
\end{align}
 By \cite[Proposition 4.1]{Guillarmou2019}, this quantity  defines a measure and moreover the partition function defined as  the total mass of this measure, i.e. $\cjg 1 \cjd_ {S_g}$, is finite  if and only if the genus $g(\Sigma)\geq 2$.\\

\noindent \textbf{Vertex operators.} On a Riemannian surface $(\Sigma,g)$ with or without boundary, we introduce the regularised vertex operators, for fixed $\alpha\in \R$ (called {\it weight}) and $x\in \Sigma^\circ$, (when $\pl\Sigma\not=\emptyset$, we use the field \eqref{defliouvillefield})
\begin{equation*}
V_{\alpha,g,\eps}(x)=\eps^{\alpha^2/2}  e^{\alpha  \phi_{g,\eps}(x) } .
\end{equation*}

Next, if the surface $\Sigma$ is closed, the correlation functions on $S_g=(\Sigma,g)$ are defined by the limit
\begin{equation}\label{defcorrelg}
\cjg \prod_i V_{\alpha_i,g}(x_i) \cjd_ {S_g}:=\lim_{\eps \to 0} \: \cjg \prod_i V_{\alpha_i,g,\eps}(x_i) \cjd_ {S_g}
\end{equation}
where we have fixed $m$ distinct points $x_1,\dots,x_m$  on $\Sigma$ with respective associated weights $\alpha_1,\dots,\alpha_m\in\R$. Non triviality of correlation functions are then summarized in the following proposition (see \cite[Proposition 4.4]{Guillarmou2019}):

 \begin{proposition}\label{limitcorel} Let $ x_1,\dots,x_m$ be distinct points on a closed surface $\Sigma$ and $ (\alpha_1,\dots,\alpha_m)\in\R^m$.
The limit \eqref{defcorrelg} exists and is non zero if and only if the weights $(\alpha_1,\dots,\alpha_m)$ obey the following Seiberg bounds
 \begin{align}\label{seiberg1}
 & \sum_{i}\alpha_i + 2 Q (g(\Sigma)-1)>0,\\ 
 &\forall i,\quad \alpha_i<Q\label{seiberg2}.
 \end{align}
 \end{proposition}

Notice that in the case of a surface with boundary and if $g'=e^{\omega}g$, then for $F:H^s(\Sigma)\to \R$ bounded continuous such that $\cjg F\cjd_{S_g}$ exists, then the relation \eqref{varYgconformal} gives
\begin{equation}\label{scalingvertex}
\cjg V_{\alpha,g',\eps}(x)F\cjd_{S_g}=(1+o(1)) e^{\frac{\alpha^2}{4} \omega(x)}\cjg  V_{\alpha,g,\eps}(x)F\cjd_{S_g} 
\end{equation}
when $\eps$ goes to $0$. The same holds for products of vertex operators:
\[\cjg \prod_{j=1}^mV_{\alpha_j,g',\eps}(x_j) F\cjd_{S_g}=(1+o(1))e^{\sum_{j=1}^m \frac{\alpha_j^2}{4} \omega(x_j)} \cjg  \prod_{j=1}^mV_{\alpha_j,g,\eps}(x_j)F\cjd_{S_g}.\]

\subsection{Liouville Amplitudes}
Here we recall the notion of amplitudes. They will serve to decompose the path integral on closed Riemann surfaces 
into composition of operators. Roughly speaking, non-intersecting embedded closed simple curves on a compact Riemann surface decompose the surface into several pieces with boundaries (i.e. the curves). For each piece, one can place boundary conditions on each connected  component of the boundary and define the path integral on this piece conditionally on the boundary values: this defines an amplitude.  Gluing these amplitudes back together allows one to recover the path integral on the original Riemann surface. This is the content of Segal's axioms as proved in \cite{GKRV21_Segal}.\\

\noindent \textbf{Hilbert space of Liouville CFT.}  The Hilbert space encodes the boundary conditions for the Segal amplitudes. 
We denote by $\T:=\{z\in \C\,|\, |z|=1\}$ the unit circle. 
A generic (real-valued) field $\tilde\varphi\in H^s(\T)$ (for $s<0$) will be decomposed into its constant mode $c$ and orthogonal part
\begin{equation}\label{param}
\tilde\varphi=c+\varphi,\quad \varphi(\theta)=\sum_{n\not=0}\varphi_ne^{in\theta}
\end{equation}
with $(\varphi_n)_{n\not=0}$  its other Fourier coefficients, which will be themselves parametrised by $\varphi_n=\frac{x_n+iy_n}{2\sqrt{n}}$  and $\varphi_{-n}=\frac{x_n-iy_n}{2\sqrt{n}}$  for $n>0$, with real $x_n,y_n$'s.  We equip $H^s(\T) $ for $s<0$ with the cylinder sigma algebra and the measure 
\begin{align}\label{Pdefin}
 \mu_0:=\dd c\otimes  \P_{\T} \quad \text{ with }\quad  
 \P_\T:=\bigotimes_{n\geq 1}\frac{1}{2\pi}e^{-\frac{1}{2}(x_n^2+y_n^2)}\dd x_n\dd y_n.
\end{align}
The Liouville Hilbert space is  $\mc{H}:=L^2(H^s(\T),\mu_0)$ with Hermitian product denoted by $\langle\cdot,\cdot\rangle_{\mc{H}}$.
Since for smooth real-valued $\tilde{\varphi}\in C^\infty(\T)$, we have 
\begin{equation}\label{defmathbfD}
(\mathbf{D}\tilde{\varphi},\tilde{\varphi})_2:= 2\sum_{n>0} n|\varphi_{n}|^2=\frac{1}{2}\sum_{n>0}((x_{n})^2+(y_{n})^2).
\end{equation}
the measure $\P_{\T}$ is the Gaussian measure associated to the covariance operator ${\bf D}$.
For $Q\in \R$, the Hilbert space $\mc{H}$ carries two commuting, unitary representations of the Heisenberg Lie algebra with generators $(\bA_n,\tilde{\bA}_n)_{n\in\Z}$ defined as follows: for $n\in\Z_{>0}$, $\pl_n:=\sqrt{n}(\pl_{x_n}-i\pl_{y_n})$ and $\pl_{-n}=\pl_n^*$ if $n>0$, we set
\begin{equation}\label{Heisenberg}
\begin{split}
\bA_n:=\frac{i}{2}\del_n,\quad\quad  \bA_0:=\frac{i}{2}(\del_c+Q), \quad  \bA_{-n}:=\frac{i}{2}(\del_{-n}-2n\varphi_n)\\
\tilde{\bA}_n:= \frac{i}{2}\del_{-n}, \quad  \tilde{\bA}_0=\frac{i}{2}(\del_c+Q),\quad   \tilde{\bA}_{-n}:=\frac{i}{2}(\del_n-2n\varphi_{-n}).
\end{split}
\end{equation}
For $n$ positive (resp. negative), $\bA_n$ is called an \emph{annihilation} (resp. \emph{creation}) operator. These operators will be important for what follows, when we introduce the Virasoro algebra.

In \cite{GKRV21_Segal} the definition of amplitudes for admissible Riemannian surfaces with parametrised boundary and marked points carrying weights satisfying the Seiberg conditions was introduced. 
In this subsection, we extend this definition to general Riemannian surfaces $S_g=(\Sigma,g, {\bf x},\bs{\zeta})$ with weights $\bs{\alpha}=(\alpha_1,\dots,\alpha_m)$, i.e. we do not assume that the metric $g$ is admissible with respect to $(J,\bs{\zeta})$. 
We define the amplitude $\mc{A}_{S_g,\bs{\alpha}}$ as follows. 

\begin{definition}{ \bf{(Amplitudes)}}\label{def:amp}
Let $S=(\Sigma,J, {\bf x},\bs{\zeta})$ be a Riemann surface of type $(g(\Sigma),m,b^+,b^-)$ with $b$ analytic parametrised boundary circles and $m$ marked points, let $\bs{\alpha}=(\alpha_1,\dots,\alpha_m)$ be weights satisfying the Seiberg bound $\alpha_i<Q$ for each $i$. Let $g$ be a compatible Riemannian metric, not necessarily admissible.\\
\noindent {\bf (A)}
If $\partial\Sigma=\emptyset$, then we define the amplitude associated to a continuous nonnegative function $F: H^{s}(\Sigma)\to \R^+$ for some $s<0$ by 
\begin{equation}\label{defampzerobound} 
\mc{A}_{S_g,\bs{\alpha}}(F):=\lim_{\eps\to 0}\langle F(\phi_g)\prod_{i=1}^m V_{\alpha_i,g,\eps}(x_i) \rangle_{\Sigma,g}
\end{equation}
using \eqref{def:pathintegralF} with the field $\phi_g= c+X_{g}$ defined using the GFF of Section \ref{SectionGFF}, and \eqref{defcorrelg} for the $g$-regularisation.
When $F=1$, the amplitude is the correlation function and is written $\mc{A}_{S_g,\bs{\alpha}}$. 
 \vskip 3mm
 
\noindent {\bf (B)}  If  $\partial\Sigma$ has $b>0$ boundary  components, $\mc{A}_{S_g,\bs{\alpha}}$ is
  a function $(F,\tilde{\bs{\varphi}})\mapsto \mc{A}_{S_g,\bs{\alpha}}(F,\tilde{\bs{\varphi}})$ of  the boundary fields $\tilde{\bs{\varphi}}:=(\tilde\varphi_1,\dots,\tilde\varphi_b)\in (H^{s}(\T))^b$  and of continuous nonnegative functions $F:H^s(\Sigma)\to \R^+$ for $s\in (-1/2,0)$.  
For  $\phi_S= X_{S,{\rm D}}+P_S \bs{\tilde \varphi}$, it is  defined by 
\begin{align}\label{amplitude}
 \mc{A}_{S_g,\bs{\alpha}}(F,\tilde{\bs{\varphi}}) :=\lim_{\eps\to 0}Z_{\Sigma,g}\mc{A}^0_{S}(\tilde{\bs{\varphi}})
 \E \big[F(\phi_S)\prod_{i=1}^m V_{\alpha_i,g,\eps}(x_i)e^{-\frac{Q}{2\pi}\int_{\partial\Sigma}k_g\phi_S\dd \ell_g -\frac{Q}{4\pi}\int_\Sigma K_g\phi_S\dd {\rm v}_g -\mu M_\gamma^g (\phi_S,\Sigma)}\big]
\end{align}
where the expectation $\E$ is over the Dirichlet GFF $X_{S,{\rm D}}$, $M_\gamma^g (\phi_S,\Sigma)$ is defined as the limit of \eqref{GMCg} as $\eps\to 0$ and\footnote{Note that $k_g=0$ when $g$ is assumed to be admissible.}
 \begin{align}\label{znormal}
 Z_{\Sigma,g}:=\det (\Delta_{g,{\rm D}})^{-\hf}\exp\big(\frac{1}{8\pi}\int_{\partial\Sigma}k_g\,\dd\ell_g\big),
 \end{align}
$k_g$ being the geodesic curvature of $\pl \Sigma$, and $\mc{A}^0_{S}(\tilde{\bs{\varphi}})$ 
the free field amplitude defined by  
\begin{align}\label{amplifree}
\mc{A}^0_{S}(\tilde{\bs{\varphi}}):=e^{-\frac{1}{2}( \tilde{\bs{\varphi}}, (\mathbf{D}_S-\mathbf{D})  \tilde{\bs{\varphi}})_2}.
\end{align}
For $F=1$, we write $ \mc{A}_{S_g,\bs{\alpha}}( \tilde{\bs{\varphi}})$.
 \end{definition}

The definitions above trivially extend to the situation when $F$ is no  longer assumed to be nonnegative but with the further requirement that $\mc{A}_{S_g,\bs{\alpha}}(|F|)<\infty$ in the case $\partial\Sigma=\emptyset$ and $ \mc{A}_{S_g,\bs{\alpha}}(|F|,\tilde{\bs{\varphi}})<\infty$, $(\dd c\otimes \P_\T)^{\otimes b}$ almost everywhere in the case $\partial\Sigma\not=\emptyset$. By \cite[Theorem 4.4]{GKRV21_Segal}, if $|F|$ is bounded 
\begin{equation}\label{L^2prop_of_amplitudes} 
\sum_{j=1}^m \alpha_j-Q\chi(\Sigma)>0 \Longrightarrow \mc{A}_{S_g,\bs{\alpha}}(F)\in 
L^2(H^{s}(\T)^b,\mu_0^{\otimes b});
\end{equation}
Here $\chi(\Sigma)$ is the Euler characteristic of $\Sigma$ and the left-hand side condition together with $\alpha_j<Q$ for all $j$ is called the \textbf{Seiberg bounds}. 
 
\subsection{Main properties of amplitudes}
Next we state the main properties of this (extended) definition of amplitudes, i.e. the Weyl covariance and the diffeomorphism invariance.  In comparison with \cite{GKRV21_Segal} where only admissible metrics are considered, our formulas will present some further boundary terms, which will be of utmost importance in this paper. Also, in what follows, we will identify functions on $\pl \Sigma=\sqcup_{j=1}^b \pl_j\Sigma$ with functions on $\sqcup_{j=1}^b \T$ by using the parametrisations of the boundary circles: if $\bzeta=(\zeta_1,\dots,\zeta_b)$ are analytic parametrisations $\zeta_j:\T\to \pl\Sigma_j$ of the boundary circles of an admissible surface $(\Sigma,g)$, we shall use the bold notation  $\bs{u}$ to consider a function 
$u\in C^\infty(\pl \Sigma)$ in the parametrisation coordinates 
\[ \bs{u}:=(u\circ \zeta_1,\dots,u\circ \zeta_b) \in C^\infty(\T)^b.\]
Let us also introduce the Liouville anomaly functional: if $g=e^{\omega}g_0$ are two conformal metrics, we let  
\begin{equation}\label{SL0}
S_{\rm L}^0(\Sigma,g_0,g):=\frac{1}{96\pi}\int_{\Sigma}(|d\omega|_{g_0}^2+2K_{g_0}\omega) {\rm dv}_{g_0}.
\end{equation}
where $K_{g_0}$ is the scalar curvature of $g_0$. We record the following properties, which are straightforward to check: if $g'=e^{\omega'}g_0$ and $g=e^{\omega}g_0$, we have ($\nu$ is the inward unit normal for $g_0$)
\begin{align}
& S_{\rm L}^0(\Sigma, g,g')= S_{\rm L}^0(\Sigma, g,g_0)+S_{\rm L}^0(\Sigma, g_0,g')+\frac{1}{48\pi}\int_{\pl \Sigma}\omega'  \pl_\nu\dd \ell_{g_0} \omega\label{IdSL0_1},\\
& S_{\rm L}^0(\Sigma, g_0,g)= -S_{\rm L}^0(\Sigma, g,g_0)-\frac{1}{48\pi}\int_{\pl \Sigma}\omega \pl_\nu \omega \dd \ell_{g_0} .\label{IdSL0_2}
\end{align}

\begin{proposition}\label{Weyl} Let $S=(\Sigma,J, {\bf x},\bs{\zeta})$ be a Riemann surface with parametrised analytic boundary and marked points, let $g_0$ be an admissible Riemannian metric and let  $g=e^{\omega}g_0$  with $\omega\in C^\infty (\Sigma)$. 
Define $S_{g_0}:=(\Sigma,g_0, {\bf x},\bs{\zeta})$ and $S_{g}:=(\Sigma,g, {\bf x},\bs{\zeta})$, and let $\bs{\alpha}$ be weights satisfying the Seiberg bounds.
If  $\omega_\pl:=\omega|_{\pl \Sigma}$, the following properties hold with $c_{\rm L}:=1+6Q^2$:\\
{\bf 1) Weyl covariance:}   
\begin{align*}
\mc{A}_{S_g,\bs{\alpha}}(F,\tilde{\bs{\varphi}})
=&
e^{ -\frac{Q}{2}(\mathbf{D}\bs{\omega}_\pl,  \tilde{\bs{\varphi}})_2-\frac{Q^2}{8}(\mathbf{D}\bs{\omega}_\pl,\bs{\omega}_\pl )_2}
\exp\Big(c_{\rm L}S_{\rm L}^0(\Sigma,g_0,g)-\sum_{i=1}^m\Delta_{\alpha_i}\omega(x_i) \Big)  \\
&\times  \mc{A}_{S_{g_0},\bs{\alpha}}\big(F(\cdot- \tfrac{Q}{2} \omega  ),\tilde{\bs{\varphi}}+\tfrac{Q}{2}\bs{\omega}_\pl\big)
\end{align*}
where $\Delta_{\alpha}:= \frac{\alpha}{2}(Q- \frac{\alpha}{2})$ is the conformal weight and $S_{\rm L}^0$ is defined in \eqref{SL0}.\\

{\bf 2) Diffeomorphism invariance:} Let $\psi:\Sigma\to\psi(\Sigma)$ be an  orientation preserving diffeomorphism and 
$\psi(S_g):=(\psi(\Sigma),\psi_*g,\psi({\bf x}),\psi\circ \bs{\zeta})$, with $\psi\circ \bs{\zeta}:=(\psi\circ \zeta_1,\dots,\psi\circ \zeta_b)$. We have
 \begin{align*} 
\mc{A}_{S_g,\bs{\alpha}}(F, \tilde{\bs{\varphi}})= \mc{A}_{\psi(S_g),\bs{\alpha}}\big( F(\cdot \circ\psi),\tilde{\bs{\varphi}}\big).
\end{align*}
\end{proposition}

\begin{proof}  Let $b$ be the number of connected components of the boundary. For any $u\in C^\infty(\pl \Sigma)$ we 
denote by $Pu$ its harmonic extension in $(\Sigma,J)$.
First we decompose   $\omega=\omega_0+P_S\omega_\pl$ with  $\omega_0\in  C^\infty(\Sigma)$ vanishing at $\pl \Sigma$ and $\omega_\pl=\omega|_{\pl \Sigma}$.
If $g_0':=e^{\omega_0}g_0$, we first prove that 
\begin{equation}\label{firstWeyl}
\mc{A}_{S_g,\bs{\alpha}}(F,\tilde{\bs{\varphi}})=   G(\omega) e^{ -\frac{Q}{2}(\mathbf{D}\bs{\omega}_\pl,  \tilde{\bs{\varphi}})_2}
\mc{A}_{S_{g_0'},\bs{\alpha}}\big(F(\cdot- \tfrac{Q}{2}P_S\omega_\pl ),\tilde{\bs{\varphi}}+\tfrac{Q}{2}\bs{\omega}_\pl\big)  
\end{equation}
where $G(\omega)$ is a deterministic function given by 
\[ G(\omega):=e^{-\frac{c_{\rm L}}{48\pi}(\int_{\pl\Sigma} \omega \pl_{\nu}\omega \dd\ell_{g_0}+2\pi({\bf D}_{S}{\bs{\omega}_\pl},\bs{\omega}_\pl)_2) -\frac{Q^2}{8}(\mathbf{D}\bs{\omega}_\pl,\bs{\omega}_\pl )_2 +c_{\rm L}S_{\rm L}^0(\Sigma,g_0',g)-\sum_i\Delta_{\alpha_i}P_S\omega_\pl(x_i)}.\]
Since  $k_{g}=e^{-\frac{P\omega_\pl}{2}}(k_{g_0'}-\tfrac{1}{2}\partial_\nu P_S\omega_\pl)$   with $\nu$ the $g_0$-unit inward normal at $\partial\Sigma$, the boundary term  in  \eqref{amplitude} becomes
\begin{equation}\label{boundgeod}
\begin{split}
\frac{Q}{2\pi}\int_{\partial\Sigma}k_{g}P_S\tilde{\bs{\varphi}}\,\dd \ell_{g}=&\frac{Q}{2\pi}\int_{\partial\Sigma}(k_{g_0'}- \tfrac{1}{2}\partial_\nu P_S\omega_\pl )P_S\tilde{\bs{\varphi}}\,\dd \ell_{g_0'}\\
=&\frac{Q}{2\pi}\int_{\partial\Sigma}k_{g_0'}P_S\tilde{\bs{\varphi}}\,\dd \ell_{g_0'}+\frac{Q}{2} ( \mathbf{D}_S\bs{\omega}_\pl,\tilde{\bs{\varphi}})_2. 
\end{split}
\end{equation}
Also, from the relation for curvatures $K_{g}=e^{-P_S\omega_\pl}(K_{g_0'} +\Delta_{g_0'}P_S\omega_\pl)=e^{-P_S\omega_\pl} K_{g_0'}$ since $P_S\omega_\pl$ is harmonic, we deduce that the term involving the curvature in  \eqref{amplitude} reads 
\begin{align*}
\frac{Q}{4\pi}\int_{\Sigma} K_{g} (X_{S,{\rm D}} +P_S\tilde{\bs{\varphi}})\dd {\rm v}_{g}=& \frac{Q}{4\pi}\int_{\Sigma}  K_{g_0'} (X_{S,{\rm D}}+P_S\tilde{\bs{\varphi}})\dd {\rm v}_{g_0'}.
\end{align*}
Now recall from \eqref{scalingmeasure} that 
\[M_\gamma^{g}(X_{S,{\rm D}},\dd x)=e^{ \frac{\gamma Q}{2} P_S\omega_\pl(x)}M_\gamma^{g_0'}(X_{S,{\rm D}},\dd x)\] 
since $g=e^{P\omega_\pl}g_0'$.
Therefore, combining with \eqref{scalingvertex}, the expectation in  \eqref{amplitude}  can be written as 
\begin{align*}
&e^{ -\frac{Q}{2}(\bs{\omega}_\pl, \mathbf{D}_S\tilde{\bs{\varphi}})_2 }e^{\sum_i \frac{\alpha_i^2}{4} P_S\omega_\pl(x_i) } 
\\
& \times \E \big[ F(\phi_S)\prod_{i=1}^mV_{\alpha_i,g_0'}(x_i)e^{-\frac{Q}{4\pi}\int_\Sigma K_{g_0'} \phi_S\dd {\rm v}_{g_0'}-\frac{Q}{2\pi}\int_{\partial\Sigma}k_{g_0'}P_S\tilde{\bs{\varphi}}\dd \ell_{g_0'}-\mu  M_\gamma^{g_0'}(X_{S,{\rm D}}+P_S(\tilde{\bs{\varphi}}+\frac{Q}{2}\bs{\omega}_\pl) ,\Sigma) }\big].
\end{align*}
We deduce that
\begin{align*}
\mc{A}_{S_g,\bs{\alpha}}(F,\tilde{\bs{\varphi}})=&\frac{Z_{\Sigma,g}\mc{A}^0_{S}(\tilde{\bs{\varphi}})}{Z_{\Sigma,g_0'}\mc{A}^0_{S}(\tilde{\bs{\varphi}}+\frac{Q}{2}\bs{\omega}_\pl)}
\mc{A}_{S_{g_0'},\bs{\alpha}}\big(F(\cdot-  \frac{Q}{2}P_S\omega_\pl ),\tilde{\bs{\varphi}}+\frac{Q}{2}\bs{\omega}_\pl\big)\\
& \times e^{\frac{Q^2}{8\pi}\int_\Sigma K_{g_0'} P_S\omega_\pl \dd {\rm v}_{g_0'}
+\frac{Q^2}{4\pi}\int_{\pl \Sigma}k_{g_0'}\omega \dd\ell'_{g_0}}
e^{- \frac{Q}{2}(\bs{\omega}_\pl, \mathbf{D}_S \tilde{\bs{\varphi}})_2 }e^{-\sum_i\Delta_{\alpha_i} P_S\omega_\pl(x_i) }.
\end{align*}
The ratio $\mc{A}^0_{S}(\tilde{\bs{\varphi}})/\mc{A}^0_{S}(\tilde{\bs{\varphi}}+\frac{Q}{2}\omega_\pl)$ produces
\[ \frac{\mc{A}^0_{S}(\tilde{\bs{\varphi}})}{\mc{A}^0_{S}(\tilde{\bs{\varphi}}+\frac{Q}{2}\bs{\omega}_\pl)}=e^{ \frac{Q}{2}((\mathbf{D}_S-\mathbf{D})\bs{\omega}_\pl,  \tilde{\bs{\varphi}})_2  +\frac{Q^2}{8}(\bs{\omega}_\pl, (\mathbf{D}_S -\mathbf{D})\bs{\omega}_\pl)_2}\]
and, observing that
\[ \begin{gathered}
\frac{Q^2}{8\pi}\int_\Sigma K_{g_0'} P_S\omega_\pl \dd {\rm v}_{g_0'}=\frac{6Q^2}{96\pi}\int_\Sigma  2K_{g_0'}( P_S\omega_\pl) \dd {\rm v}_{g_0'},
\\
\textrm{ and}\quad  \frac{Q^2}{8}(\mathbf{D}_S \bs{\omega}_\pl, \bs{\omega}_\pl)_2=\frac{6Q^2}{96\pi}\int_\Sigma |\dd P_S\omega_\pl|^2_{g_0'} \dd {\rm v}_{g_0'},
\end{gathered}\]
we end up with
\begin{equation}\label{amplitudesgg_0'}
\begin{split}
\mc{A}_{S_g,\bs{\alpha}}(F,\tilde{\bs{\varphi}})=&\frac{Z_{\Sigma,g}}{Z_{\Sigma,g_0'}}e^{ -\frac{Q}{2}(\mathbf{D}\bs{\omega}_\pl,   \tilde{\bs{\varphi}})_2-\frac{Q^2}{8}(\mathbf{D}\bs{\omega}_\pl,\bs{\omega}_\pl )_2+\frac{Q^2}{4\pi}\int_{\pl \Sigma}k_{g_0'}\omega_\pl \dd\ell_{g_0} }
e^{6Q^2S_{\rm L}^0(\Sigma,g_0', g)-\sum_i\Delta_{\alpha_i}P_S\omega_\pl(x_i)} \\
& \times   
\mc{A}_{S_{g_0'},\bs{\alpha}}\big(F(\cdot- \frac{Q}{2}P_S\omega_\pl ),\tilde{\bs{\varphi}}+\frac{Q}{2}\bs{\omega}_\pl\big) .
\end{split}\end{equation}
Then we use the variations of the regularised Laplacian (see  \cite[Eq (1.17)]{OsgoodPS88}) given by 
\[\frac{Z_{\Sigma,g}}{Z_{\Sigma,g_0'}}=\exp\Big(S_{\rm L}^0(\Sigma,g_0', g)+\frac{1}{24\pi}\int_{\pl \Sigma}k_{g_0'}\omega\, \dd\ell_{g_0'}\Big)\]
and we remark that $k_{g_0'}=-\frac{1}{2}\pl_{\nu}\omega_0=-\frac{1}{2}\pl_{\nu}(\omega-P_S\omega_\pl)$ and $\dd \ell_{g_0'}=\dd \ell_{g_0}$, so that 
\[\int_{\pl \Sigma}k_{g_0'}\omega\, \dd\ell_{g_0'}=-\frac{1}{2}\int_{\pl \Sigma}\omega \pl_{\nu}\omega\, \dd\ell_{g_0}
-\pi({\bf D}_{S}\bs{\omega}_\pl,\bs{\omega}_\pl)_2.\]
Combining the  last two equations with \eqref{amplitudesgg_0'} shows \eqref{firstWeyl}.
Next, we can use  \cite[Proposition 4.7]{GKRV21_Segal} which says that
\[\mc{A}_{S_{g_0'},\bs{\alpha}}(F,\tilde{\bs{\varphi}})=\mc{A}_{S_{g_0},\bs{\alpha}}(F(\cdot-\omega_0),\tilde{\bs{\varphi}})\exp\Big(c_{\rm L}S_{\rm L}^0(\Sigma,g_0,g'_0)-\sum_i\Delta_{\alpha_i}\omega_0(x_i) \Big)\] 
and we use the following identity (recall \eqref{IdSL0_1}) 
\[S_{\rm L}^0(\Sigma,g_0',g)+S_{\rm L}^0(\Sigma,g_0,g'_0)=
S_{\rm L}^0(\Sigma,g_0,g)+\frac{1}{48\pi}\int_{\pl \Sigma}\omega \pl_\nu\omega \,\dd\ell_{g_0}+\frac{1}{24}({\bf D}_S \bs{\omega}_\pl,\bs{\omega}_\pl)_2\]
to deduce the Weyl covariance. For diffeomorphism invariance, this follows easily from the relation $\phi_S=\phi_{\psi(S)}\circ \psi$.
 \end{proof}

\subsection{Gluing of amplitudes}
 
In this subsection, we summarise the results obtained in  \cite[Section 5]{GKRV21_Segal}  about the gluing of  amplitudes, which identifies amplitudes as a functor on the category of 2-cobordisms. The twist here, in comparison with  \cite[Section 5]{GKRV21_Segal}, is that the metric will not be supposed to be admissible along the cutting curve. 

Let $S_g=(\Sigma,g,{\bf x},\bzeta)$ be an admissible Riemannian surface with $b$ boundary components and $S=(\Sigma,J,{\bf x},\bzeta)$ the underlying Riemann surface with parametrised boundary and marked points, let $\bs{\alpha}$ be weights attached to the marked points and satisfying Seiberg bounds. Let $\zeta:\T\to\cC\subset\Sigma$ be an analytic curve such that cutting along $\cC$ produces two surfaces $S^1_{g_1}=(\Sigma_1,g_1,{\bf x}_1,\bzeta_1)$ and $S^2_{g_2}=(\Sigma_2,g_2,{\bf x}_2,\bzeta_2)$, where $g_i=g|_{\Sigma_i}$, ${\bf x}_i={\bf x}\cap\Sigma_i$, and let $\bs{\alpha}_i$ be the weights in $\bs{\alpha}$ attached to ${\bf x}_i $ for $i=1,2$. 
Note that $g_i$ ($i=1,2$) are not  necessarily admissible along $\mc{C}$. Let $\bs{\zeta}_i$ collect the parametrisations of the boundary components of $\Sigma_i$ (thus including $\zeta$). Boundary fields on $\Sigma_i$ will then be written as $(\tilde{\bvarphi}_i,\tilde{\varphi})\in H^{s}(\T)^{b_i+1}$ for $s\in (-1/2,0)$, the last entry corresponding to the boundary field on $\cC$.  The surface $\Sigma_i$ has $b_i+1$ boundary components, with $b=b_1+b_2$ and  the boundary component coming from $\cC$ parametrised by $\zeta$.

\begin{proposition}\label{glue1}
Let $F_1,F_2$ be measurable  nonnegative functions respectively  on $H^{s}(\Sigma_1)$ and  $H^{s}(\Sigma_2)$ for $s<0$ and let us denote by $F_1\otimes F_2$ the functional on $H^{s}(\Sigma)$ defined by  \[F_1\otimes F_2(\phi_S):=F_1(\phi_{S}|_{\Sigma_1})F_2(\phi_S|_{\Sigma_2}).\] 
Then $\mu_0^{\otimes b}$-a.e. $\tilde{\bvarphi}=(\tilde{\bvarphi}_1,\tilde{\bvarphi}_2)\in H^{s}(\T)^{b_1+b_2}$, we have
\[\mc{A}_{S_g,\bs{\alpha}}(F_1\otimes F_2,\tilde{\bvarphi})=C\int \mc{A}_{S^1_{g_1},\bs{\alpha}_1}(F_1,(\tilde{\bvarphi}_1,\tilde{\varphi}))\mc{A}_{S^2_{g_2},\bs{\alpha}_2}(F_2,(\tilde{\bvarphi}_2,\tilde{\varphi}))\d\mu_0(\tilde{\varphi})\]
where $C= \frac{1}{(\sqrt{2} \pi)} $ if $\partial\Sigma \not =\emptyset$ and $C= \sqrt{2} $ if $\partial\Sigma =\emptyset$.
\end{proposition}

We can also consider the case of self-gluing: we assume that the surface cut along $\cC$ is the compactification $\Sigma_{\mc{C}}$ of the connected open surface $\Sigma\setminus\cC$, with $b+2$ boundary components parametrised by $\bzeta_\cC:=(\bzeta,\zeta,\zeta)$, in which case the boundary fields on $\Sigma_\cC$ will be parametrised by $(\tilde{\bvarphi}',\tilde{\varphi}_1,\tilde{\varphi}_2)$, the last two entries corresponding to the boundary field on $\cC$.
\begin{proposition}\label{glue2}
For all $F$ measurable  nonnegative on $H^{s}(\Sigma)$ (with $s<0$) and $\mu_0^{\otimes b}$-a.e. $\tilde{\bvarphi}'\in(H^{s}(\T))^{b}$, we have
\[\mc{A}_{S_g,\bs{\alpha}}(F,\tilde{\bvarphi}')=C\int\mc{A}_{\Sigma_\cC,g,{\bf x},\bzeta_\cC,\bs{\alpha}}(F,\tilde{\bvarphi}',\tilde{\varphi},\tilde{\varphi})\d\mu_0 (\tilde{\varphi})\]
where $C= \frac{1}{(\sqrt{2} \pi)} $ if $\partial\Sigma \not =\emptyset$ and $C= \sqrt{2} $ if $\partial\Sigma =\emptyset$.
\end{proposition}
 
\noindent {\it Proof of Prop. \ref{glue1} and \ref{glue2}.}
Both propositions above are proved in \cite{GKRV21_Segal} under the condition that the metrics under consideration are admissible along the cutting curve $\mc{C}$. It is straightforward to deduce the case of non-admissible metrics.  For this, the first step of the argument  just consists in applying the Weyl anomaly (see Proposition \ref{Weyl}) to replace the metrics $g_1,g_2$ on $\Sigma_1,\Sigma_2$ by  conformally  equivalent admissible ones along the boundary $\cC$. We can then  apply   \cite[Prop. 5.1 and 5.2]{GKRV21_Segal} to glue the amplitudes on $\Sigma_1,\Sigma_2$ and get an amplitude on $\Sigma$ equipped with a metric obtained by  gluing the two  admissible metrics on both surfaces $\Sigma_1,\Sigma_2$. Once this is done,  we re-apply  the Weyl anomaly to recover the original metric $g$ on $\Sigma$. We are then left with a further boundary Girsanov type term of the form 
$$e^{ -\frac{Q}{2}(\mathbf{D}_{S,\cC}\bs{\omega}_\pl,  \tilde{\bs{\varphi}})_2-\frac{Q^2}{8}(\mathbf{D}_{S,\cC}\bs{\omega}_\pl,\bs{\omega}_\pl )_2}$$
obtained by gathering both similar terms coming from applying the Weyl anomaly on $\Sigma_1,\Sigma_2$. Applying the Girsanov transform to this term has the effect of shifting the boundary field $\tilde{\bs{\varphi}}$ to $\tilde{\bs{\varphi}}-\frac{Q}{2}\bs{\omega}_\pl$, hence our claim.
\qed

\section{Annulus propagators}\label{Annulus_Propagator}

In this section, we shall focus on the particular case when $\Sigma$ is an annulus with parametrised boundary, the surface is then of type $(0,0,1,1)$. We start with an observation that the gluing of amplitudes  allows us to construct a projective representation of the semigroup of annuli. We shall see below that for a subfamily of these annuli, called holomorphic annuli, there is an exact (non-projective) representation. 

\subsection{Amplitudes of annuli and representation of Segal semigroup}\label{projective_rep}

Let $S=(A,J^{A},\bzeta^A)$ be an annulus  and $\bzeta^A=(\zeta^A_1,\zeta^A_2)$ the analytic parametrisation of the two boundary circles. Let $g_0$ be an admissible metric on $A$. 
There is a unique minimizer $\omega_A\in H_0^1(A)$ to the functional 
\[\omega\in H_0^1(A)\mapsto S_{\rm L}^0(A,g_0,e^{\omega}g_0)\] 
which satisfies the elliptic equation $\Delta_{g_0}\omega_A=K_{g_0}$, and $\omega_A\in C^\infty(A)$ by elliptic regularity. This provides a metric $g_A=e^{\omega_A}g_0$ with  scalar curvature $K_{g_A}=0$ but the geodesic curvature $k_{g_A}$ of the boundary is no  longer $0$. We note that $g_A$ does not depend on the choice of admissible metric $g_0$: indeed, any other admissible   metric $g'_0$ has the form $e^{\omega_0}g_0$ for some $\omega_0\in C_c^\infty(A^\circ)$ and we can use the formula, when $\omega|_{\pl A}=0$,
\[S_{\rm L}^0(A, g'_0,e^{\omega}g'_0)=S_{\rm L}^0(A, g'_0,g_0)+S_{\rm L}^0(A, g_0,e^{\omega+\omega_0}g_0)\] 
to deduce that the minimizer of $\omega\mapsto S_{\rm L}^0(A, g'_0,e^{\omega}g'_0)$ is $\omega'_A=-\omega_0+\omega_A$, thus producing the same metric $e^{\omega'_A}g_0'=g_A$ as before. We call $g_S$ the uniform metric on $S=(A,J_\C,\bzeta^A)$.

For two annuli $S_1=(A_1,J^{A_1}, \bs{\zeta}^{A_1})$ and $S_2=(A_2,J^{A_2}, \bs{\zeta}^{A_2})$ of type $(0,0,1,1)$, with  
we denote by $A_1\#A_2$ the  gluing of $A_1$ with $A_2$ by identifying the outgoing boundary $\pl_1A_2$ of $A_2$ to the incoming one $\pl_2A_1$ of $A_1$ via $\zeta^{A_2}_1(e^{{\rm i}\theta})=\zeta^{A_1}_2(e^{{\rm i}\theta})$, 
and $g_{S_1}\# g_{S_2}$ the metric obtained by gluing $g_{S_1}$ with $g_{S_2}$; note that this metric is piecewise smooth and continuous with a jump of the derivative at the gluing curve. The glued annulus $A_1\# A_2$ is equipped with the glued complex structure $J^{A_1\#A_2}$
and with  the boundary parametrisation 
$\bs{\zeta}^{A_1\#A_2}=(\zeta^{A_1}_1,\zeta^{A_2}_2)$ if $\bs{\zeta}^{A_j}=(\zeta^{A_j}_{1},\zeta^{A_j}_{2})$ for $j=1,2$, and let $S_1\#S_2:=(A_1\#A_2,J^{A_1\# A_2},\bzeta^{A_1\#A_2})$.
The gluing map endows the set   of annuli with a semigroup structure.
Consider the functional
\[\Omega(S_1,S_2)=S_{\rm L}^0(A_1\#A_2,g_{S_1\#S_2},g_{S_1}\# g_{S_2}).\]
\begin{proposition}\label{prop:proj_rep}
The following map 
\[ S=(A,J^A,\bs{\zeta}^A) \mapsto \mc{A}_{A,g_S,\bs{\zeta}^A}\in \mc{L}(\mc{H})\]
is a projective representation of the semigroup of annuli into the space of bounded operators, in the sense that
\[  \mc{A}_{A_1,g_{S_1},\bs{\zeta}^{A_1}}\circ  \mc{A}_{A_2,g_{S_2},\bs{\zeta}^{A_2}}=e^{-c_{\rm L}\Omega(S_1,S_2)}\mc{A}_{A_1 \# A_2,g_{S_1\# S_2},\bs{\zeta}^{A_1\#A_2}} 
\]
where $S_1=(A_1,J^{A_1},\bzeta^{A_1})$ and $S_2=(A_2,J^{A_2},\bzeta^{A_2})$
and $\Omega$ satisfies the cocycle relation
\[ \Omega(S_1\# S_2,S_3)+\Omega(S_1,S_2)=\Omega(S_1,S_2\# S_3)+\Omega(S_2,S_3).\]
\end{proposition}
\begin{proof}
The fact that $\mc{A}_{A,g_{S},\bs{\zeta}^A}$ is bounded on $\mc{H}$ follows from \cite[Theorem 4.4]{GKRV21_Segal}, which produces a pointwise bound on the amplitude: there is $a>0$ such that
\[ |\mc{A}_{A,g_S,\bs{\zeta}^A}(c_1,\varphi_1,c_2,\varphi_2)|\leq e^{-a(c_1-c_2)^2}B(\varphi_1,\varphi_2)\]
with $B\in L^2(\Omega_\T^2)$. For $F,G\in \mc{H}$, we thus have by Cauchy--Schwarz inequality
\[ |\cjg\mc{A}_{A,g_S,\bs{\zeta}^A}F,G\cjd_{\mc{H}}| \leq \int_{\R^2}\|B\|_{L^2(\Omega_\T^2)}e^{-a(c_1-c_2)^2}\|F(c_1,\cdot)\|_{L^2(\Omega_\T)}
\|G(c_2,\cdot)\|_{L^2(\Omega_\T)}\d c_1\d c_2 \]
and we can use the fact that the convolution operator on $\R$ by $e^{-ac^2}$ is bounded on $L^2(\R,\d c)$ to deduce that there is $C>0$ such that 
\[|\cjg\mc{A}_{A,g_S,\bs{\zeta}^A}F,G\cjd_{\mc{H}}|\leq C \|B\|_{L^2(\Omega_\T^2)} \|F\|_{\mc{H}}  \|G\|_{\mc{H}}.\]
We can use Proposition \ref{Weyl} to write $\mc{A}_{A_j,g_{S_j},\bs{\zeta}^{A_j}}= e^{c_{\rm L}S^0_{\rm L}(A_j,g_j,g_{S_j})}\mc{A}_{A_j,g_j,\bs{\zeta}^{A_j}}$ for some admissible metric $g_j$ conformal to $g_{S_j}$, then we can apply the gluing result 
of Proposition \ref{glue1} for admissible metrics, and then write the result in terms of $g_{S_1\#S_2}$, which makes the additional term $e^{c_{\rm L}S_{\rm L}^0(A_1\#A_2,g_{S_1\#S_2},g_1\#g_2)}$ appearing. Gathering everything and using the cocycle property of $S_{\rm L}^0$, one obtains the announced result. The cocycle relation on $\Omega$ follows from the same reasoning.
\end{proof}

\subsection{Functional spaces} \label{sec:funct_space}
Let ${\rm Hol}(\D)$ be the complex vector space of holomorphic maps on the unit disk $f:\D\to\C$ which are smooth up to the boundary $\T=\pl\D$. 
Every such $f$ has a power series expansion 
\[f(z)=\sum_{n=-1}^\infty f_nz^{n+1},\]
providing this way  complex linear coordinates $f_n:{\rm Hol}(\D)\to\C$.  For $\eps\geq 0$ fixed, we introduce the countable family of seminorms (for $k\geq 0$) 
\begin{equation}\label{seminorm}
\|f\|^2_{\eps,k}=\sum_{n=-1}^\infty (1+\eps)^{2|n|+2}(1+|n|)^{2k}|f_n|^2 + \sum_{n=-\infty}^{-2} (1+\eps)^{2|n|}n^{2k}|f_n|^2.
\end{equation}
We endow ${\rm Hol}(\D)$ with the family of seminorms $( \|\cdot\|^2_{\eps=0,k})_{k\geq 0}$ for $\eps=0$ (only the indices $n\geq -1$ actually contribute), which turn it   into a Fr\'echet space. Let ${\rm Hol}^\bullet(\D)$ be the closed subspace of ${\rm Hol}(\D)$ that consists of functions $f$ such that $f(0)=0$, equipped with the induced topology.

Next, for $\eps>0$ small, we consider the subspace ${\rm Hol}_\eps(\D)$ of ${\rm Hol}(\D)$ that consists of  holomorphic functions on the open disk $(1+\eps)\D$ that are smooth up to the boundary  of $(1+\eps)\D$. This space is a Frechet space by using the semi-norms $( \|\cdot\|^2_{\eps,k})_{k\geq 0}$.
Let ${\rm Hol}^\bullet_\eps(\D):={\rm Hol}^\bullet(\D)\cap {\rm Hol}_\eps(\D)$.

 We also consider the annuli $\A_{\delta_1,\delta_2}:=\{z\in \C\, |\, |z|\in [\delta_1,\delta_2]\}$ if $\delta_1<\delta_2$. We use the shortcut $\A_\delta$ for $\A_{\delta,1}$ if $\delta<1$.
Then we define ${\rm Hol}_\eps(\A)$ as the complex vector space of holomorphic maps $f:\A_{\delta,\delta^{-1}} \to\C$ smooth up to the boundary of $\A_{\delta,\delta^{-1}}$ with $\delta=(1+\eps)^{-1}$. Each function $f\in {\rm Hol}_\eps(\A)$ has a Laurent series expansion 
\[f(z)=\sum_{n=-\infty}^\infty f_nz^{n+1}\]
converging in the annulus.
The topology on ${\rm Hol}_\eps(\A)$ is then induced by the family of seminorms $( \|\cdot\|^2_{\eps,k})_{k\geq 0}$.

\subsection{Quantisation of Segal's semigroup of holomorphic annuli} 
Segal's semigroup of holomorphic annuli is the open subset of ${\rm Hol}^\bullet(\D)$ defined by
\begin{equation}\label{defS}
\mc{S}:=\{f\in {\rm Hol}^\bullet(\D)\,\,\, |\, f \textrm{ is a biholomorphism onto its image and } f(\D)\subset\D^\circ\}
\end{equation}
where $\D^\circ$ stands for the interior of the unit disk $\D$. As an open subset of $\mathrm{Hol}^\bullet(\D)$, $\cS$ comes equipped with a structure of infinite dimensional manifold modeled on $\mathrm{Hol}^\bullet(\D)$. Moreover, the composition 
$\cS\times\cS\to\cS$, which maps $(f_1,f_2)\mapsto f_1\circ f_2$, is holomorphic, turning $\cS$  into an analytic semigroup. We will also need to consider the following sets 
\begin{equation}\label{S_eps_S_+}
 \mc{S}_\eps:={\rm Hol}_\eps^\bullet (\D)\cap \mc{S}, \quad  
 \mc{S}_>:=\cup_{\eps\in (0,1)}\mc{S}_\eps=\{f\in\mc{S}\, |\, f_{|_\T}\text{ is analytic}.\}
\end{equation}
and we equip $ \mc{S}_\eps$ with the topology induced by ${\rm Hol}_\eps^\bullet (\D)$. 
 Now we quantise  $\mc{S}$ along the lines described in  \cite{BGKRV} in the sense that we give a representation of $\mc{S}$ into a space of bounded operators acting on the Liouville Hilbert space $\mc{H}$, using   the Gaussian free field in $\D$ and Gaussian multiplicative chaos theory. 
Let $X_\D:=X_{S,{\rm D}}$ be the GFF on $S=(\D,J_\C,\zeta^\D)$ with Dirichlet condition at the boundary $\T$, $J_\C$ being
the canonical complex structure on $\C$ and $\zeta^\D(e^{{\rm i}\theta}):=e^{{\rm i}\theta}$. Let $X=X_\D+P\tilde\varphi$, with a boundary field $\tilde \varphi=c+\varphi\in H^s(\T)$ ($s<0$), with $P\tilde{\varphi}:=c+P_S\varphi$ being the harmonic extension of $\tilde{\varphi}$ in $(\D,J_\C,\zeta^\D)$. We write $M_\gamma$ for the Gaussian multiplicative chaos measure $M^{g=|\dd z|^2}_\gamma$. In what follows, $\E_\varphi$ stands for conditional expectation with respect to the field $\varphi$. Recall now the following definition introduced in \cite{BGKRV}.
\begin{definition}\label{defTf}
For each $f\in \mc{S}$,  the Liouville quantisation of $f$ is the operator $\bT_f$ defined by  (for $s<0$)
\begin{align*}
& \bT_f: C_b^0(H^{s}(\T))\to L^\infty(H^{s}(\T)),\\ 
&\bT_fF(\tilde{\varphi}):=|f'(0)|^\frac{Q^2}{2}\E_\varphi\Big[F\Big(\big(X \circ f +Q\log\Big|\frac{f'}{f}\Big|\big)\big|_\T\Big)e^{-\mu\int_{\A_f}|x|^{-\gamma Q}M_\gamma (X, \dd x) }\Big], \end{align*}
where $\A_f=\D\setminus f(\D^\circ)$ and $C_b^0(H^{s}(\T))$ denotes the space of bounded continuous maps on $H^s(\T)$, and $X=X_\D+P\tilde\varphi$.
\end{definition}

The above definition provides a representation of  $\mc{S}$ in the following sense.
 \begin{proposition}\label{comprule}
The map $f\in \mathcal{S}\mapsto {\bf T}_f$ obeys the composition rule   ${\bf T}_{f_1}\circ  {\bf T}_{f_2}={\bf T}_{f_1\circ f_2}$.
\end{proposition}  

\begin{proof}
We have 
\begin{equation*}
\bT_{f_1\circ f_2}F(\tilde{\varphi}):=|(f_1\circ f_2)'(0)|^\frac{Q^2}{2}\E_\varphi\Big[F\Big(\big(X\circ (f_1\circ f_2)+Q\log\Big|\frac{(f_1\circ f_2)'}{f_1\circ f_2}\Big|\big)\big|_\T\Big)e^{-\mu\int_{\A_{f_1\circ f_2}}|x|^{-\gamma Q}M_\gamma (X, \dd x) }\Big].
\end{equation*}
\textbf{Step 1.} We first consider the case $\mu=0$. We use the Markov property of the GFF along with conformal invariance. When $A \subset \D$, we denote $X_A$ the Dirichlet GFF on $(A,J_\C)$ and  $P_A (v)$ the harmonic extension in $(A,J_\C)$ of a function $v$ defined on $\partial A$ . By conformal invariance of the Dirichlet GFF, we have
\begin{equation}\label{GFFConfInv}
(X_{f_1 (\D)} (f_1(z)))_{z \in \D} \overset{({\rm Law})}{=}  (X_{\D} (z))_{z \in \D} 
\end{equation}
and also standard complex analysis gives for any function $v$ defined on $\D$ that for $z \in \D$ 
\begin{equation}\label{PfConfInv}
P_{f_1 (\D)} (  v|_{\partial f_1(\D)} )(f_1(z))= P_{\D} (  v \circ f_1 |_{\T}  ) (z)  .
\end{equation}
Indeed in the above, both functions are harmonic and have the same boundary value. Therefore, conditioning on $X_{\D}\circ f_1 |_\T $, we get by the Markov property
\begin{align*}
& \E_{ X_{\D}\circ f_1 |_\T   }\Big[F\Big(\big(X\circ (f_1\circ f_2)+Q\log\Big|\frac{(f_1\circ f_2)'}{f_1\circ f_2}\Big|\big)\big|_\T\Big) \Big] = \E_{  X_{\D}\circ f_1|_\T  }[F(Y)]], \\
& Y:=  \tilde X_{f_1(\D)} \circ (f_1 \circ f_2)+P_{\D} ( \tilde{\varphi}|_{\T}) \circ (f_1\circ f_2)+P_{f_1 (\D)} (  X_{\D} |_{\partial f_1(\D)}  ) \circ (f_1 \circ f_2)   +Q\log\Big|\frac{(f_1\circ f_2)'}{f_1\circ f_2}\Big|\big)\big|_\T
\end{align*}
where $\tilde X_{f_1(\D)}$ is a Dirichlet GFF on $f_1(\D)$, independent from $X_{\D}\circ f_1|_\T$. Using the previous considerations, the latter quantity is equal to
\begin{align*}
 \E_{ X_{\D}\circ f_1|_\T  }\Big[F\Big(   \big(\tilde X_{\D} \circ f_2+P_{\D} (\tilde{\varphi} |_{\T}) \circ (f_1\circ f_2) +P_{\D} (  X_\D \circ f_1 |_\T) \circ f_2   +Q\log\Big|\frac{(f_1\circ f_2)'}{f_1\circ f_2}\Big|\big)\big|_\T\Big) \Big]
\end{align*}
where $\tilde X_{\D}$ is independent from $X_{\D}$. Now we write
\begin{align*}
  &(\tilde X_{\D} \circ f_2 +P_{\D} (  X_\D \circ f_1 |_\T) \circ f_2  +Q\log\Big|\frac{(f_1\circ f_2)'}{f_1\circ f_2}\Big|)|_\T\\
  &=  (\tilde X_{\D} \circ f_2 +P_{\D} (  (X_\D \circ f_1 + Q \log \Big|\frac{f_1'}{f_1}\Big| )|_\T) \circ f_2   +
 Q \log \Big|\frac{f_2'}{f_2}\Big| +\delta ) |_\T 
\end{align*}
where $\delta$ is defined on $\T$ by
 \begin{equation*}
\delta= -Q \log \Big|\frac{f_2'}{f_2}\Big|  +Q\log\Big|\frac{(f_1\circ f_2)'}{f_1\circ f_2}\Big|  -QP_{\D} \big( \log \Big|\frac{f_1'}{f_1}\Big|_\T \big)  \circ f_2  .
\end{equation*}
One can notice that $\delta= \tilde{\delta} \circ f_2$ where $\tilde{\delta}$ is defined on $\D$ as
\begin{equation*}
\tilde{\delta}(z)= Q \log \Big|\frac{f_1'(z) z}{f_1 (z)}\Big|  -QP_{\D} ( \log \Big|\frac{f_1'}{f_1}\Big|_\T ) (z).
\end{equation*}
Now $\tilde \delta$ is harmonic on $\D$ and equal to $0$ (recall that $f_1$ has a unique simple zero at $0$) on $\T$ hence it is zero and so is $\delta$. Gathering the above considerations, we get that 

\begin{align*}
& \E_{ X_{\D}\circ f_1  |_\T  }\Big[F\Big(\big(X\circ (f_1\circ f_2)+Q\log\Big|\frac{(f_1\circ f_2)'}{f_1\circ f_2}\Big|\big)\big|_\T\Big)  \Big] \\
& =  \E_{ X_{\D}\circ f_1|_\T  }\Big[F\Big(\big(  (\tilde X_{\D} \circ f_2 +P_{\D} (  (X_\D \circ f_1 +P_{\D} (\varphi |_{\T}) \circ f_1 + Q \log \Big|\frac{f_1'}{f_1}\Big|\big) \big|_\T) \circ f_2   +
 Q \log \Big|\frac{f_2'}{f_2}\Big|  \big) \big|_\T \Big)  \Big].
\end{align*}

Finally we note that for all $z \in \D$ we have $P_{\D} ( (P_{\D} (\varphi |_{\T}) \circ f_1 )|_{\T})(z)= P_{\D} ( \varphi |_{\T})(f_1(z))$ (two harmonic functions with the same boundary values) for all $z\in \D$ hence $(P_{\D} ( (P_{\D} (\varphi |_{\T}) \circ f_1 ))|_{\T}) \circ f_2) |_{\T}= P_{\D} ( \varphi |_{\T}) \circ (f_1\circ f_2) |_{\T}$. This yields the result.

\textbf{Step 2.} Now, we turn to the potential part, i.e. the case $\mu\not=0$. Recall that the Markov property of the GFF says that for $z \in f_1(\D)$ we have
\begin{equation*}
X_\D(z)=  \tilde{X}_{f_1(\D)}(z)+ P_{ f_1(\D)  }(  X_\D|_{\partial f_1(\D)}) (z)
\end{equation*}
hence, by conformal invariance (\eqref{GFFConfInv} and \eqref{PfConfInv}), we have for $x\in \D$
\begin{equation*}
(X_\D \circ f_1 )(x)=  \tilde{X}_{\D}(x)+ P_\D(  X_\D \circ f_1|_{\T} ) (x).
\end{equation*}
We also have the identity
\begin{equation*}
\log \Big|\frac{f_1'(x)}{f_1(x)}\Big|=    P_\D( \log \Big|\frac{f_1'}{f_1}\Big|{|_{\T}}   ) (x)+    \log \frac{1}{|x|}.
\end{equation*}
Then we can proceed along the same lines as previously to take care of the potential part by noticing that
\begin{equation*}
\int_{\A_{f_1\circ f_2}}|x|^{-\gamma Q}M_\gamma (X, \dd x)= \int_{\A_{f_1}}|x|^{-\gamma Q}M_\gamma (X, \dd x)+\int_{f_1(\D) \setminus f_1 (f_2(\D))}|x|^{-\gamma Q}M_\gamma (X, \dd x)
\end{equation*}
and making the change of variable $x=f_1(y)$ in the second integral to get that 
\[\begin{split}
\int_{f_1(\D) \setminus f_1 (f_2(\D))}|x|^{-\gamma Q}M_\gamma (X, \dd x)& =\int_{\D  \setminus  f_2(\D)}\Big(\frac{|f_1'(x)|}{|f_1(x)|}\Big)^{\gamma Q}M_\gamma (X, \dd x)\\
&=\int_{\D  \setminus  f_2(\D)}|x|^{-\gamma Q}M_\gamma (X+ QP_\D( \log \Big|\frac{f_1'}{f_1}\Big|{|_{\T}}   ) , \dd x).
\end{split}\]
Here we have used the conformal change of the GMC, see for instance \cite[Theorem 2.8]{Berestycki_lqggff}.
\end{proof}

 \subsection{Integral kernels of propagators are amplitudes}\label{sub:intk}
In this subsection, we compute the integral kernels of the operator $\bT_f$ for those $f \in \mc{S}$ that are analytic across the boundary of $\D$, namely for $f\in \mc{S}_>$ (recall \eqref{S_eps_S_+}). 
More precisely, we relate the operator $\bT_f$ to the amplitude of the annulus $\A_f$ equipped with some specific metric and boundary parametrisation. To identify these objects, we need to introduce some material. Let us choose $f\in \mc{S}_>$.   
 We define the curves $\mc{C}$, the simply connected domain $\D_f$ and the annulus $\A_f$ by 
 \[ \mc{C}:=f(\T)\subset \D^\circ, \quad \D_f:=f(\D),\quad \mathbb{A}_f:=\D\setminus \D^\circ_f,\]
see Figure \ref{pic1}. We equip the annulus $\mathbb{A}_f$ with the metric 
 \begin{equation}\label{gA}
 g_{\mathbb{A}}=\frac{|\dd z|^2}{|z|^2}
 \end{equation} 
and we consider the boundary parametrisation $\bs{\zeta}_f=(\zeta_1,\zeta_2)$ with $\zeta_1(e^{{\rm i}\theta})=e^{{\rm i}\theta}$ for $\pl_1\A_f=\T$ 
and $\zeta_2(e^{{\rm i}\theta})=f(e^{{\rm i}\theta})$ for 
$\pl_2\mathbb{A}_f=\mc{C}$. 
The metric $g_{\mathbb{A}}$ has scalar curvature equal to $0$, it is not admissible on $\mathbb{A}_f$ near $\mc{C}$ 
but it is admissible near $\T$. 

 \begin{figure}[h]\begin{tikzpicture}
\draw[fill=lightgray,draw=blue,very thick] (0,0) circle (2) ;
\coordinate (Z1) at (1.5,0) ;
\coordinate (Z2) at (0,0) ;
\coordinate (Z3) at (-0.8,1.3) ;
\coordinate (Z4) at (-0.7,-0.9) ;
\draw[>=latex,draw=red,very thick,fill=white] (Z1) to[out=120,in=40] (Z2) to [out=220,in=40](Z3) to [out=220,in=70](Z4) to [out=250,in=300](Z1);
  \draw (Z2) node[above,red]{$ \mc{C}$} ;
\draw (2,0) node[right,blue]{$ \T$} ;
 \draw (0.5,0.5) node[above]{$ \mathbb{A}_f$} ;
 \draw (0.3,-0.7) node[above]{$ \D_f$} ;
\end{tikzpicture}
\caption{Grey area: the annulus $\A_f$. White area: the image $f(\D)=\D_f$.} \label{pic1}
\end{figure}
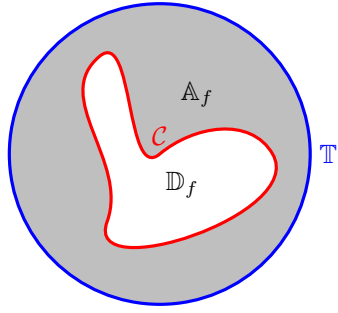  
For $\eps>0$ small enough, the map 
\[f: \A_{1,1+\eps}\to f(\A_{1,1+\eps})\subset \D\] 
is a holomorphic chart near the interior boundary $\pl_2\mathbb{A}_f=\mc{C}$ of $\mathbb{A}_f$. 
Let us consider the harmonic function $h$ defined on a neighborhood of $ \D_f$ by 
\begin{equation}\label{defh_t}
\forall |z|\leq 1+\eps , \quad h(f(z))=\log\frac{|zf'(z)|}{|f(z)|}.
\end{equation} 
The pull-back of  $g_\mathbb{A}$ near $\mc{C}$ under the map $f$ is
\begin{equation}\label{f_t*g_A}
f^*g_\mathbb{A}= e^{2h\circ f} g_{\mathbb{A}}.
\end{equation}
 Any admissible metric $g$ on $\mathbb{A}_f$ with parametrisation $\bzeta=(\zeta_1,\zeta_2)$ of the boundary as above can be written in the form
\begin{equation}\label{defg_tadmissible}
 g_f=e^{-2 \chi_f}g_{\mathbb{A}}
 \end{equation}
where $\chi_f \in C^\infty(\mathbb{A}_f)$ is equal to $h$ near $\mc{C}$ and $0$ near $\T$. A possible choice, 
that depends smoothly on $f$ near a given $f_0\in \mc{S}_>$, is to take $\chi_f=h\chi$ where $\chi \in C^\infty(\mathbb{A}_f)$ equals $1$ near $\mc{C}$ and $0$ near $\T$ and does not depend on $f$ but only on $f_0$. For later use, we define
\begin{equation}\label{defomegat}
\omega(e^{{\rm i}\theta}):=h(f(e^{{\rm i}\theta}))=\log \frac{|f'(e^{{\rm i}\theta})|}{|f(e^{{\rm i}\theta})|}
\end{equation} 
and $\bs{\omega}=(0,\omega)\in C^\infty(\T)^2$.

The main results  of this subsection are the following proposition and its corollary, which relate   the integral kernel of ${\bf T}_f$  to   the amplitude of the annulus $\mathbb{A}_f$. 
\begin{proposition}\label{compTf}
If $f\in \mc{S}_>$ then 
\begin{equation}\label{prop_amplitudeA_t^v}
{\bf T}_fF(\tilde{\varphi}_1)= C_f
\int F(\tilde{\varphi}_2)e^{Q(\mathbf{D} \omega ,  \tilde{ \varphi}_2 )_2} \mc{A}_{\mathbb{A}_f,g_{\mathbb{A}},\bzeta_f}(\tilde{\bs{\varphi}}-Q\bs{\omega})\dd\mu_0(\tilde{\varphi}_2)
\end{equation}
where $\tilde{\bs{\varphi}}-Q\bs{\omega}=(\tilde{\varphi}_1,\tilde{\varphi}_2-Q\omega)$, with $\omega$ given by \eqref{defomegat} and 
\[C_f=\frac{1}{\sqrt{2}\pi}e^{\frac{c_{\rm L}}{12}\log|f'(0)|+\frac{1}{12}({\bf D}\omega,\omega)_2}.
\]
\end{proposition} 
This proposition is proved below. For now, let us notice that, 
combining this proposition with Proposition \ref{Weyl}, we obtain the following:
\begin{corollary}\label{cor:annulus_propagator}
For $f\in \mc{S}_>$   
\[ \sqrt{2}\pi e^{\frac{c_\mathrm{L}}{12}W(f,g_f)} {\bf T}_fF(\tilde{\varphi}_1)= 
\int \mc{A}_{\mathbb{A}_f,g_f,\bzeta_f}(\tilde{\varphi}_1,\tilde{\varphi}_2)F(\tilde{\varphi}_2) \,\dd\mu_0(\tilde{\varphi}_2)
\]
where $g_f$ is any admissible metric on $(\A_f,J_\C,\bzeta_f)$  and
\begin{equation}\label{def_of_W}
W(f,g):= -\log|f'(0)|-\Big(\mathbf{D}\log \frac{|f'|}{|f|}|_{\T},\log \frac{|f'|}{|f|}|_{\T}\Big)_2-12S_{\rm L}^0(\mathbb{A}_f,g,g_{\mathbb{A}}),
\end{equation}
with the functional $S_{\rm L}^0$  given by \eqref{SL0}.
\end{corollary}
 
The proof of Proposition \ref{compTf} relies on several  lemmas. To simplify the notations, in these lemmas we use the shortcut notations 
$\D,\D_f,\A_f$  for the Riemann surfaces (with boundary parametrisation) $(\D,J_\C,\zeta^\D(e^{{\rm i}\theta})=e^{{\rm i}\theta})$, $(\D_f=f(\D),J_\C,f\circ \zeta^\D)$ and  $(\A_f,J_\C,\bs{\zeta}_f)$
where $J_\C$ is the canonical complex structure on  $\C$; e.g. we write $P_\D \tilde{\varphi},P_{\D_f}\tilde{\varphi},P_{\A_f}\tilde{\bs{\varphi}}$ for the harmonic extensions of $\tilde{\varphi}\in C^\infty(\T)$ and $\tilde{\bs{\varphi}}\in C^{-\infty}(\T)^2$ on $(\D,J_\C,\zeta^\D)$,  $(\D_f=f(\D),J_\C,f\circ \zeta^\D)$ and $(\A_f,J_\C,\bs{\zeta}_f)$, 
we write ${\bf D}_{\A_f}$ for the Dirichlet-to-Neumann operator of $(\A_f,J_\C,\bs{\zeta}_f)$ acting  on  $C^{-\infty}(\T)^2$ and ${\bf D}_{\D,\xi_f}$
the Dirichlet-to-Neumann operator acting on $C^{-\infty}(\T)$ (corresponding to a field on $\mc{C}=f(\T)$) defined as in \eqref{defDSigmaC} with $\xi_f(e^{{\rm i}\theta})=f(e^{{\rm i}\theta})$ the parametrisation of $\mc{C}$.
Our first step is the following lemma proved in \cite[Lemma 5.3]{GKRV21_Segal} (see eq. (5.6) there):
\begin{lemma}\label{LoiX}
If $F$ is a bounded measurable functional on $H^{s}(\T)$ for $s<0$, we have
\begin{equation}\label{densityDN}
\E[F(X_{\D}\circ f |_{\T})]=\frac{1}{ \pi^{1/2}\det(\mathbf{D}_{\D,\xi_f}(2\mathbf{D}_0)^{-1})^{-1/2} }\int F( \tilde{ \varphi})e^{-\frac{1}{2}(\tilde{ \varphi}, (\mathbf{D} _{\D,\xi_f} -2\mathbf{D})\tilde{ \varphi})_2}\dd\mu_0  (\tilde{ \varphi})
\end{equation}
where $\mathbf{D}_0:=\Pi_0+\mathbf{D}$  on $H^s(\T) $ and $ \Pi_0(\tilde\varphi):=  ( \tilde\varphi,1)_2$. Here 
$\mc{C}=f(\T)$.
\end{lemma}

Next we claim:
\begin{lemma}\label{shiftP}
Let $\tilde\varphi \in H^{s}(\T)$ for $s<0$ and define the function $p(e^{{\rm i}\theta}):=  P_\D\tilde\varphi ( f(e^{{\rm i}\theta}))$ on $\T=\pl \D$. Let $F$ be a bounded measurable functional on $H^{s}(\T)$. Then the following holds 
\begin{equation}\label{densityDN-shift}
\E[F((X_\D\circ f+P_\D\tilde\varphi \circ f)|_{\T})]= \E\Big[e^{(\mathbf{D}_{\D,\xi_f}p,X_\D\circ f)_2-\tfrac{1}{2}(\mathbf{D}_{\D,\xi_f}p,p)_2}F(X_\D\circ f|_{\T} )\Big]
\end{equation}
where the expectation is taken with respect to the Dirichlet GFF $X_\D$.
\end{lemma}
\begin{proof} Recall  from \eqref{DNmapandGreen} that the inverse of the DN map $\mathbf{D}_{\D,\xi_f}$ is the Green function $G_{\D,{\rm D}}$ (with Dirichlet condition at $\T=\pl \D$) restricted to the curve $\mc{C}$, i.e. for $u\in C^\infty(\T)$ 
\begin{equation}\label{girsgreen}
(\mathbf{D}_{\D,\xi_f}u,G_{\D,{\rm D}}(f(\cdot),f(e^{{\rm i}\theta}))_2=u(e^{{\rm i}\theta}).
\end{equation}
Our lemma then follows from the Cameron-Martin theorem, by applying this with $u=p$.
 \end{proof}

\begin{lemma}\label{GFFQshift}
Let $g_\D=|\dd z|^2$ be the Euclidean metric and define $\phi_{\D}:=X_\D+P_\D\tilde\varphi_1$ for $\tilde\varphi_1\in H^{s}(\T)$ for $s<0$. If $F$ is a bounded measurable functional on $H^{s}(\T)$, then 
\begin{align*}
&\E_{\varphi_1}\Big[F\big(\big(\phi_{\D}\circ f+Q\log\frac{|f'|}{|f|}\big)\big|_{\T}\big)\Big]\\
&=Ce^{-\frac{Q^2}{2}(\mathbf{D}_{\D,\xi_f}\omega,\omega)_2}\int F( \tilde{ \varphi}_2) 
e^{-\frac{Q}{2\pi}\int_{\mc{C}}k_{g_\mathbb{A}}( \tilde{ \varphi}_2\circ f^{-1}) \,\dd \ell_{g_\mathbb{A}}+Q(\mathbf{D}_{\mathbb{A}_f}(0,\omega),  \tilde{\bs{\varphi}} )_2 -\frac{1}{2}(\tilde{\bs{\varphi}} , (\mathbf{D}_{\mathbb{A}_f} - \mathbf{D})\tilde{\bs{\varphi}})_2}\dd\mu_0  (\tilde{ \varphi}_2) 
\end{align*}
where  $C:=\pi^{-1/2}\det(\mathbf{D}_{\D,\xi_f}(2\mathbf{D}_0)^{-1})^{1/2}$ and $\tilde{ \bs{\varphi}}=(\tilde\varphi_1,\tilde\varphi_2)$.
\end{lemma}

\begin{proof}  
Recall that $h$ is defined by \eqref{defh_t} and $\omega(e^{{\rm i}\theta})=\log \frac{|f'(e^{{\rm i}\theta})|}{|f(e^{{\rm i}\theta})|}$. First we need to compute the geodesic curvature of the  Riemannian  manifold $(\mathbb{A}_f,g_{\mathbb{A}})$ along its boundary. We claim
\begin{equation}\label{curvature}
k_{g_{\mathbb{A}}}\phantom{}_{|\partial \D}=0\quad \text{ and }\quad  k_{g_{\mathbb{A}}}\circ f(e^{{\rm i}\theta})=-\frac{|f(e^{{\rm i}\theta})| } {|f'(e^{{\rm i}\theta})| }  \partial_{r}(h\circ f(re^{{\rm i}\theta}))|_{r=1}.
\end{equation}
Indeed, using \eqref{f_t*g_A} and the relation $k_{e^{2u}g_{\mathbb{A}}}=e^{-u}(k_{g_{\mathbb{A}}}-\partial_\nu u)$ if $u\in C^\infty(\Sigma)$ 
with $\nu$ the inward normal vector with respect to $g_{\mathbb{A}}$, we get \eqref{curvature}.

Next, using \eqref{girsgreen} and the Cameron-Martin theorem, we get
\begin{align*}
\E_{\varphi_1}[F((\phi_{\D}\circ f+Q\log\frac{|f'|}{|f|})|_{\T})]
&=
\E_{\varphi_1}\Big[F((X_\D\circ f+P_\D\tilde\varphi_1\circ f)|_{\T})e^{ Q(\mathbf{D}_{\D,\xi_f}\omega,X_\D\circ f)_2}\Big]e^{-\frac{Q^2}{2}(\mathbf{D}_{\D,\xi_f}\omega,\omega)_2}.
\end{align*}
Now observe that
\begin{equation}\label{identitykg_A}
\mathbf{D}_{\D,\xi_f}\omega= -k_{g_\mathbb{A}}\circ f-\partial_\nu P_{\mathbb{A}_f}(0,\omega)|_{\mc{C}}\circ f
=-k_{g_\mathbb{A}}\circ f+\pi_2(\mathbf{D}_{\mathbb{A}_f}(0,\omega))
\end{equation}
where $\pi_2:C^\infty(\T)\times C^\infty(\T)\to C^\infty(\T)$ is the projection on the second component.
Indeed this follows from $(-\partial_\nu P_{\D_f}\omega)\circ f=-k_{g_\mathbb{A}}\circ f$ ($\nu$ is the inward pointing unit vector on $\pl \D_f)$, which in turn is a consequence of \eqref{curvature} and the fact that $P_{\D_f}\omega=h$. Therefore
\begin{align*}
\E_{\varphi_1}&[F((\phi_{\D}\circ f+Q\log\frac{|f'|}{|f|})|_{\T})]\\
=&\E_{\varphi_1}\Big[F((X_\D\circ f+P_\D\tilde\varphi_1\circ f)|_{\T})e^{-\frac{Q}{2\pi}\int_{\mc{C}}k_{g_\mathbb{A}}X_\D\,\dd \ell_{g_\mathbb{A}} +Q(\mathbf{D}_{\mathbb{A}_f}(0,\omega) ,(0, X_\D\circ f))_2}\Big] e^{-\frac{Q^2}{2}(\mathbf{D}_{\D,\xi_f}\omega,\omega)_2}\\
=&\E_{\varphi_1}\Big[F(X_\D\circ f|_{\T})e^{-\frac{Q}{2\pi}\int_{\mc{C}}k_{g_\mathbb{A}}(X_\D-P_\D\tilde\varphi_1)\,\dd \ell_{g_\mathbb{A}} +Q(\mathbf{D}_{\mathbb{A}_f}(0,\omega),(0, X_\D\circ f))_2}
e^{(\mathbf{D}_{\D,\xi_f}p,X_\D\circ f)_2-\tfrac{1}{2}(\mathbf{D}_{\D,\xi_f}p,p)_2}\Big] 
\\&\times e^{-\frac{Q^2}{2}(\mathbf{D}_{\D,\xi_f}\omega,\omega)_2-Q(\mathbf{D}_{\mathbb{A}_f}(0,\omega) ,(0, p))_2}
\end{align*}
where we have used Lemma \ref{shiftP} in the last equality to shift the term $p=P_\D\tilde\varphi_1\circ f$. 
Finally we can use Lemma \ref{LoiX} to obtain
\begin{align}
\label{computation1}
 &\E_\varphi  [F((\phi_{\D}\circ f+Q\log\frac{|f'|}{|f|})|_\T)]=  C e^{-\frac{Q^2}{2}(\mathbf{D}_{\D,\xi_f}\omega,\omega)_2-Q(\mathbf{D}_{\mathbb{A}_f}(0,\omega) ,(0, p))_2}e^{\frac{Q}{2\pi}\int_{\mc{C}}k_{g_\mathbb{A}}P_\D\tilde\varphi_1\,\dd \ell_{g_\mathbb{A}}}\\
 & \times \int F(\tilde{ \varphi}_2)e^{-\frac{Q}{2\pi}\int_{\mc{C}}k_{g_\mathbb{A}}\tilde{\varphi}_2\circ f^{-1}\,\dd \ell_{g_\mathbb{A}} +Q(\mathbf{D}_{\mathbb{A}_f}(0,\omega) ,(0, \tilde{ \varphi}_2))_2}  e^{(\mathbf{D}_{\D,\xi_f}p,\tilde{ \varphi}_2)_2-\frac{1}{2}(\mathbf{D}_{\D,\xi_f}p,p)_2-\frac{1}{2}(\tilde{ \varphi}_2, (\mathbf{D} _{\D,\xi_f} -2\mathbf{D})\tilde{ \varphi}_2)_2}\dd\mu_0  (\tilde{ \varphi}_2)\nonumber
\end{align}
where $C=\pi^{-1/2}\det(\mathbf{D}_{\D,\xi_f}(2\mathbf{D}_0)^{-1})^{1/2} $.
We need to simplify this expression. For this, we need the following lemma:
\begin{lemma}\label{colin}
For $\tilde{\varphi}_1,\tilde{\varphi}_2\in C^\infty(\T)$ and $p:=P_{\D}\tilde{\varphi}_1\circ f\in C^\infty(\T)$, one has 
\begin{align*}
( {\bf D}_{\D,\xi_f}(\tilde{\varphi}_2-p),(\tilde{\varphi}_2-p))_2=&({\bf D}_{\mathbb{A}_f}(\tilde{\varphi}_1,\tilde{\varphi}_2),(\tilde{\varphi}_1,\tilde{\varphi}_2))_2 -( {\bf D}\tilde{\varphi}_1,\tilde{\varphi}_1)_2
+( {\bf D}\tilde{\varphi}_2,\tilde{\varphi}_2)_2  \\
(\mathbf{D}_{\D,\xi_f}\omega,p)_2=&-(\mathbf{D}_{\mathbb{A}_f}(0,\omega), (\tilde{\varphi}_1,0))_2.
\end{align*}
\end{lemma}

\begin{proof}
Recall that the DN map as quadratic form is conformally invariant and equal to the Dirichlet energy of the harmonic extension by \eqref{Greenformula} and Lemma \ref{DNmapInv}. 
For the second identity in the Lemma, we use the definition of the DN maps, Green's formula twice and $\Delta_{\D} P_{\D}\tilde{\varphi}_1=0$ in $\D$  (thus $P_{\D}\tilde{\varphi}_1$ is smooth in $\D$) to get 
\[\begin{split}
(\mathbf{D}_{\D,\xi_f}\omega,p)_2+(\mathbf{D}_{\mathbb{A}_f}(0,\omega), (\tilde{\varphi}_1,0))_2=& 
\int_{\mathbb{A}_f}\nabla P_{\mathbb{A}_f}(0,\omega).\nabla P_{\D}\tilde{\varphi}_1 \dd x+\int_{\D_f}\nabla P_{\D_f}\omega\cdot \nabla 
P_{\D}\tilde{\varphi}_1 \dd x=0.
\end{split}\]
For the first identity, we use
\begin{equation}\label{usefulidentity} 
P_{\mathbb{A}_f}(0,p)=P_{\D}\tilde{\varphi}_1-P_{\mathbb{A}_f}(\tilde{\varphi}_1,0), \quad P_{\D_f}p=P_{\D}\tilde{\varphi}_1.
\end{equation}
Then,
\[ ( {\bf D}_{\D,\xi_f}(\tilde{\varphi}_2-p),(\tilde{\varphi}_2-p))_2=\int_{\mathbb{A}_f}|\nabla P_{\mathbb{A}_f}(0,\tilde{\varphi}_2-p)|^2\dd x +\int_{\mathbb{D}_f}|\nabla P_{\mathbb{D}_f}(\tilde{\varphi}_2-p)|^2\dd x .\]
We have, using \eqref{usefulidentity}, that for $\tilde{\bs{ \varphi}}=(\tilde{\varphi}_1,\tilde{\varphi}_2)$
\[\begin{split}
\int_{\mathbb{A}_f}|\nabla P_{\mathbb{A}_f}(0,\tilde{\varphi}_2-p)|^2\dd x =&\int_{\mathbb{A}_f}|\nabla P_{\mathbb{A}_f}(0,\tilde{\varphi}_2)|^2\dd x +\int_{\mathbb{A}_f}
|\nabla P_{\mathbb{A}_f}(0,p)|^2 \dd x -2\int_{\mathbb{A}_f}\nabla P_{\mathbb{A}_f}(0,\tilde{\varphi}_2)\cdot\nabla P_{\mathbb{A}_f}(0,p)\dd x \\
=&
 \int_{\mathbb{A}_f}|\nabla P_{\mathbb{A}_f}(0,\tilde{\varphi}_2)|^2\dd x +\int_{\mathbb{A}_f}|\nabla P_{\D}\tilde{\varphi}_1|^2\dd x +\int_{\mathbb{A}_f}|\nabla P_{\mathbb{A}_f}(\tilde{\varphi}_1,0)|^2 \dd x\\
 &  -2\int_{\mathbb{A}_f}\nabla P_{\mathbb{A}_f}(0,\tilde{\varphi}_2)\cdot \nabla P_{\D}\tilde{\varphi}_1\dd x  -2\int_{\mathbb{A}_f}\nabla P_{\D}\tilde{\varphi}_1\cdot \nabla P_{\mathbb{A}_f}(\tilde{\varphi}_1,0)\dd x \\
 & +2\int_{\mathbb{A}_f}\nabla P_{\mathbb{A}_f}(0,\tilde{\varphi}_2)\cdot\nabla P_{\mathbb{A}_f}(\tilde{\varphi}_1,0)\dd x \\
=&
({\bf D}_{\mathbb{A}_f}\tilde{\bs{ \varphi}},\tilde{\bs{ \varphi}})_2+\int_{\mathbb{A}_f}|\nabla P_{\D}\tilde{\varphi}_1|^2\dd x -2\int_{\mathbb{A}_f}\nabla P_{\D}\tilde{\varphi}_1\cdot \nabla P_{\mathbb{A}_f}\tilde{\bs{ \varphi}}\, \dd x .
\end{split}\]
Next, we compute 
\[\int_{\mathbb{D}_f}|\nabla P_{\mathbb{D}_f}(\tilde{\varphi}_2-p)|^2\dd x = \int_{\D_f}|\nabla P_{\D}\tilde{\varphi}_1|^2\dd x+ \int_{\D_f}|\nabla P_{\D_f} \tilde{\varphi}_2|^2\dd x -2\int_{\D_f}\nabla P_{\D}\tilde{\varphi}_1\cdot \nabla P_{\D_f}\tilde{\varphi}_2 \,\dd x.\]
Therefore, gathering all terms and using that $P_{\D_f}\tilde{\varphi}_2+P_{\mathbb{A}_f}\tilde{\bs{ \varphi}}=P_{\D}\tilde{\varphi}_1$ in $\D$,
\[\begin{split}
( {\bf D}_{\D,\xi_f}(\tilde{\varphi}_2-p),(\tilde{\varphi}_2-p))_2
=&
({\bf D}_{\mathbb{A}_f}\tilde{\bs{ \varphi}},\tilde{\bs{ \varphi}})_2+\int_{\D}|\nabla P_{\D}\tilde{\varphi}_1|^2\dd x
-2\int_{\mathbb{A}_f}|\nabla P_{\D}\tilde{\varphi}_1|^2\dd x + \int_{\D_f}|\nabla P_{\D_f}\tilde{\varphi}_2|^2\dd x \\
=& ({\bf D}_{\mathbb{A}_f}\tilde{\bs{ \varphi}},\tilde{\bs{ \varphi}})_2 -\int_{\D}|\nabla P_{\D}\tilde{\varphi}_1|^2\dd x + \int_{\D_f}|\nabla P_{\D_f}\tilde{\varphi}_2|^2\dd x \\
=&  ({\bf D}_{\mathbb{A}_f}\tilde{\bs{ \varphi}},\tilde{\bs{ \varphi}})_2 -( {\bf D}\tilde{\varphi}_1,\tilde{\varphi}_1)_2+ \int_{\D_f}|\nabla P_{\D_f}\tilde{\varphi}_2|^2\dd x 
\end{split}\]
where ${\bf D}$ is the DN map of $\D$.

Now, we can use Lemma \ref{DNmapInv} and the expression \eqref{Greenformula} of the DN map in terms of Dirichlet energy to deduce that 
\[\int_{\D_f}|\nabla P_{\D_f}\tilde{\varphi}_2|^2\dd x:=( {\bf D}\tilde{\varphi}_2,\tilde{\varphi}_2)_2.\]
This ends the proof of Lemma \ref{colin}.
\end{proof}
 
Using the first claim of Lemma \ref{colin}, we deduce from \eqref{computation1} that
\begin{equation}\label{computation2}
\begin{split}
\E_{\varphi_1} \Big[F\big(\big(\phi_{\D}\circ f+Q\log\frac{|f'|}{|f|}\big)\big|_{\T}\big)\Big]=& C \int F(\tilde{ \varphi}_2)e^{-\frac{Q}{2\pi}\int_{\mc{C}}k_{g_\mathbb{A}}(\tilde{ \varphi}_2\circ f^{-1}-P_\D\tilde\varphi_1)\,\dd \ell_{g_\mathbb{A}} +Q(\mathbf{D}_{\mathbb{A}_f}(0,\omega) ,(0, \tilde{ \varphi}_2))_2}\\
&\quad \times  e^{-\frac{Q^2}{2}(\mathbf{D}_{\D,\xi_f}\omega,\omega)_2-Q(\mathbf{D}_{\mathbb{A}_f}(0,\omega) ,(0, p))_2}  e^{-\frac{1}{2}(\tilde{ \bs{\varphi}} , (\mathbf{D} _{\mathbb{A}_f} - \mathbf{D})\tilde{ \bs{\varphi}})_2}\dd\mu_0  (\tilde{ \varphi}_2) .
\end{split}
\end{equation}
Also, we can use \eqref{identitykg_A} and the second identity of Lemma \ref{colin} to get the relation
$$-\frac{1}{2\pi}\int_{\mc{C}}k_{g_\mathbb{A}}P_\D\tilde\varphi_1\,\dd \ell_{g_\mathbb{A}} +(\mathbf{D}_{\mathbb{A}_f}(0,\omega), (0, p) )_2 =(\mathbf{D}_{\D,\xi_f}\omega,p)_2=-(\mathbf{D}_{\mathbb{A}_f}(0,\omega), (\tilde{\varphi}_1,0))_2$$
which we can plug in \eqref{computation2} to obtain
\begin{align*}
\E_{\varphi_1}&\Big[F\big(\big(\phi_{\D}\circ f+Q\log\frac{|f'|}{|f|}\big)\big|_{\T}\big)\Big]\\
=&C e^{-\frac{Q^2}{2}(\mathbf{D}_{\D,\xi_f}\omega,\omega)_2}\int F(\tilde{ \varphi}_2)e^{-\frac{Q}{2\pi}\int_{\mc{C}}k_{g_\mathbb{A}}\tilde{ \varphi}_2\circ f^{-1}\,\dd \ell_{g_\mathbb{A}} +Q(\mathbf{D}_{\mathbb{A}_f}(0,\omega),\tilde{ \bs{\varphi}})_2 -\frac{1}{2}(\tilde{ \bs{\varphi}} , (\mathbf{D} _{\mathbb{A}_f} - \mathbf{D})\tilde{ \bs{\varphi}})_2}\dd\mu_0  (\tilde{ \varphi}_2) .
\end{align*}
This completes the proof of Lemma \ref{GFFQshift}.
\end{proof}

\begin{proof}[Proof of Proposition \ref{compTf}]
By the Markov property of the GFF, the law of the field $X_\D$ in $\A_f$, conditionally on the boundary values $X_\D\circ f=\tilde{\varphi}_2$, is $X_{\A_f}+P_{\A_f} (0,\tilde{\varphi}_2)$. Therefore, for bounded measurable functional $G$, 
$$\E_{ \tilde{ \bs{\varphi}}}[G(X_\D+P_\D\tilde{\varphi}_1)]=\E_{ \tilde{ \bs{\varphi}}}[G(X_{\A_f}+P_{\A_f} \tilde{ \bs{\varphi}}+P_{\A_f} (0,p))]$$ where we have used the relation $P_\D\tilde{\varphi}_1=P_{\A_f} ( \tilde{\varphi}_1,p)$ and the notation $\E_{ \tilde{ \bs{\varphi}}}$ for expectation of the GFF conditionally on its boundary values $\tilde{ \bs{\varphi}}$. Then we can use this relation together with Lemma \ref{GFFQshift}  to get, for bounded measurable functionals $F,G$  with $G$ depending only on the field on $\A_f$ (expectation is over the Dirichlet GFF on $\D$),
\begin{align*}
&\E \Big[F\big(\big(\phi_{\D}\circ f+Q\log\frac{|f'|}{|f|}\big)\big|_{\T}\big)G(\phi_{\D})\Big]\\
&=
Ce^{-\frac{Q^2}{2}(\mathbf{D}_{\D,\xi_f}\omega,\omega)_2}\int F( \tilde{ \varphi}_2)H( \tilde{\bs{\varphi}}) e^{-\frac{Q}{2\pi}\int_{\mc{C}}k_{g_\mathbb{A}}( \tilde{ \varphi}_2\circ f^{-1} )\,\dd \ell_{g_\mathbb{A}}+Q(\mathbf{D}_{\A_f}(0,\omega),  \tilde{\bs{ \varphi}} )_2 -\frac{1}{2}(\tilde{ \bs{\varphi}} , (\mathbf{D} _{\A_f} - \mathbf{D})\tilde{ \bs{\varphi}})_2}\dd\mu_0  (\tilde{ \varphi}_2) 
\end{align*}
with $\phi_{\D}=X_\D+P_\D\tilde\varphi_1$ as  above and 
\[H(\tilde{\bs{\varphi}})=\E[G(X_{\A_f}+P_{\A_f}  \tilde{ \bs{\varphi}} -QP_{\A_f}\bs{\omega})],\]
where the expectation is over the Dirichlet GFF on $\A_f$ and $\bs{\omega}=(0,\omega)$.
In particular, it follows that  
\begin{align*}
&\E_{\varphi_1}\Big[F\big(\big(\phi_{\D}\circ f+Q\log\frac{|f'|}{|f|}\big)\big|_{\T}\big)\exp\big(-\mu \int_{\A_f}|x|^{-\gamma Q}M_\gamma(\phi_{\D}, \dd x)\big)\Big]\\
&=
Ce^{-\frac{Q^2}{2}(\mathbf{D}_{\D,\xi_f}\omega,\omega)_2}\int F( \tilde{ \varphi}_2)H( \tilde{\bs{\varphi}}) e^{-\frac{Q}{2\pi}\int_{\mc{C}}k_{g_\mathbb{A}}( \tilde{ \varphi}_2\circ  f^{-1} )\,\dd \ell_{g_\mathbb{A}}+Q(\mathbf{D}_{\A_f}\bs{\omega} ,  \tilde{\bs{ \varphi}} )_2 -\frac{1}{2}(\tilde{ \bs{\varphi}} , (\mathbf{D} _{\A_f} - \mathbf{D})\tilde{ \bs{\varphi}})_2}\dd\mu_0  (\tilde{ \varphi}_2) \\
&= Ce^{J(\omega)}\int F( \tilde{ \varphi}_2)H( \tilde{\bs{\varphi}}) e^{-\frac{Q}{2\pi}\int_{\mc{C}}k_{g_\mathbb{A}}((\tilde{ \varphi}_2-Q\omega)\circ  f^{-1} )\,\dd \ell_{g_\mathbb{A}} +Q(\mathbf{D}\bs{\omega} ,  \tilde{\bs{ \varphi}} )_2-\frac{1}{2}(\tilde{ \bs{\varphi}}-Q\bs{\omega} , (\mathbf{D} _{\A_f} - \mathbf{D})(\tilde{ \bs{\varphi}}-Q\bs{\omega}))_2}\dd\mu_0  (\tilde{ \varphi}_2),
\end{align*}
with (below $\phi_{\A_f}=X_{\A_f}+P_{\A_f}\tilde{\bs{\varphi}}$ and we use \eqref{scalingmeasure})
\[
H( \tilde{\bs{\varphi}})=\E_{ \tilde{ \bs{\varphi}}}\big[e^{-\mu \int_{\A_f}|x|^{-\gamma Q}M_\gamma(\phi_{\A_f}-QP_{\A_f}\bs{\omega},\dd x)}\big] =\E_{ \tilde{ \bs{\varphi}}}[e^{-\mu M^{g_\A}_\gamma(\phi_{\A_f}-QP_{\A_f}\bs{\omega},\A_f)}] ,\] 
and (using \eqref{identitykg_A} in the last identity below)
\[\begin{split}
J(\omega)= &-\frac{Q^2}{2\pi}\int_{\mc{C}}k_{g_\A}\omega\circ f^{-1} d\ell_{g_\A}-\frac{Q^2}{2}((\mathbf{D}_{\D,\xi_f}+{\bf D})\omega,\omega)_2+\frac{Q^2}{2}(\mathbf{D} _{\mathbb{A}_f}\bs{\omega},\bs{\omega})_2\\
 =& -\frac{Q^2}{2}({\bf D}\omega,\omega)_2-\frac{Q^2}{4\pi}\int_{\mc{C}}k_{g_\A}\omega\circ f^{-1} \dd\ell_{g_\A}.
\end{split}\]
 Let us finally compute (using \eqref{curvature})
 \[ \int_{\mc{C}}k_{g_\A}\omega\circ f^{-1} d\ell_{g_\A}=-\int_0^{2\pi} \pl_r \Big(\log|r\frac{f'(re^{{\rm i}\theta})}{f(re^{{\rm i}\theta})}|\Big) 
\log|\frac{f'(re^{{\rm i}\theta})}{f(re^{{\rm i}\theta})}| d\theta =-2\pi ({\bf D}\omega,\omega)_2\]
which implies that $J(\omega)=0$.    
Finally we can use \cite[Lemma 5.4]{GKRV21_Segal} to obtain that
\[ C=\pi^{-1/2}\det(\mathbf{D}_{\D,\xi_f}(2\mathbf{D}_0)^{-1})^{1/2}=\frac{1}{\sqrt{2}\pi}\Big(
\frac{\det(\Delta_{\D,{\rm D}})}{\det(\Delta_{\A_f,{\rm D}})\det(\Delta_{\D_f,{\rm D}})}\Big)^{1/2}\]  
where $\Delta_{\A_f,{\rm D}}, \Delta_{\D_f,{\rm D}}$ and $\Delta_{\D,{\rm D}}$ are the Laplacians for the flat metric $g_\D=|\dd z|^2$ with Dirichlet conditions.
Using \cite[Eq (1.17)]{OsgoodPS88} with $\int_{\mc{C}}k_{g_\A}=2\pi (\mathbf{D}\omega,1)_2=0$ and $f^*g_\D=|f'(z)|^2|\dd z|^2$, it follows that (recall definition \eqref{znormal}):
\begin{align*}
&\Big(\frac{\det(\Delta_{\D,{\rm D}})}{\det(\Delta_{\D_f,{\rm D}})}\Big)^{1/2}= e^{\frac{1}{24\pi} \int_{\D}|\nabla \log |f'||^2 +\frac{1}{12\pi}\int_{\T}\log|f'| }= e^{\frac{1}{24\pi} \int_{\D}|\nabla \log |f'||^2\dd x +\frac{1}{6}\log|f'(0)|},\\
& \det(\Delta_{\A_f,{\rm D}})^{-1/2} =Z_{\A_f,g_\A}e^{\frac{1}{24\pi}\int_{\mathbb{A}_f}|\nabla \log |x||^2 \dd x+\frac{1}{12\pi}\int_{\mc{C}}k_{g_\A}\log |x|\dd\ell_{g_\A}+\frac{1}{4}-\frac{1}{8\pi}\int_{\mc{C}}\pl_{\nu}\log|x|\dd\ell_{g_\A}} 
\end{align*}
Notice that, by integration by parts and using that $f(z)/z$ is holomorphic in $\D$,
\[ \begin{gathered}
 \frac{1}{24\pi}\int_{\D}|\nabla \log |f'(x)||^2\,\dd x = \frac{1}{24\pi}\int_\T \log |f'(e^{{\rm i}\theta})|\mathbf{D}(\log |f'(e^{{\rm i}\theta})|)\dd \theta,\\ 
 \frac{1}{4\pi}\int_{\mc{C}}(\frac{k_{g_\A}}{3}\log |x|-\frac{1}{2}\pl_{\nu}\log|x|)\dd\ell_{g_\A} =
-\frac{1}{4\pi}\int_\T \Big(\frac{1}{3}{\bf D} \omega(e^{{\rm i}\theta})\log |f(e^{{\rm i}\theta})|\Big)\dd \theta-\frac{1}{4}.
\end{gathered}\]
and 
\[\begin{split}
\frac{1}{24\pi}\int_{\mathbb{A}_f}|\nabla \log |x||^2 \dd x= &-\frac{1}{24\pi}\int_{\mc{C}}(\pl_{\nu}\log|x|)\log|x|\dd \ell_{g_\D} \\
= &-\frac{1}{24\pi}\int_{\T}(\log |f(e^{{\rm i}\theta})|{\bf D}\log |f(e^{{\rm i}\theta})|+\log |f(e^{{\rm i}\theta})|)\dd\theta ,
\end{split}\]
Moreover, using that the DN map  ${\bf D}$ on $\D$ is self-adjoint and that $\omega(e^{{\rm i}\theta})=\log |f'(e^{{\rm i}\theta})/f(e^{{\rm i}\theta})|$, the sum of the first two lines is equal to 
\[   \frac{1}{24\pi}\int_\T  \log |f'(e^{{\rm i}\theta})f(e^{{\rm i}\theta})|{\bf D} \omega(e^{{\rm i}\theta})\dd\theta -\frac{1}{12}\log |f'(0)|.\]
Summing everything, we have shown the formula 
\[C=\frac{1}{\sqrt{2}\pi}e^{\frac{1}{12}\log|f'(0)|+\frac{1}{12}({\bf D}\omega,\omega)_2}Z_{\A_f,g_\A}.\]
Gathering the above computations, we deduce the formula \eqref{prop_amplitudeA_t^v}, and the proof of Proposition  \ref{compTf} is complete.
\end{proof}

 \subsection{Regularity properties of the quantised Segal semigroup}
 
The connection of the operators ${\bf T}_f$ with amplitudes will be instrumental  in obtaining  regularity estimates for the operators ${\bf T}_f$. To begin with, we claim:
\begin{proposition}\label{boundTf}
For $f\in \mc{S}_>$, the operator ${\bf T}_f$ extends to a bounded operator on $\mc{H}$.  
\end{proposition}  

\begin{proof}
In \cite[Theorem 4.4]{GKRV21_Segal}, the following bound is proved:  $\exists a>0$, $\forall R>1$, $\exists C_R>0$ so that,  writing  $\tilde{\varphi}_1= c_1+\varphi_1$ and $\tilde{\varphi}_2= c_2+\varphi_2$ (using the convention \eqref{param})  
\[ 
\mc{A}_{\mathbb{A}_f,g_f,\bzeta_f}(\tilde{\varphi}_1,\tilde{\varphi}_2) \leq C_Re^{-R(c_1+c_2)_+-a(c_1-c_2)^2}A(\varphi_1,\varphi_2),\quad \textrm{ with }A \in L^2(\Omega_\T^2)
\]
where $c_+=\max(c,0)$. Using Cauchy-Schwarz, we thus write for $F\in L^2(\R\times \Omega_\T)$ 
 \begin{equation} \label{CS_inequality}
|{\bf T}_fF(c_1,\varphi_1)|\leq C_R\int_{\R}K(c_1-c_2) \Big(\int_{\Omega_\T}|F(c_2,\varphi_2)|^2\dd\mathbb{P}_\T(\varphi_2)\Big)^{1/2}\Big(\int_{\Omega_\T}|A(\varphi_1,\varphi_2)|^2\dd \mathbb{P}_\T(\varphi_2)\Big)^{1/2}\dd c_2
\end{equation}
where $K(c)=e^{-ac^2}$. Using that the convolution by $K$ is a bounded operator on $L^2(\R)$, there exists $C>0$ such that
\[\|{\bf T}_fF\|_{\mc{H}}\leq C\|A\|_{L^2(\Omega_\T^2)}\|F\|_{\mc{H}}.\qedhere\]
\end{proof}

 Given two topological vector spaces $E,F$, we use the standard notation $\mc{L}(E,F)$ for the space of continuous linear operators $E\to F$.

\begin{proposition}\label{contTf}
For each $\epsilon>0$, the mapping $f\in \mc{S}_\epsilon\mapsto  \mathbf{T}_f\in\mc{L}(\mc{H},\mc{H})$ is continuous. 
\end{proposition}

\begin{proof}
By Proposition \ref{boundTf}, for $f\in \mc{S}_>$,  $ \mathbf{T}_f$ extends to a bounded operator on $\mathcal{H}$. 	
Corollary \ref{cor:annulus_propagator} shows that the continuity of the map follows from the continuity of the operator with  integral kernel the amplitude associated to $ \mathbf{T}_f$, since  the prefactor $W(f,g_f)$ defined in \eqref{def_of_W}  is continuous in $f$. So, we are going to prove the continuity of the operator with  integral kernel given by the amplitude associated to $ \mathbf{T}_f$. Recall that the amplitude is given by the expression 
\begin{equation*}
\mc{A}_{\mathbb{A}_f,g_f,\bzeta_f}(\tilde{\varphi}_1,\tilde{\varphi}_2)= C(f) e^{ -\frac{1}{2}(\tilde{ \bs{\varphi}} , (\mathbf{D} _{\A_f} - \mathbf{D})\tilde{ \bs{\varphi}})_2}  \E_{\tilde{\bs{\varphi}}} \Big  [ e^{-\frac{Q}{4\pi}\int_{\A_f}K_{g_f}(X_{\A_f,{\rm D}}+P\tilde{ \bs{\varphi}}){\rm dv}_{g_f}} e^{- \mu M_\gamma^{g_f} (X_{\A_f,{\rm D}}+P\tilde{ \bs{\varphi}},  \A_f)   }  \Big]
\end{equation*}
where $\tilde{\bs{\varphi}}= (\tilde \varphi_1, \tilde \varphi_2)=  (c_1+\varphi_1,c_2+\varphi_2)$,   $C(f)=\det(\Delta_{\A_f,g_f,{\rm D}})^{-1/2}$   and $X_{\A_f,{\rm D}}$ is the Dirichlet GFF on  $ \A_f$ equipped with the canonical complex structure $J_\C$. The map $f\mapsto C(f)$ is continuous (indeed differentiable): indeed, if $\Sigma$ is a fixed surface with boundary, it is standard (see \cite{Ray-Singer}) that the determinant of Laplacian $g\mapsto \det \Delta_{\Sigma, g,{\rm D}}$ is continuous as a function of the metric $g$ on $\Sigma$, thus by choosing a diffeomorphism $\psi_f:\A_{f_0}\to \A_f$ depending continuously on $f$ one can use that $\det \Delta_{\A_f,g_f,{\rm D}}=\det \Delta_{\A_{f_0},\psi_f^*g_f,{\rm D}}$ to deduce the continuity of $f\mapsto C(f)$.
Let us consider a   sequence $(f_n)_n$ that converges to $f$ in  $\mc{S}_\epsilon$. In the sequel and to simplify notation, we simply denote 
$(\Sigma_n,g_n,\bzeta_n):=(\A_{f_n},g_{f_n},\bzeta_{f_n})$ and $(\Sigma,g,\bzeta):=(\A_{f},g_f,\bzeta_f)$, where $\bzeta_{f_n},\bzeta_{f}$ are the parametrisations of respectively $\A_{f_n}$ and $\A_f$ induced by $({\rm Id},f_n)$ and $({\rm Id},f)$. We also let $S_n:=(\Sigma_n,J_\C,\bs{\zeta}_n)$ and $S:=(\Sigma,J_\C,\bzeta)$ where $J_\C$ is the canonical complex structure on $\C$, let $X_n$ and $X$ denote the Dirichlet GFFs on $\Sigma_{n}$ and $\Sigma$. Similarly let $P_n \tilde{\bs{\varphi}}:=P_{S_n} \tilde{\bs{\varphi}}$ and $P \tilde{\bs{\varphi}}:=P_S \tilde{\bs{\varphi}}$ denote the harmonic extensions on $\Sigma_n$ and $\Sigma$ of the boundary fields $\tilde{\bs{\varphi}}$ (recall \eqref{def:harmonic_extension}).
The corresponding measures $ {\rm dv}_{g_n}$ and $ {\rm dv}_{g}$ are denoted  ${\rm dv}_{n}$ and $ {\rm dv}$ and we finally denote 
 $K_n=K_{g_n}$ and $K=K_g$ the scalar curvatures of $g_n$ and $g$.

Let $\eps>0$. We consider the set $\Sigma^{\eps}$ defined as the subset of $\Sigma$ at Euclidean distance $\geq \eps$ from the boundary (recall that the boundary is $\T\cup f(\T)$) of $\Sigma$. For $n$ large enough, we have $\Sigma^{\eps} \subset \Sigma_n$. We will show that the amplitudes can be split into two parts, one coming from the behaviour near the boundary and the other corresponding to the contribution from the interior, the interior/boundary distinction being those points at distance larger/less than $\eps$ from the boundary.  We will show that the operator norm of the near boundary contribution is smaller than $\eps$, uniformly   in $n$, whereas the bulk contribution converges as $n$ tends to $\infty$. For this,   we decompose the amplitudes as the sum $\mc{A}_{\Sigma_n, g_n,\bzeta_n}( \tilde{\bs{\varphi}})= A^{(1)}_{n,\eps}+A^{(2)}_{n,\eps}$  where
\begin{equation*}
 A^{(1)}_{n,\eps}= C(f_n)e^{ -\frac{1}{2}(\tilde{ \bs{\varphi}} , (\mathbf{D} _{S_n} - \mathbf{D})\tilde{ \bs{\varphi}})_2}   \E_{ \tilde{\bs{\varphi}}} \left  [  e^{-\frac{Q}{4\pi}\int_{\Sigma_{n}}K_{n} (X_{n}+P_n\tilde{ \bs{\varphi}} ){\rm dv}_{n} }   e^{-  \mu M_{\gamma}^{g_n} (X_{n}+P_n\tilde{ \bs{\varphi}}, \Sigma^{\eps})   }    \right]
\end{equation*}
and 
\begin{align*}
 A^{(2)}_{n,\eps}=& C(f_n)e^{ -\frac{1}{2}(\tilde{ \bs{\varphi}} , (\mathbf{D} _{S_n} - \mathbf{D})\tilde{ \bs{\varphi}})_2}  
 \E_{ \tilde{\bs{\varphi}} } \left  [   e^{-\frac{Q}{4\pi}\int_{\Sigma_{n}}K_{n} (X_{n}+P_n\tilde{ \bs{\varphi}}){\rm dv}_{n} } \big( e^{-  \mu M_{\gamma}^{g_n} (X_{n}+P_n\tilde{ \bs{\varphi}} ,  \Sigma_{n})   } -  e^{- \mu  M_{\gamma}^{g_n} (X_{n}+P_n\tilde{ \bs{\varphi}}, \Sigma^{\eps})   }  \big)  \right].
\end{align*}
We first treat the term $ A^{(2)}_{n,\eps}$. Using the inequality  $1-e^{- x} \leq   (x^{\eta}\wedge 1)$ for $\eta\in (0,1)$, we deduce that
\begin{align}\label{boundA2n}
| A^{(2)}_{n,\eps}| \leq  & e^{ -\frac{1}{2}(\tilde{ \bs{\varphi}} , (\mathbf{D} _{S_n} - \mathbf{D})\tilde{ \bs{\varphi}} )_2}  
\\
&\times \E_{ \tilde{\bs{\varphi}}} \left  [  e^{-\frac{Q}{4\pi}\int_{\Sigma_n}K_{n} (X_{n}+P_n\tilde{ \bs{\varphi}}){\rm dv}_{n} } \big(\big(\mu M_{\gamma }^{g_n} (X_{n}+P_n\tilde{ \bs{\varphi}} ,  \Sigma_{n} \setminus \Sigma^{\eps})\big)^\eta\wedge 1\big)   e^{-  \mu M_{\gamma }^{g_n} (X_{n}+P_n\tilde{ \bs{\varphi}} , \Sigma^{\eps})   }   \right] .\nonumber
\end{align}
We will show that $\iint | A^{(2)}_{n,\eps}|^2 (\dd \mu_0)^{\otimes 2}\leq C\eps^{1/p}$ (for some $p>1$ and some irrelevant constant $C>0$) so that the operator with integral kernel $A^{(2)}_{n,\eps}$ has operator norm $\mathcal{H}\to\mathcal{H}$ less than $C\eps^{1/(2p)}$. Let us denote by $\Sigma_n'$  the double of $\Sigma_{n}$ obtained by gluing $(\Sigma_n,g_n,\bzeta_n)$ with 
$(\Sigma_n,g_n,\bzeta_n\circ o)$ where $o(z)=1/z$, where the incoming boundary of $(\Sigma_n,g_n,\bzeta_n)$ is glued to the outgoing one of $(\Sigma_n,g_n,\bzeta_n\circ o)$ and conversely. This glued surface is a torus equipped with an antiholomorphic involution $\tau$ fixing two embedded circles, the left copy of $\Sigma_n$ is a subset of $\Sigma'_n$ and the right copy becomes $\tau(\Sigma_n)$ the other half. 
We use the same notations $g_n,K_n,{\rm v}_n$ on the doubled surface $\Sigma_n'$. 
By the gluing lemma \cite[Proposition 5.2]{GKRV21_Segal} applied to the right hand side of \eqref{boundA2n}, we have
\begin{align*}
\iint & | A^{(2)}_{n,\eps}|^2 (\dd \mu_0)^{\otimes 2}\\
\leq &C'(f_n)\int \E\Big[e^{-\frac{Q}{4\pi}\int_{\Sigma'_n}K_n(c+X'_n){\rm dv}_{n}}\big(\mu^\eta M_{\gamma}^{g_n} (c+X'_{n} ,  \Sigma_n \setminus \Sigma^\eps)^\eta\wedge 1\big)\,
\big( \mu^\eta M_{\gamma}^{g_n} (c+X'_{n} , \tau( \Sigma_n \setminus \Sigma^{\eps}))^\eta\wedge 1\big)\,
\\
&   \qquad \qquad\qquad \qquad  \times e^{-  \mu M_{\gamma }^{g_n} (c+X'_{n}, \Sigma^{\eps}) -\mu M_{\gamma }^{g_n} (c+X'_{n},\tau  (\Sigma^{\eps}))   }  \Big]  \,\dd c 
   \\
   \leq &C'(f_n)   \int \E\Big[e^{-\frac{Q}{4\pi}\int_{\Sigma'_n}K_n(c+X'_n){\rm dv}_{n}}\mu^\eta M_{\gamma }^{g_n} (c+X'_{n} ,  \Sigma_{n} \setminus \Sigma^{\eps})^\eta
   e^{  -\mu M_{\gamma}^{g_n} (c+X'_{n},\tau  (\Sigma^{\eps}))   }  \Big]  \,\dd c 
\end{align*}
where $C'(f_n):=\big(\frac{{\rm v}_{g_n}(\Sigma_n')}{\det(\Delta_{\Sigma_n',g_n})}\big)^\hf$ is a metric dependent constant, and $X'_n$ is the   GFF on $\Sigma_n'$ in the glued metric $g_n$. On the torus $\Sigma_n'$, the Gauss-Bonnet formula gives $\int_{\Sigma_n'}K_n {\rm dv}_{n}=0$, so that the $c$ contribution from the integral in the first exponential term is $0$. Furthermore, applying the Cameron-Martin theorem to the curvature term involving the GFF   gives
\begin{multline*}
\iint  | A^{(2)}_{n,\eps}|^2 (\dd \mu_0)^{\otimes 2}  \leq  C'(f_n) e^{\frac{Q^2}{32\pi^2}V'(f_n)}
\\
\times \int \E\Big[ \big(\mu M_{\gamma }^{g_n} (c+X'_{n}+G'(f_n) ,  \Sigma_n \setminus \Sigma^\eps)\big)^\eta 
   e^{  -\mu M_{\gamma }^{g_n} (c+X'_{n}+G'(f_n),\tau  (\Sigma^\eps))   }  \Big]  \,\dd c 
\end{multline*}
where $G'(f_n)(x):=\int_{\Sigma_n'} K_n(x')G'_n(x,x')  {\rm dv}_{n}(x')$ and $V'(f_n):=\int_{\Sigma_n'}K_{n}  G'(f_n)  {\rm dv}_{n} $, with $G'_n$ the Green function on $(\Sigma_n',g_n)$. 
Next, we perform the change of variables $y=e^{\gamma c}M_{\gamma }^{g_n} (X'_{n}+G'(f_n),\tau  (\Sigma^\eps))$, integrate out the $y$ variable and use    H\"older's inequality   to get  \begin{align*}
\iint &  | A^{(2)}_{n,\eps}|^2 (\dd \mu_0)^{\otimes 2} 
\\
 \leq & C'(f_n) e^{\frac{Q^2}{32\pi^2}V'(f_n)} \frac{1}{\gamma}\Gamma(\eta)  \E\Big[ M_{\gamma}^{g_n} (X'_{n}+G'(f_n) ,  \Sigma_n \setminus \Sigma^\eps)^\eta 
    M_{\gamma }^{g_n} (X'_{n}+G'(f_n),\tau  (\Sigma^\eps))^{-\eta   }  \Big]   
    \\
 \leq & C'(f_n) e^{\frac{Q^2}{32\pi^2}V'(f_n)} 
 \frac{1}{\gamma}\Gamma(\eta) 
  \E\Big[ M_{\gamma}^{g_n} (X'_{n}+G'(f_n) ,  \Sigma_n \setminus \Sigma^\eps)^{p \eta}\Big]^{1/p} \E\Big[    M_{\gamma }^{g_n} (X'_{n}+G'(f_n),  \Sigma^\eps)^{-\eta q  }  \Big]   ^{1/q}
\end{align*}
where $p,q> 1$ are conjugate exponents. We choose $p$ in such a way that $p\eta=1$. Furthermore, since $f_n\to f$ in $C^\infty$ there is a family of diffeomorphism $\psi_n:\Sigma'_n\to \Sigma'$ where $\Sigma'$ is the double of $\Sigma$ (defined as $\Sigma_n'$ but with $f$ replacing $f_n$) and $(\psi_n)_*g_n$ is a smooth family of metrics on $\Sigma'$ converging in $C^\infty$ to the metric on $\Sigma'$ induced by $g$.
We can then apply Lemma \ref{limit_Green} and deduce  that $G'(f_n)$ has a limit in $L^\infty$, thus
\begin{equation}\label{estgreen}
\sup_n\sup_{x\in \Sigma'_n }|G'(f_n)(x)|\leq C
\end{equation}
for some $C>0$. Then
\begin{align*}
\iint    | A^{(2)}_{n,\eps}|^2 (\dd \mu_0)^{\otimes 2} 
 \leq  C'(f_n) e^{\frac{Q^2}{32\pi^2}V'(f_n)}  \frac{1}{\gamma}\Gamma(\eta)e^{2\gamma C }  \E\Big[ M_{\gamma }^{g_n} (X'_{n} ,  \Sigma_n \setminus \Sigma^\eps) \Big]^{1/p} \E\Big[    M_\gamma^{g_n} (X'_{n} ,  \Sigma^\eps)^{-\eta q  }  \Big]   ^{1/q}.
\end{align*}
We have to evaluate both expectations above. First, by item (1) of Lemma \ref{lemgreen} below and for some constant $C$ (which may change along lines), 
 \begin{align*}
 \E\Big[ M_{\gamma}^{g_n} (X'_{n} ,  \Sigma_n \setminus \Sigma^\eps) \Big]   \leq C {\rm v}_n( \Sigma_n \setminus \Sigma^\eps  )  \leq C \eps.
\end{align*}
Next, we can choose a non-empty open set $\mc{O}\subset  \Sigma^\eps$ over which $\chi_f=0$  (recall \eqref{defg_tadmissible}) and we get
$$\E\Big[    M_{\gamma}^{g_n} (X'_{n},  \Sigma^\eps)^{-\eta q  }  \Big] \leq \E\Big[    M_{\gamma}^{g_{\A}} (X'_{n} ,  \mc{O})^{-\eta q  }  \Big] .$$
The last expectation can then be shown to be bounded independently of $n$ using the standard argument of Kahane's convexity  inequality (see e.g. \cite{rhodes2014_gmcReview}), by using item (2) of Lemma \ref{lemgreen}.  All in all, and taking into account that $C'(f_n)$  is bounded uniformly in $n$ (as it depends continuously on the metric \cite{Ray-Singer} and $f_n\to f$ in $C^\infty$), we get 
\[\iint   | A^{(2)}_{n,\eps}|^2 (\dd \mu_0)^{\otimes 2} \leq C\eps^{1/p}.\]
 This proves that the operator with integral kernel $A^{(2)}_{n,\eps}$ can be made arbitrarily small in operator norm with $\eps$, uniformly in  $n$.
 
We can write similarly  $\mc{A}_{\Sigma,g,\bzeta}( \tilde{\bs{\varphi}})=  A^{(1)}_{\eps}+A^{(2)}_{\eps}$  where   $A^{(2)}_{\eps}$ yields an operator on $\mc{H}$ whose norm is less or equal to $C \eps^{1/(2p)}$ and  $A^{(1)}_{\eps} $ is given by 
\begin{equation*}
A^{(1)}_{\eps} = C(f)e^{ -\frac{1}{2}(\tilde{ \bs{\varphi}}, (\mathbf{D}_{S} - \mathbf{D})\tilde{ \bs{\varphi}})_2}  \E_{ \tilde{\bs{\varphi}}} \left  [ e^{-\frac{Q}{4\pi}\int_{\Sigma}K(X+P\tilde{ \bs{\varphi}}){\rm dv}  }   e^{-  \mu M_{\gamma }^{g} (X+P\tilde{ \bs{\varphi}},\Sigma^\eps)   }  \right].
\end{equation*}
We now estimate the difference $A^{(1)}_{n,\eps}-A^{(1)}_{\eps}$. For simplicity and because the map $f\in  \mc{S}_\epsilon\mapsto  C(f) $ is continuous, 
 it suffices to consider the case where $C(f)$ is replaced by $1$, which we do now. 
 We write $A^{(1)}_{n,\eps}-A^{(1)}_{\eps}= B^{(1)}_{n,\eps}+B^{(2)}_{n,\eps} $ where 
\[ \begin{split}
B^{(1)}_{n,\eps}= &e^{ -\frac{1}{2}(\tilde{ \bs{\varphi}}  , (\mathbf{D}_{S} - \mathbf{D})\tilde{ \bs{\varphi}})_2}   \E_{ \tilde{\bs{\varphi}}} \Big[ e^{-\frac{Q}{4\pi}\int_{\Sigma_n}K_{n} (X_n+P_n\tilde{ \bs{\varphi}} ){\rm dv}_{n} -  \mu M_{\gamma }^{g_n} (X_n +P_n \tilde{ \bs{\varphi}},\Sigma^\eps)   }  \Big]
 \\
 &-e^{ -\frac{1}{2}(\tilde{ \bs{\varphi}}  , (\mathbf{D}_{S} - \mathbf{D})\tilde{ \bs{\varphi}})_2} \E_{ \tilde{\bs{\varphi}} }  \Big[   e^{-\frac{Q}{4\pi}\int_{\Sigma}K (X+P\tilde{ \bs{\varphi}} ){\rm dv} -  \mu M_{\gamma }^g (X+P\tilde{ \bs{\varphi}},\Sigma^\eps)   }  \Big]  
\end{split}\]
and 
\begin{equation*}
B^{(2)}_{n,\eps} = ( e^{ -\frac{1}{2}(\tilde{ \bs{\varphi}}  , (\mathbf{D}_{S_n} - \mathbf{D})\tilde{ \bs{\varphi}} )_2} - e^{ -\frac{1}{2}(\tilde{ \bs{\varphi}} , (\mathbf{D}_{S} - \mathbf{D})\tilde{ \bs{\varphi}})_2} )
 \E_{ \tilde{\bs{\varphi}} } \Big  [  e^{-\frac{Q}{4\pi}\int_{\Sigma_n}K_{n} (X_n+P_n\tilde{ \bs{\varphi}} ){\rm dv}_{n} - \mu M_{\gamma}^{g_n} (X_n+P_n \tilde{ \bs{\varphi}} , \Sigma^\eps)   }  \Big].
\end{equation*}
Concerning $B^{(2)}_{n,\eps}$, we can apply again the Cameron-Martin theorem to get that
\begin{equation}\label{exprb2n}
\begin{split}
B^{(2)}_{n,\eps}=& ( e^{ -\frac{1}{2}(\tilde{ \bs{\varphi}} , (\mathbf{D}_{S_n} - \mathbf{D})\tilde{ \bs{\varphi}} )_2} - e^{ -\frac{1}{2}(\tilde{ \bs{\varphi}}  , (\mathbf{D}_{S} - \mathbf{D})\tilde{ \bs{\varphi}} )_2} ) \\
& \times e^{\frac{Q^2}{32\pi^2}V(f_n)-\frac{Q}{4\pi}\int_{\Sigma_n}K_{n}P_n\tilde{ \bs{\varphi}}{\rm dv}_{n} }  \E_{ \tilde{\bs{\varphi}}} \Big  [    e^{- \mu  M_{\gamma }^{g_n} (X_n+P_n \tilde{ \bs{\varphi}} +G(f_n), \Sigma^\eps)   }  \Big]
\end{split}
\end{equation}
where $G(f_n)(x):=\int_{\Sigma_n} K_{n}(x')G_n(x,x')  {\rm dv}_{n}(x')$ and $V(f_n):=\int_{\Sigma_n}K_{n}  G(f_n)  {\rm dv}_{n} $, with $G_n$ the Green function on $\Sigma_n$ with Dirichlet boundary conditions (and similarly for $G(f)$ and $V(f)$ for the corresponding quantities associated to $f$ and $X$ on $\Sigma$). Let us now single out the zero mode of the two boundary fields as $\tilde{ \bs{\varphi}}=(c_1,c_2)+  \bs{\varphi}$ and let us set $y:=c_2-c_1$. Note then that $P_n\tilde{ \bs{\varphi}}=c_1+P_ny+P_n \bs{\varphi}$ so that 
\begin{equation}\label{minorcurv}
\int_{\Sigma_n}K_{n}P_n\tilde{ \bs{\varphi}}{\rm dv}_{n}=\int_{\Sigma_n}K_{n}(c_1+P_ny+P_n \bs{\varphi}){\rm dv}_{n}\geq -C|y|-C_N\|  \bs{\varphi}\|_{H^{-N}}
\end{equation} 
where we have used Gauss-Bonnet to remove the $c_1$-contribution, the bound $|P_ny|\leq |y|$ as a harmonic map and item (5) of Lemma \ref{lemgreen}, for arbitrary $  N>0$ and some constant $C_{N} >0$ uniform in $n$. Here and below, $\|\cdot\|_{H^{-N}}$ denotes the $(H^{-N}(\T))^2$ norm.
Then we plug  this estimate in the expression \eqref{exprb2n} for $B^{(2)}_{n,\eps}$,  use again   bounds uniform in $n$ for $V(f_n)$,  and bound the expectation by $1$.  The fact that $B^{(2)}_{n,\eps}\to 0$ in $\mc{L}(\mc{H})$ as $n\to \infty$ then follows from item (6) of Lemma \ref{lemgreen}.

Next we focus on  $B^{(1)}_{n,\eps}$ and we want to prove it converges to $0$ in $\mc{L}(\mc{H})$ as $n\to \infty$. Applying Cameron-Martin gives
\[ \begin{split}
B^{(1)}_{n,\eps}= &e^{ -\frac{1}{2}(\tilde{ \bs{\varphi}} , (\mathbf{D}_{S} - \mathbf{D})\tilde{ \bs{\varphi}})_2}   e^{\frac{Q^2}{32\pi^2}V(f_n)-\frac{Q}{4\pi}\int_{\Sigma_n}K_{n}  P_n\tilde{ \bs{\varphi}}{\rm dv}_{n} } \E_{ \tilde{\bs{\varphi}}} \Big[   e^{- \mu  M_{\gamma }^{g_n} (X_n+P_n \tilde{ \bs{\varphi}} +G(f_n), \Sigma^\eps)   }  \Big]
 \\
 &-e^{ -\frac{1}{2}(\tilde{ \bs{\varphi}} , (\mathbf{D}_{S} - \mathbf{D})\tilde{ \bs{\varphi}})_2} e^{\frac{Q^2}{32\pi^2}V(f)-\frac{Q}{4\pi}\int_{\Sigma}K  P\tilde{ \bs{\varphi}}{\rm dv} } \E_{ \tilde{\bs{\varphi}} }  \Big[   e^{-  \mu M_{\gamma }^{g} (X+P\tilde{ \bs{\varphi}} +G(f), \Sigma^\eps)   }  \Big].
\end{split}\]
Arguing as for $B^{(2)}_{n,\eps}$, we can show that the operator with integral kernel 
$$
e^{ -\frac{1}{2}(\tilde{ \bs{\varphi}}  , (\mathbf{D}_{S} - \mathbf{D})\tilde{ \bs{\varphi}})_2}  e^{ -\frac{Q}{4\pi}\int_{\Sigma_n}K_{n}  P_n\tilde{ \bs{\varphi}}{\rm dv}_{n} } \E_{ \tilde{\bs{\varphi}} } \Big[   e^{-  \mu M_{\gamma }^{g_n} (X_n+P_n \tilde{ \bs{\varphi}} +G(f_n), \Sigma^\eps)   }  \Big]
$$
is bounded as an operator $\mc{H}\to\mc{H}$ and, because the map $f\in \mc{S}_>\mapsto V(f)$ is continuous, we can thus replace $V(f_n)$ by $V(f)$. As a consequence, we can replace $V(f)$ by 1 without loss of generality. 

In the same way, we want to replace the curvature  term $\int_{\Sigma_n}K_{n}  P_n\tilde{ \bs{\varphi}}{\rm dv}_{n} $ by $\int_{\Sigma_n}K  P\tilde{ \bs{\varphi}}{\rm dv} $. So we have to evaluate the cost of this replacement. Using again the decomposition $\tilde{ \bs{\varphi}}=(c_1,c_2)+  \bs{\varphi}$ as before, and by bounding the expectation by $1$, this cost is less than $e^{ -\frac{1}{2}(\tilde{ \bs{\varphi}}  , (\mathbf{D}_{S} - \mathbf{D})\tilde{ \bs{\varphi}})_2}C_{N,n}(|y|+\|  \bs{\varphi}\|_{H^{-N}})e^{C(|y|+\|  \bs{\varphi}\|_{H^{-N}})}$, for some sequence $C_{N,n}\to 0$ as $n\to\infty$ by item (5) of Lemma \ref{lemgreen}, and thus converges to $0$ as $n\to\infty$ in $\mathcal{L}(\mc{H})$ by item (6) of Lemma \ref{lemgreen}. This means that we can estimate $B^{(1)}_{n,\eps}$ by the "difference of expectations" times $ e^{ -\frac{1}{2}(\tilde{ \bs{\varphi}} , (\mathbf{D}_{S} - \mathbf{D})\tilde{ \bs{\varphi}})_2}   e^{C(|y|+\|  \bs{\varphi}\|_{H^{-N}})}$, up to a term tending to $0$ as $n\to\infty$ in $\mathcal{L}(\mc{H})$. We call $\tilde{B}^{(1)}_{n,\eps}$ this quantity. We will show that  $\tilde{B}^{(1)}_{n,\eps}$  tends to $0$ in $L^2(\mu_0^{\otimes 2})$, which is equivalent to the convergence in Hilbert-Schmidt norm for the operator with this kernel.

For $c_1$ positively  large, the two potentials in each expectation will make them arbitrarily small, uniformly in $n$. Indeed, using  the inequality $e^{ -\frac{1}{2}(\tilde{ \bs{\varphi}} , (\mathbf{D}_{S} - \mathbf{D})\tilde{ \bs{\varphi}})_2} \leq e^{-C_0|c_1-c_2|^2/2}$ for some $C_0>0$ (see  \cite[Lemma 4.5]{GKRV21_Segal}), we get
\[\begin{gathered}
 \iint \mathbf{1}_{c_1>A} |\tilde{B}^{(1)}_{n,\eps}|^2 \dd \mu_0^{\otimes 2}
 \leq e^{-\kappa A}
 \iint e^{\kappa c_1} e^{-C_0|c_1-c_2|^2+C|c_1-c_2|+ C\|\bs{\varphi}\|_{H^{-N}}} F( \tilde{\bs{\varphi}}) \dd \mu_0^{\otimes 2}( \tilde{\bs{\varphi}}),\\ 
\textrm{ with } F( \tilde{\bs{\varphi}}):=\big| \E_{ \tilde{\bs{\varphi}} }  \big[   e^{-  \mu M_{\gamma }^{g} (X+P\tilde{ \bs{\varphi}} +G(f), \Sigma^\eps)   }  \big]   \big|^2+
\big|\E_{ \tilde{\bs{\varphi}}} \big[   e^{- \mu  M_{\gamma }^{g_n} (X_n+P_n \tilde{ \bs{\varphi}} +G(f_n), \Sigma^\eps)   }  \big]\big|^2
\end{gathered}\]
The further exponential term $e^{\kappa c_1}$ comes from the inequality $\mathbf{1}_{c_1>A}\leq e^{\kappa (c_1-A)}$. 
The two expectations then behave in the same way, which can be seen using Kahane's comparison inequalities for GMC  combined with items (2) and (5) of Lemma \ref{lemgreen}. Using again \eqref{minorcurv} and bounding the Green's function term by a constant, we deduce that
\begin{align*}
&\iint \mathbf{1}_{c_1>A} |\tilde{B}^{(1)}_{n,\eps}|^2 \dd \mu_0^{\otimes 2}
\\
&  \qquad \leq 2e^{-\kappa A}\iint  e^{\kappa c_1} e^{-C_0|c_1-c_2|^2 + C\|\bs{\varphi}\|_{H^{-N}}+C|c_1-c_2|}  \Big| \E_{ \tilde{\bs{\varphi}} }  \Big[   e^{-  \mu e^{\gamma c_1}e^{-\gamma|c_1-c_2|-C_N\gamma\|  \bs{\varphi}\|_{H^{-N}}}M_{\gamma }^{g} (X, \Sigma^\eps)   }  \Big]   \Big|^2 
   \dd \mu_0^{\otimes 2}.
\end{align*}
We use Jensen's inequality to drop the square inside the expectation. Then, making the change of variables $y=c_2-c_1$ and $v=\mu e^{\gamma c_1}e^{|c_1-c_2|-C_N\|  \bs{\varphi}\|_{H^{-N}(\T)}}M_{\gamma }^{g} (X, \Sigma^\eps)  $, we obtain 
\[
\iint \mathbf{1}_{c_1>A} |\tilde{B}^{(1)}_{n,\eps}|^2 \dd \mu_0^{\otimes 2}
 \leq
 C'\int   e^{-C_0|y|^2 +(C+\kappa)|y|}  \dd y\E  \Big[    M_{\gamma }^{g} (X, \Sigma^\eps)  ^{-\frac{\kappa}{\gamma}}  \Big]     
\iint   e^{C' \|\bs{\varphi}\|_{H^{-N}}}\P^{\otimes 2}(\dd \bs{\varphi} ).
\]
for some constant $C'>0$ independent of $n$.
It is a standard fact in GMC theory (see \cite{rhodes2014_gmcReview}) that the expectation with the inverse power of GMC is finite. We check that the other terms are also finite so that this quantity can be made arbitrarily small for large $A$. The fact that 
$e^{C' \|\bs{\varphi}\|_{H^{-N}}}$ is integrable can be checked by bounding
\[ \int   e^{C' \|\varphi\|_{H^{-N}}}\P(\dd \varphi)\leq  \prod_{n=1}^\infty \int e^{-\frac{1}{2}(x_n^2+y_n^2)+C_n(|x_n|+|y_n|)}\frac{\dd x_n\dd y_n}{2\pi}\leq e^{\frac{1}{2}\sum_n C_n^2}<\infty\]  
where $C_n=\mc{O}(|n|^{-N})$.
For negatively large $c_1$, the difference of expectations produce a decaying term of the form $e^{\gamma c_1}$, which allows us to adapt the previous argument. Indeed,  using the inequalities $|1-e^{-x}|\leq |x|^{1/2}$ for $x\geq 0$, and $(a+b)^2\leq 2a^2+2b^2$,  we deduce
\begin{align*}
\Big| \E_{ \tilde{\bs{\varphi}}} \big[ &  e^{- \mu  M_{\gamma }^{g_n} (X_n+P_n \tilde{ \bs{\varphi}} +G(f_n), \Sigma^\eps)   }  \big]
 - \E_{ \tilde{\bs{\varphi}} }  \big[   e^{-  \mu M_{\gamma }^{g} (X+P\tilde{ \bs{\varphi}} +G(f), \Sigma^\eps)   }  \big] \Big| ^2
 \\
 \leq & 2 \mu  \Big(\E_{ \tilde{\bs{\varphi}}} \big[     M_{\gamma }^{g_n} (X_n+P_n \tilde{ \bs{\varphi}} +G(f_n), \Sigma^\eps)     \big]
 + \E_{ \tilde{\bs{\varphi}} }  \big[      M_{\gamma }^{g} (X+P\tilde{ \bs{\varphi}} +G(f), \Sigma^\eps)     \big] \Big)
 \\
  \leq & 2C \mu  e^{\gamma c_1} e^{|y|}\Big(\E_{ \tilde{\bs{\varphi}}} \big[     M_{\gamma }^{g_n} (X_n+P_n  \bs{\varphi} +G(f_n), \Sigma^\eps)     \big]
 + \E_{ \tilde{\bs{\varphi}} }  \big[      M_{\gamma }^{g} (X+P  \bs{\varphi} +G(f), \Sigma^\eps)     \big] \Big).
\end{align*}
Next we use  item (3) and  (5) of Lemma \ref{lemgreen} to get rid of the terms $G(f_n)$ and $G(f)$, and the harmonic extensions, so as to get for $A>0$
\begin{align*}
\iint \mathbf{1}_{c_1<-A} |\tilde{B}^{(1)}_{n,\eps}|^2 \dd \mu_0^{\otimes 2}
 \leq
2C \iint  \mathbf{1}_{c_1<-A} e^{\gamma c_1}  e^{ -C_0|y|^2}   e^{C|y|} \dd c_1 \dd y  \iint e^{C\|  \bs{\varphi}\|_{H^{-N}(\T)}}  
\P^{\otimes 2}(\dd \bs{\varphi} )\times ({\rm v}_n+{\rm v})(\Sigma^\epsilon),
\end{align*}
where we have used again $e^{ -\frac{1}{2}(\tilde{ \bs{\varphi}} , (\mathbf{D}_{S} - \mathbf{D})\tilde{ \bs{\varphi}})_2} \leq e^{-C_0|c_1-c_2|^2/2}$. As before, this term can be made arbitrarily small for large $A$, uniformly in $n$.

To complete the proof, it thus remains to evaluate the integral over the region where $c_1$ belongs to a compact set, namely the quantity (after using again the estimate $e^{ -\frac{1}{2}(\tilde{ \bs{\varphi}} , (\mathbf{D}_{S} - \mathbf{D})\tilde{ \bs{\varphi}})_2} \leq e^{-C_0|c_1-c_2|^2/2}$)
\begin{align*}
 \int_{|c_1|\leq A}\int_\R  \iint e^{ -C_0|y|^2} e^{C(|y|+\|  \bs{\varphi}\|_{H^{-N}(\T)})}  D_n(c_1,y,\bs{\varphi})\dd c_1 \dd y\P^{\otimes 2}(\dd \bs{\varphi} )
\end{align*}
with 
$$D_n(c_1,y,\bs{\varphi}):=\Big| \E_{ \tilde{\bs{\varphi}}} \big[    e^{- \mu  M_{\gamma }^{g_n} (X_n+P_n \tilde{ \bs{\varphi}} +G(f_n), \Sigma^\eps)   }  \big]
 - \E_{ \tilde{\bs{\varphi}} }  \big[   e^{-  \mu M_{\gamma }^{g} (X+P\tilde{ \bs{\varphi}} +G(f), \Sigma^\eps)   }  \big] \Big| ^2 .$$
We have to show that it converges to $0$ as $n\to\infty$. This follows from the dominated convergence theorem. Indeed, $D_n$ is bounded by $2$ and, furthermore, it converges almost surely in $(c_1,y,\bs{\varphi})$ to $0$. Up to the convergences  of the harmonic extensions and the curvature terms, treated with items (4) and (3)  of Lemma \ref{lemgreen}, this amounts to showing that the expectation  of the exponential of the GMC with the field $X_n$ converges to that with the field $X$. Because the function $x\mapsto \exp(-x)$ is convex,  this follows from  the combination of  Kahane's inequalities with item (2) of Lemma \ref{lemgreen}.
\end{proof}

We next prove the technical Lemma that has been used in the proposition above. As above, $\|\cdot\|_{H^{-N}}$ denotes the $H^{-N}(\T)^2$ norm.
\begin{lemma}\label{lemgreen}
 Let $(f_n)_n$ be a sequence    converging in $\mc{S}$ towards  $f$ as above and $(\Sigma_{n},g_n):=(\A_{f_n},g_{f_n})$, $(\Sigma,g):=(\A_f,g_f)$, and denote $S=(\A_f,J_\C,\bzeta_f)$ and $S_n=(\A_{f_n},J_\C,\bzeta_{f_n})$. Let $P_n \tilde{\bs{\varphi}} $ and $P \tilde{\bs{\varphi}} $   denote the harmonic extensions of $\tilde{\bs{\varphi}}\in H^s(\T)^2$ for $s<0$ on $\Sigma_{n}$ and $\Sigma$,  $ {\rm dv}_n$ and $ {\rm dv}$   the corresponding Riemannian measures (resp. boundary lengths), $G_n$ and $G$ the corresponding Dirichlet Green's functions. Let   $\Sigma^\eps=\{x\in \Sigma\,|\, d_{\R^2}(x,\pl \Sigma)\geq \eps\}$ for $\eps>0$ small. Then
 \begin{enumerate}
\item there exists $C>0$ such that for all $n$, ${\rm v}_n( \Sigma_{n} \setminus \Sigma^{\eps}  )   \leq C\eps$,
\item there exists a sequence $(a_n)_n$ such that $a_n\to 0$ as $n\to\infty$ and $\sup_{(\Sigma^{\eps})^2}|G_n-G|\leq  a_n$,
\item   $\sup_{x\in \Sigma^{\eps}}|\int_{\Sigma_n}G_n(x,x')K_{g_n}(x'){\rm dv}_n(x')-\int_{\Sigma}G(x,x')K_g(x'){\rm dv}(x')|\to 0$ as $n\to\infty$, where $K_{g_n}$ and $K_g$ are the scalar curvatures of $g_n$ and $g$,
\item the sequence $\sup_{\Sigma^{\eps}}|P_n \tilde{\bs{\varphi}} -P \tilde{\bs{\varphi}} |$ converges to $0$ when $n\to\infty$ for any $\tilde{\bs{\varphi}} \in  (H^{s}(\T))^2 $ with $s<0$,
\item for fixed $N\in\R$, we have the estimates, for some constants $C_N>0$ and $C_{N,n}>0$ such that $C_{N,n}\to 0$ as $n\to\infty$ for fixed $N$,
\begin{align*}
{\rm i)} & & \sup_n\big|\int_{\Sigma_{n}}K_{g_n}P_n\tilde{ \bs{\varphi}} {\rm dv}_{n}\big| \leq & C_N \|\tilde{ \bs{\varphi}} \|_{H^{-N}},
\\
{\rm ii)}  & &  \big|\int_{\Sigma_{n}}K_{g_n}P_n\tilde{ \bs{\varphi}} {\rm dv}_{n}-\int_{\Sigma}K_gP\tilde{ \bs{\varphi}} {\rm dv}\big| \leq & C_{N,n} \|\tilde{ \bs{\varphi}} \|_{H^{-N}},
 \\
{\rm iii)}  & &  \sup_{n}\sup_{x\in \Sigma^{\eps}}|P_n\tilde{ \bs{\varphi}}|+|P\tilde{ \bs{\varphi}}|\leq  & C_{N} \|\tilde{ \bs{\varphi}} \|_{H^{-N}},
\end{align*}
\item for each $C>0$ and $N>0$, the operator on $\mc{H}$ with integral kernel  given by 
\[(e^{ -\frac{1}{2}(\tilde{ \bs{\varphi}}  , (\mathbf{D} _{S_n} - \mathbf{D})\tilde{ \bs{\varphi}} )_2}- e^{ -\frac{1}{2}(\tilde{ \bs{\varphi}}  , (\mathbf{D} _{S} - \mathbf{D})\tilde{ \bs{\varphi}} )_2})  e^{C(|y|+ \| \bs{\varphi} \|_{H^{-N}})}
\] 
goes to $0$ in $\mc{L}(\mc{H})$ as $n\to \infty$, where $y=c_1-c_2$ if $\tilde{\bs{\varphi}}=(c_1,c_2)+\bs{\varphi}$ with the two fields in $\bs{\varphi}$ orthogonal to constants on $\T$. Furthermore, the family of operators on $\mc{H}$ with integral kernels  given by 
$$(e^{ -\frac{1}{2}(\tilde{ \bs{\varphi}}  , (\mathbf{D}_{S_n} - \mathbf{D})\tilde{ \bs{\varphi}})_2}+e^{ -\frac{1}{2}(\tilde{ \bs{\varphi}}  , (\mathbf{D}_{S} - \mathbf{D})\tilde{ \bs{\varphi}})_2})e^{C(|y|+\|  \bs{\varphi}\|_{H^{-N}})}$$ is bounded in $\mc{L}(\mc{H})$ uniformly in $n$.
\end{enumerate}
\end{lemma}
\begin{proof} (1) is a direct application of Lebesgue theorem since ${\bf 1}_{\Sigma_{n}}{\rm dv}_n\to {\bf 1}_{\Sigma}{\rm dv}$ as Radon measures   in $\C$ and ${\rm v}(\Sigma\setminus \Sigma^{\eps})\leq C\eps/2$ for some constant $C>0$ independent of $\eps$.\\
(2) First, since $g_n$ is conformal to $|\dd x|^2$, we notice that $G_n(x,x')$ (resp. $G(x,x')$) are
 also equal to the Dirichlet Green function of the Euclidean Laplacian $\Delta$ on $(\Sigma_n,|\dd x|)$ (resp. $(\Sigma,|\dd x|$). 
 Let $G_{\D^*}(x,x')$ be the Dirichlet Green function 
on the exterior $\D^*$ of the unit disk on the Riemann sphere $\hat{\C}$. Let $\chi,\tilde{\chi},\hat{\chi}\in C_c^\infty(\C)$ be equal to $1$ in $\A_{\delta,\delta^{-1}}$ for $\delta<1$ close to $1$ and supported in $\A_{\delta',{\delta'}^{-1}}$ for $\delta'<\delta$ also close to $1$, and $\tilde{\chi}=1$ on support of $\chi$ while $(1-\hat{\chi})=1$ on support of $(1-\chi)$. Then on $\Sigma_n=\A_{f_n}$
\begin{align*} &\Delta_x ( (\tilde{\chi}(f_n^{-1}(x))G_{\D^*}(f_n^{-1}(x),f_n^{-1}(x'))(\chi(f_n^{-1}(x')))+(1-\hat{\chi}(f_n^{-1}(x))G_\D(x,x')(1-\chi(f_n^{-1}(x')))\\
&= \delta(x-x')+Q_{f_n}(x,x')
\end{align*}
where $\delta(x-x')$ is the Dirac mass on the diagonal (the kernel of the identity) and $Q_{f_n}\in C^\infty(\Sigma_n \times \Sigma_n)$ is vanishing on $\pl( \Sigma_n \times \Sigma_n)$. Let us call $G_{f_n}^0$ the operator with the integral kernel $G_{f_n}^0(x,x')$ above such that $\Delta G_{f_n}^0={\rm Id}+Q_{f_n}$, and similarly 
$G_f^0$ such that $\Delta G_{f}^0={\rm Id}+Q_{f}$. The operators $Q_{f_n}$ and $Q_f$ are smoothing, thus ${\rm Id}+Q_{f_n}$ and ${\rm Id}+Q_{f}$ are Fredholm of index $0$ on $L^2(\Sigma_n)$ and $L^2(\Sigma)$, and their integral kernels vanish respectively on the boundary $\pl (\Sigma_n\times \Sigma_n)$ and $\pl (\Sigma\times \Sigma)$.  
Then, there is a finite rank smoothing operator $W_f=\sum_{j=1}^Nu_j\otimes v_j$ such that $u_j,v_j\in C^\infty(\Sigma)$ vanish on $\pl \Sigma$ and 
such that $\Delta (G^0_{f}+W_f)={\rm Id}+Q'_f$, where $Q'_f$ is smoothing, its integral kernel vanishes on $\pl (\Sigma\times \Sigma)$ and $({\rm Id}+Q'_f)$ is invertible on all Sobolev spaces $H^k(\Sigma)$ for $k\in \N$ (the $v_j$ are in $\ker ({\rm Id}+Q_f)$ and $\Delta u_j=v_j$ with $v_j|_{\pl \Sigma}=0$). Moreover $({\rm Id}+Q'_f)^{-1}={\rm Id}+L_f$ with $L_f$ an operator with smooth integral kernel   vanishing on $\pl (\Sigma \times \Sigma)$: this follows from the identity $Q'_f+L_f+Q'_fL_f=Q'_f+L_f+L_fQ'_f=0$ which in turn gives 
\[ L_f=-Q'_f+(Q'_f)^2 +Q'_fL_fQ'_f.\]
 It is convenient to work on the 
fixed space $L^2(\A_f)$ and we can use a family of diffeomorphisms $\psi_n:\Sigma_n\to \Sigma$ such that $\psi_n\to {\rm Id}$ as $n\to \infty$ uniformly on all compact subset $\Omega\subset \Sigma^\circ$. We will  conjugate our operators on $\Sigma_n$ by the pull-back $\psi_n^*:L^2(\Sigma)\to L^2(\Sigma_n)$.
We have that $(\psi_n)_*\Delta(\psi_n)^* ((\psi_n)_*G^0_{f_n}\psi_n^*+W_{f})={\rm Id}+Q'_{f_n}$ for some smoothing operator $Q'_{f_n}$ on $\Sigma$ converging to $Q'_f$ in $\mc{L}(L^2(\Sigma))$ as $n\to \infty$. Moreover, the integral kernel of $Q'_{f_n}$ converges in any $H^N(\Sigma \times \Sigma)$ for all $N\in \N$. 
 By continuity of the inverse, $({\rm Id}+Q'_{f_n})^{-1}=:{\rm Id}+L_{f_n}$ is invertible on $L^2(\Sigma)$, $L_{f_n}$ is smoothing with integral kernel vanishing on $\pl (\Sigma \times \Sigma)$ (same argument as for $L_f$), $L_{f_n}\to L_f$ in $\mc{L}(L^2(\Sigma))$. As above we have 
 $L_{f_n}=-Q'_{f_n}+(Q'_{f_n})^2 +Q'_{f_n}L_{f_n}Q'_{f_n}$, and $Q'_{f_n}\to Q'_f$ as  maps $H^{-N}(\Sigma)\to H^N(\Sigma)$ for all $N>0$. We deduce that $L_{f_n}\to L_f$ as  a map $H^{-N}(\Sigma)\to H^N(\Sigma)$.
We then obtain on $\Sigma_n$
\[ \Delta (G^0_{f_n}+(\psi_n)^*W_{f}(\psi_n)_*)({\rm Id}+(\psi_n)^*L_{f_n}(\psi_n)_*)={\rm Id}\]
on $\Sigma_n$. Since $W_n:=(\psi_n)^*W_{f}(\psi_n)_*$ and $L_n:=(\psi_n)^*L_{f_n}(\psi_n)_*$ have integral kernels vanishing on $\pl(\Sigma_n \times \Sigma_n)$, we deduce that the Dirichlet Green function on $\Sigma_n$ is the integral kernel of the operator
\[ G_n=(G^0_{f_n}+W_n)({\rm Id}+L_n).\]
Similarly the Dirichlet Green function of $\Sigma$ is
\[ G=(G^0_{f}+W_{f})({\rm Id}+L_{f}).\]
Notice now that $G_{\D^*}(x,x')=-\log|x-x'|+F^*(x,x')$ with $F^*$ smooth and similarly $G_{\D}(x,x')=-\log|x-x'|+F(x,x')$ with $F$ smooth.
Therefore, on any compact subset $\Omega\subset \Sigma^\circ$, we have  for $n$ large enough 
\[ \begin{split}
G^0_{f_n}(x,x')-G^0_f(x,x')=& \big(\tilde{\chi}(f_n^{-1}(x))\chi(f_n^{-1}(x'))+(1-\hat{\chi}(f_n^{-1}(x)))(1-\chi(f_n^{-1}(x')))\big)
\log|x-x'|\\
&- \Big(\tilde{\chi}(f^{-1}(x))\chi(f^{-1}(x'))+(1-\hat{\chi}(f^{-1}(x)))(1-\chi(f^{-1}(x')))\big)
\log|x-x'|\\
& +H_n(x,x')
\end{split}\] 
with $H_n\in C^\infty(\Omega\times \Omega)$ converging in $C^N(\Omega\times \Omega)$ norm to $0$.
Notice that $G^0_{f_n}-G^0_f=H_n(x,x')$ in a neighborhood of the diagonal of $\Omega\times \Omega$, it is thus smooth in $\Omega\times \Omega$ and $\sup_{x,x'\in \Omega}|(G-G_n)|(x,x')\to 0$ as $n\to \infty$.\\
(3) Using the family of diffeomorphism $\psi_n:\Sigma_n \to \Sigma$ as above to work on the fixed surface $\Sigma$,  
(3) is a consequence of Lemma \ref{limit_Green} since $(\psi_n)_*g_n\to g$ in $C^\infty(\Sigma)$.\\
(4) The Poisson kernels $P_n(x,y)$ and $P(x,y)$ defined by $P_nu(x)=\int_{\pl \Sigma_n}P_{n}(x,y)u(y)\dd\ell_n(y)$ 
for $u\in C^\infty(\pl \Sigma_n)$ and $Pu(x)=\int_{\pl \Sigma}P(x,y)u(y)\dd\ell(y)$ for $u\in C^\infty(\pl \Sigma)$ are obtained from the Dirichlet Green's function by 
\[ P(x,y)=\pl_{\nu'}G(x,x')|_{x'=y}, \quad P_n(x,y)=\pl_{\nu'}G_n(x,x')|_{x'=y}\]
where $\nu'$ is the inward unit normal vector acting on the right variable on respectively $\pl \Sigma$ and $\pl \Sigma_n$   and $\dd\ell,\dd\ell_n$ are the induced measures on $\pl \Sigma$ and $\pl \Sigma_n$. Using the parametrisation of the boundary of $\Sigma_n,\Sigma$ and using our description of the difference of Green's function above, we directly see that since $f_n\to f$ in $C^\infty$
\[  \| P(\cdot, f(\cdot))- P_n(\cdot,f_n(\cdot))\|_{C^1(\Sigma^\eps\times \T)}+ \| P(\cdot, \cdot)- P_n(\cdot,\cdot)\|_{C^1(\Sigma^\eps\times \T)}\to 0\]
as $n\to \infty$. This shows (4).\\
(5) The integral kernel $(x,x')\mapsto P_n(x,f_n(x'))$ converges to $(x,x')\mapsto P(x,f(x'))$ in $C^\infty(\Omega\times \T)$ 
if $\Omega\subset \Sigma^\circ$ is any fixed compact set, and similarly $(x,x')\mapsto P_n(x,x')$ converges to $(x,x')\mapsto P(x,x')$ in $C^\infty(\Omega\times \T)$. Notice that $K_{g_n}$ is supported in a fixed compact set $\Omega$ of $\Sigma^\circ$ for all $n$ large enough as $g_n=(f_n)_*g_{\A}$ in $f_n(\A_{1,\delta^{-1}})$ for some $\delta<1$ (recall \eqref{defg_tadmissible}) and $K_{g_\A}=0$. These two facts together and the fact that $K_{g_n}\to K_g$ in $L^1(\Omega)$ imply (5).\\
(6) The operator ${\bf D}_{S_n}-{\bf D}$ is smoothing (i.e. with smooth integral kernel), its integral kernel in $C^\infty((\T\sqcup \T)\times (\T\sqcup \T))$ viewed as a $2\times 2$ matrix with values in $C^\infty(\T^2)$ can be computed by 
\[ \begin{split}
({\bf D}_{S_n}-{\bf D})_{jk} (\theta,\theta')= &\pl_\nu\pl_{\nu'}(G_{n}(x,x')-G^0_{f_n}(x,x'))|_{x=\zeta_j(e^{{\rm i}\theta}),x'=\zeta_k(e^{{\rm i}\theta'})}\\
=& \pl_\nu\pl_{\nu'}(W_n(x,x')+W_nL_n(x,x')+G^0_{f_n}L_n(x,x'))|_{x=\zeta_j(e^{{\rm i}\theta}),x'=\zeta_k(e^{{\rm i}\theta'})}
\end{split}
\]
where $\zeta_1(e^{{\rm i}\theta})=e^{{\rm i}\theta}$ and $\zeta_2(e^{{\rm i}\theta})=f_n(e^{{\rm i}\theta})$, $\nu,\nu'$ are the unit normal 
inward-pointing vectors in the left and right variables. The same holds with $f_n$ replaced by $f$.
Since $f_n\to f$ in the $C^\infty(\T)$ topology, our description above of $G_{f_n}L_{n},W_{n}$ and $L_{n}$ and the fact that, for each $N\in \N$, their integral kernels converge in  $C^N$ norms respectively to $G_{f}L_{f},W_{f}$ and $L_f$ implies that the integral kernel of 
${\bf D}_{S_n}-{\bf D}$ converges in $C^N$ norm to the integral kernel of ${\bf D}_{S}-{\bf D}$. The boundedness on $\mc{H}$ of the operator 
$\mc{A}^0_{S}$ with integral kernel $\mc{A}^0_{S}(\tilde{\bs{\varphi}})=e^{ -\frac{1}{2}(\tilde{ \bs{\varphi}} , (\mathbf{D} _S - \mathbf{D})\tilde{ \bs{\varphi}})_2}$ follows from the estimate proved in \cite[Lemma 4.5]{GKRV21_Segal}: $\exists a>0$  and $C>0$ such that for each $\tilde{\varphi}_j=c_j+\varphi_j$, with $\varphi_j\in H^{-s}(\T)$ orthogonal to constants and $c_j\in \R$,
\begin{equation}\label{A0bound} 
|\mc{A}^0_{S}(c_1+\varphi_1,c_2+\varphi_2)|\leq Ce^{-a(c_1-c_2)^2}.
\end{equation}
Moreover, the proof of these estimates  shows  that they are locally uniform with respect to $f$.
Using \eqref{A0bound} we have  for all $\tilde{\bs{\varphi}}=(c_1+\varphi_1,c_2+\varphi_2) \in H^{-s}(\T)^2$
\[ |e^{ -\frac{1}{2}(\tilde{ \bs{\varphi}} , (\mathbf{D} _{S} - \mathbf{D})\tilde{ \bs{\varphi}})_2}-e^{ -\frac{1}{2}(\tilde{ \bs{\varphi}} , (\mathbf{D} _{S_n} - \mathbf{D})\tilde{ \bs{\varphi}})_2}|\leq C e^{-a(c_1-c_2)^2} |e^{ -\frac{1}{2}(\tilde{ \bs{\varphi}} , (\mathbf{D} _{S_n} - \mathbf{D}_S)\tilde{ \bs{\varphi}})_2}-1|.\]
Now by the discussion above and using that $(\mathbf{D} _{S_n} - \mathbf{D}_S)$ kills the constant $(1,1)$, 
we have for each $N$ fixed
\[\begin{split}
|(\tilde{ \bs{\varphi}} , (\mathbf{D} _{S_n} - \mathbf{D}_{S})\tilde{ \bs{\varphi}})_2|
& \leq \eps_n\|\tilde{\bs{\varphi}}-(\frac{c_1+c_2}{2},\frac{c_1+c_2}{2})\|_{H^{-N}}^2\\
& \leq \eps_n(\frac{(c_1-c_2)^2}{2}+\|\varphi_1\|^2_{H^{-N}(\T)}+\|\varphi_2\|^2_{H^{-N}(\T)})
\end{split}\]
for some $\eps_n \to 0$ as $n\to \infty$ thus (writing $\varphi_j(\theta)=\sum_{k\not=0}\varphi_{jk}e^{ik\theta}$ as in \eqref{param})
\begin{align*}
& |(\mc{A}^0_{S_n}-\mc{A}^0_S)(\tilde{\bs{\varphi}})|\leq e^{-a\frac{(c_1-c_2)^2}{2}}B_n(\varphi_1,\varphi_2)\\
& B_n(\varphi_1,\varphi_2):=\max\Big(e^{\eps_n \sum_{j=1}^2\sum_{k=1}^\infty \frac{|\varphi_{jk}|^2}{k^{2N}}}-1, 1-e^{-\eps_n
\sum_{j=1}^2\sum_{k=1}^\infty \frac{|\varphi_{jk}|^2}{k^{2N}}}\Big).
\end{align*}
A direct calculation  using the estimate on $B_n$ shows that there is $C>0$ such that
\[  \|B_n\|_{L^2(\Omega_\T^2)} \leq C\eps_n.\]
Therefore, by Cauchy-Schwarz, for each $F,G\in \mc{H}$ 
\[\begin{split} 
|\cjg (\mc{A}^0_{S_n}-\mc{A}^0_S)F,G\cjd_{\mc{H}}|\leq & \|B_n\|_{L^2(\Omega_\T^2)}\int_{\R^2}e^{-\frac{a}{2}(c_1-c_2)^2}
\|F(c_1,\cdot)\|_{L^2(\Omega_\T)}\|G(c_2,\cdot)\|_{L^2(\Omega_\T)}\d c_1\d c_2\\
 \leq & C\eps_n \|F\|_{\mc{H}}\|G\|_{\mc{H}}
\end{split}\]
using the fact that the convolution operator by the Gaussian $e^{-ac^2/2}$ is bounded on $L^2(\R)$.
This proves the claim.
\end{proof}

\subsection{Holomorphic flows of Markovian vector fields and Virasoro algebra}\label{sec:Markovian_operator}

We first recall the construction in \cite{BGKRV} of a representation of the Virasoro algebra with central charge $c_{\rm L}=1+6Q^2$ into the Hilbert space $\mc{H}$. 

First, a holomorphic vector field ${\rm v}=v(z)\pl_z$ on a neighborhood of $\D$ is called \textbf{Markovian} if  
\begin{equation}\label{eq:def_markovian}
\Re(\bar{z}v(z))<0,\qquad \forall z\in\T. 
\end{equation} 
This condition ensures that its flow, defined by the ODE $\pl_tf_t(z)=v(f_t(z))$ with $f_0(z)=z$, is a family of holomorphic univalent maps with 
$f_{t+s}(\D)\subset f_t(\D)$ for all $t,s\geq 0$ and so that there is $z_0\in \D$ with $f_t(\D)\to z_0$ as $t\to \infty$; in particular $z_0$ is the unique zero of $v$. In what follows we will assume that $v(0)=0$, i.e. $z_0=0$. A particularly important case is when $v(z)=v_0(z):=-z$, in which case $f_t(z)=e^{-t}z$ is the flow of dilations. 
   
To the geometric semigroup $(f_t)_{t\geq0}\subset \mc{S}$, we associate a Markovian semigroup on $\cH$ in the following way.
For $F\in \mc{H}=L^2(H^{s}(\T),\mu_0)$ ($s<0$), we define the operator 
\[ \begin{array}{l}
 P_t^0 : \mc{H}\to \mc{H} , \\
 P^0_tF(c,\varphi):=|f_t'(0)|^{\frac{Q^2}{2}}\E_{\varphi}[F\big(c+(X\circ f_t+Q\log|f_t'/f_t|)|_\T\big)]
\end{array}\]
where $\E_\varphi$ means expectation conditionally on $\varphi$ (i.e. the expectation is in the $X_\D$ random variable; recall the notation $X=X_\D+P\tilde\varphi$). Note that this corresponds to taking $\mu=0$ and $f=f_t$ in Definition \ref{defTf}.

If $\omega:=-{\rm Re}(v'(0))>0$ is large enough (in the sense 
\begin{equation}\label{omegalargeenough}
\omega > K\sup_{z\in\D}|v(z)-{\rm Re}(v'(0))z|
\end{equation} for some $K>0$ large enough), it is shown in \cite[Theorem 2.8]{BGKRV} that $P_t^0=e^{-t{\bf H}^{0}_{\rm v}}$ is a continuous contraction semigroup on $\mc{H}$ with generator 
\[\bH_{\rm v}^0=\bL_{\rm v}^0+\tilde{\bL}_{\rm v}^0\] 
where, if $v(z)=-\sum_{n=0}^\infty v_nz^{n+1}$ and ${\rm v}=v(z)\pl_z$,
\[\bL_{\rm v}^0=\sum_{n=0}^\infty v_n\bL_n^0, \quad \tilde\bL_{\rm v}^0=\sum_{n=0}^\infty \bar{v}_n\tilde{\bL}_n^0.\]
Here $(\bL_n^0)_{n\in \Z}$ and $(\tilde\bL_n^0)_{n\in \Z}$ are two commuting unitary representations of the Virasoro algebra, as densely defined operators on $\cH$, given by (using the notations \eqref{Heisenberg} and Wick notation $:\!\bA_p\bA_q\!:$, meaning we put annihilation operators to the right)
\begin{equation}\label{defL0}
\bL_n^0:=-iQ(n+1)\bA_n+\sum_{m\in\Z}:\!\bA_m\bA_{n-m}\!:, \quad \tilde{\bL}_n^0:=-iQ(n+1)\tilde{\bA}_n+\sum_{m\in\Z}:\!\tilde{\bA}_m\tilde{\bA}_{n-m}\!:.
\end{equation}
For all $n,m\in\Z$, one has
\begin{align*}
&[\bL_n^0,\bL_m^0]=(n-m)\bL_{n+m}^0+\frac{c_L}{12}(n^3-n)\delta_{n,-m},\\
&[\tilde{\bL}_n^0,\tilde{\bL}_m^0]=(n-m)\tilde{\bL}_{n+m}^0+\frac{c_L}{12}(n^3-n)\delta_{n,-m},\\
&[\bL_n^0,\tilde{\bL}_m^0]=0,\qquad(\bL_n^0)^*=\bL_{-n}^0,\qquad(\tilde{\bL}_n^0)^*=\tilde{\bL}_{-n}^0.
\end{align*}

For $\mu\not=0$, one can also consider the Liouville semigroup ${\bf T}_{f_t}:\mc{H}\to \mc{H}$ associated to $v$. It is proved in \cite[Theorem 1.1]{BGKRV} that it is a continuous contraction semigroup on $\mc{H}$ of the form
\begin{equation}\label{def:propagatorvLiouville} 
{\bf T}_{f_t}=e^{-t\bH_{\rm v}}
\end{equation}
with generator 
\begin{equation}\label{defopHv}
\bH_{\rm v}=\sum_{n=0}^\infty v_n\bL_n+\sum_{n=0}^\infty \bar{v}_n\tilde{\bL}_n, \quad \textrm{ if } v(z)=-\sum_{n=0}^\infty v_nz^{n+1} , \, {\rm v}=v(z)\pl_z
\end{equation}
where 
\[ \bL_n =\bL_n^0+ \frac{\mu}{2}\int_0^{2\pi}e^{in\theta}e^{\gamma \varphi(\theta)}\dd \theta, \quad \tilde\bL_n =\tilde\bL_n^0+ \frac{\mu}{2}\int_0^{2\pi}e^{-in\theta}e^{\gamma \varphi(\theta)}\dd \theta. \]
Here $\bL_n,\tilde{\bL}_n$ and $\bH_{\rm v}$ are unbounded operators on $\mc{H}$, mapping continuously $\mc{D}(\mc{Q})\to \mc{D}'(\mc{Q})$ where $\mc{D}(\mc{Q})$ is the domain of the quadratic form $\mc{Q}$
\begin{equation}\label{FQ}
\mc{Q}(F)=\int_{\R}\Big(\frac{1}{2} \|\pl_cF\|_{L^2(\Omega_\T)}^2+\frac{Q^2}{2}\|F\|_{L^2(\Omega_\T)}^2+ 
\sum_{n\geq 1}(\|{\bf A}_nF\|_{L^2(\Omega_\T)}^2+\|\tilde {\bf A}_nF\|_{L^2(\Omega_\T)}^2)+\mu e^{\gamma c}\cjg VF,F\cjd_{L^2(\Omega_\T)}\Big) \dd c 
\end{equation}
associated to the \textbf{Liouville Hamiltonian}
\[{\bf H}:={\bf H}_{{\rm v}_0}={\bf L}_0+\tilde{{\bf L}}_0,\]
with ${\rm v}_0=-z\pl_z$ being the generator of dilations, and $\mc{D}'(\mc{Q})$ is the dual to $\mc{D}(\mc{Q})$. Above $V$ is a non-negative unbounded operator defined using the Gaussian multiplicative chaos, and  when $\gamma<\sqrt{2}$  it is the multiplication by  positive potential given by $V(\varphi)=\lim_{\eps\to 0}\int_0^{2\pi} e^{\gamma \varphi_\eps(\theta)-\frac{\gamma^2}{2}\E[\varphi_\eps(\theta)^2]}\dd \theta$ that belongs to $\bigcap_{\delta>0}L^{\frac{2}{\gamma^2}-\delta}(\Omega_\T)$ , if $\varphi_\eps(\theta)$ is a regularisation at scale $\eps>0$ of $\varphi$. When $\gamma\geq\sqrt{2}$, the interpretation of $V$ can be done using a Girsanov transform, see \cite[Section 5.2]{GKRV20_bootstrap}. The operators $\int_0^{2\pi}e^{-in\theta}e^{\gamma \varphi(\theta)}\dd \theta$ appearing in ${\bf L}_n$ are defined similarly, see \cite{BGKRV}.
Furthermore the definition of the operators $\bL_n,\tilde{\bL}_n:\mc{D}(\mc{Q})\to \mc{D}'(\mc{Q})$ can be extended to $n\in\Z$ by the relation 
\[\cjg \bL_{-n}F,G\cjd_{\mc{H}}:=\cjg F,\bL_n G\cjd_{\mc{H}}\qquad\text{ and }\qquad \cjg \tilde{\bL}_{-n}F,G\cjd_{\mc{H}}=\cjg F,\tilde{\bL}_n G\cjd_{\mc{H}}.\]

Because the vector field $v\pl_z$ is assumed to be holomorphic on a neighborhood of the disc $\D$, it turns out that the family $(f_t)_{t\geq 0}$ actually belongs to $\mc{S}_>$. Therefore   we can rewrite Corollary \ref{cor:annulus_propagator} in this context as:
\begin{corollary}\label{cor:annulus_propagator_t}
Let ${\rm v}=v(z)\pl_z$ be a holomorphic Markovian vector field defined on a neighborhood of the disc $\D$ with $v(0)=0$. If $\omega:=-{\rm Re}(v'(0))>0$ is large enough, in the sense \eqref{omegalargeenough}, and if  
 $(f_t)_{t\geq 0}\subset \mc{S}_>$ is the flow of holomorphic maps defined by $\partial_t f_t=v(f_t)$, $f_0(z)=z$, we  have the following expression 
 \[ e^{-t\bH_{\rm v}}F(\tilde{\varphi}_1)=\frac{e^{-\frac{\rm c_L}{12}W(f_t,g_{f_t})}}{\sqrt{2}\pi}
\int \mc{A}_{\mathbb{A}_{f_t},g_{f_t},\bzeta_t}(\tilde{\varphi}_1,\tilde{\varphi}_2)F(\tilde{\varphi}_2) \,\dd\mu_0(\tilde{\varphi}_2)
\]
where $\bzeta_t(e^{{\rm i}\theta})=(e^{{\rm i}\theta},f_t(e^{{\rm i}\theta}))$, $g_{f_t}$ is any family of admissible metrics on $\A_{f_t}$ 
depending smoothly on $t$, with $W(f,g)$ given by \eqref{def_of_W}.
\end{corollary}

This corollary will be used in the next section to differentiate the propagator $\mathbf{T}_f$.

\section{Differentiability of propagators}\label{Differentiability_propagator}
 
Here we investigate the differentiability of the map $f\mapsto \mathbf{T}_f$. Recall the definition of $\mc{S}_\eps$ in \eqref{S_eps_S_+}.

\begin{lemma}\label{lemopen}
The set $\mc{S}_\eps$ is  open in ${\rm Hol}^\bullet_\eps(\D)$, for all $\eps>0$. 
\end{lemma}

\begin{proof}
Let $f_0\in \mc{S}_\eps$. Since the mapping $\theta\mapsto f_0(e^{{\rm i}\theta})$ has non vanishing derivative we get $ \inf_{\theta\in [0,2\pi]}|f'_0(e^{{\rm i}\theta})|>0$ and therefore we can find a neighborhood $\mc{O}$ of $f_0$ such that $m:=\inf_{f\in\mc{O}}\inf_{\theta\in[0,2\pi]}|f'(e^{{\rm i}\theta})|>0$. By Taylor's formula, we can then find two constants $c,\delta>0$ such that for all $f\in\mc{O}$ and $z,z'\in\T$
$$|f(z)-f(z')| \geq c|z-z'|$$
  for $|z-z'|\leq \delta$ and all $f\in\mc{O}$. Let now $\eta:=\inf_{|z-z'|\geq\delta}|f_0( z)-f_0(z')|>0$ and $d:={\rm dist}(f_0(\T),\T)>0$. We stress that $\sup_{z\in \T}|f(z)|\leq  \|f\|_{0,0}\leq  \|f\|_{\eps,0}$. Let $B$ be the ball of radius $\frac{\eta\wedge d}{4 }/\big(1+\frac{\eta\wedge d}{4  }\big)$ centered at $f_0$ in ${\rm Hol}^\bullet_\eps(\D)$, in particular $\|f-f_0\|_{0,0}\leq \frac{\eta\wedge d}{4  }$ for $f\in B$. Then $\mc{O}\cap  B$ is a nonempty  open neighborhood of $f_0$ contained in $\mc{S}_\eps$.
\end{proof}
 
We consider the  vector space $\mc{C}_{\rm exp}(\R\times \Omega_\T)$ 
of functions $F:\R\times \Omega_\T\to \C$ so that there is $N>0$ such that $F(c,\varphi)=F(c,x_1,y_1, \cdots, x_N,y_N)$ for all $c\in\R$ with $F$ smooth, and 
\[ \forall k\in \N, \forall \alpha,\beta \in \N^{N}, \exists L\geq 0, \forall M\geq 0,  \exists C>0, 
\quad  |\pl_c^k \pl_{x}^\alpha\pl_y^\beta F(c,\varphi)|\leq Ce^{-M|c|}\cjg \varphi\cjd_N^{L}
\]
if $\pl_{x}^\alpha=\pl_{x_1}^{\alpha_1}\dots \pl_{x_N}^{\alpha_N}$ and $\cjg \varphi\cjd_N=(1+\sum_{|n|\leq N}|\varphi_n|^2)^{1/2}$ is the Japanese bracket. Also, by an abuse of notation, we still denote by $\cjg \cdot,\cdot  \cjd_{\mc{H}}$ the duality pairing between $\mc{D}(\mc{Q})$ and $\mc{D}'(\mc{Q})$.

The main result of this section is the following:
\begin{theorem}\label{th:differentiabiliteTf}
The mapping $f\in \mc{S}_\eps\mapsto \bT_f  \in \mc{L}(\mc{D}(\mc{Q}),\mc{D}(\mc{Q}))$   is differentiable with differential  
$$v\in  {\rm Hol}^\bullet_\eps(\D)\mapsto D_v\bT_f \in \mc{L}(\mc{D}(\mc{Q}),\mc{D}(\mc{Q}))$$
 characterised by 
$$\forall F\in \mc{D} (\mc{Q}),G\in \mc{D}'(\mc{Q}),\quad \cjg D_v\bT_fF,G\cjd_{\mc{H}} := -\mc{Q}_{\rm w}(F,\bT^*_f G) $$
where $\mc{Q}_{\rm w}$ is the closure of the $\mc{Q}$-continuous bilinear form defined on $\mc{C}_{\rm exp}(\R\times \Omega_\T)\times\mc{C}_{\rm exp}(\R\times \Omega_\T)$ by
\begin{equation} 
 \mc{Q}_{\rm w}(F,G):=\sum_nw_n\cjg\mathbf{L}_nF,G\cjd_{\mc{H}} +\sum_n\bbar{w}_n\cjg\widetilde{\mathbf{L}}_nF,G\cjd_{\mc{H}} 
\end{equation}
with ${\rm w}=w(z)\pl_z=v(z)/f'(z)\partial_z$ and $v(z)/f'(z)=-\sum_{n\geq 0} w_nz^{n+1}$ for $|z|=1$.  Equivalently,  $D_v\bT_f=-\bT_f\mathbf{H}_{\rm w}:\mc{D}(\mc{Q})\to \mc{D}(\mc{Q})$  with the operator  $\mathbf{H}_{\rm w}:\mc{D}(\mc{Q})\to \mc{D}'(\mc{Q})$ defined as in \eqref{defopHv}.
\end{theorem}
 
 The next subsection is devoted to the proof of this result.

 \subsection{Proof of Theorem \ref{th:differentiabiliteTf}}\label{app:differen}
 
\subsubsection{Quadratic forms}

First, we  extend the definition of quadratic forms $\mc{Q}_{\rm v}$ for  ${\rm v}=v(z)\partial_z $ a holomorphic vector field and we study the regularity properties of this map as a function of $v$.

Consider the quadratic form $\mc{Q}_0$ with domain $  \mc{D}(\mc{Q}_0)$ defined by \eqref{FQ} with $\mu=0$ (i.e. the potential term is turned off). Recall the definition of the operators ${\bf L}_n^0$ in \eqref{defL0}  and that there exist some constant $C>0$ such that for all $n \in \Z$ and all $F,G\in\mc{C}_{\rm exp}(\R\times \Omega_\T)$  (see  \cite[Lemma 2.1]{BGKRV})
\begin{equation}\label{bilinearboundLn}
\begin{split} 
| \cjg {\bf L}^0_n F,G \cjd_{\mc{H}}|\leq C(1+|n|)^{3/2}\mc{Q}_0(F)^{1/2}\mc{Q}_0(G)^{1/2},\\
| \cjg \tilde{{\bf L}}^0_n F,G \cjd_{\mc{H}}|\leq C(1+|n|)^{3/2}\mc{Q}_0(F)^{1/2}\mc{Q}_0(G)^{1/2}  
\end{split} \end{equation}
The operator ${\bf L}_n^0$ and $\tilde{\bf L}_n^0$ thus extend as bounded operators 
${\bf L}_n^0,\tilde{\bf L}_n^0: \mc{D}(\mc{Q}_0)\to \mc{D}'(\mc{Q}_0)$. As a direct consequence, the following holds:

\begin{corollary}\label{corfq1}
Let $v(z)=-\sum_{n\geq -1}v_nz^{n+1}\in  {\rm Hol}_\eps(\D)$  and ${\rm v}=v(z)\pl_z$  then the operator 
$$
\mathbf{H}^0_{\rm v}:=\sum_{n=-1}^\infty {\rm Re}(v_n)(\mathbf{L}^0_n +\widetilde{\mathbf{L}}^0_n  ) +{\rm i}\,{\rm Im}(v_n)(\mathbf{L}^0_n -\widetilde{\mathbf{L}}^0_n  ) $$
extends as a bounded operator $\mathbf{H}^0_{\rm v}: \mc{D}(\mc{Q}_0)\to \mc{D}'(\mc{Q}_0)$ with norm bounded by $C\|v\|_{\eps=0,3}$ for some $C>0$ independent of $v$.
\end{corollary}
This is a control of the free field part. Now we study the case with the Liouville potential, i.e. when $\mu>0$. For this, we consider  ${\rm v}=v(z)\partial_z $ with  $v \in  {\rm Hol}_\eps(\D)$ and the bilinear form on $\mc{C}_{\rm exp}(\R\times \Omega_\T)$, denoted by $(F,G)\mapsto  \langle V_{\rm v} F,G\rangle_{\mc{H}}$ where $V_{\rm v}:=\frac{1}{2\pi}\int_{0}^{2 \pi} e^{\gamma\varphi(\theta)}\d\varrho_v(\theta)$ with $\d\varrho_v(\theta)=-\Re(e^{-{\rm i}\theta}v(e^{{\rm i}\theta}))\d\theta$. Notice that the GMC potential $V_{\rm v}$ makes sense as a random variable when $\gamma<\sqrt{2}$. For $ \sqrt{2}\leq \gamma<2$, this bilinear form can be given a meaning via the Cameron-Martin theorem as in \cite[Section 5.2]{GKRV20_bootstrap}.

\begin{corollary}\label{fqcont}
Let $v(z)=-\sum_{n\geq -1}v_nz^{n+1}\in  {\rm Hol}_\eps(\D)$ and ${\rm v}=v(z)\pl_z$. Then  the quadratic form
\[\forall F,G \in\mc{D}(\mc{Q}),\quad \mc{Q}_{\rm v}(F,G):=\cjg \mathbf{H}^0_{\rm v}F,G\cjd_{\mc{H}}+\mu \langle V_{\rm v} F,G\rangle_{\mc{H}}\]
 is $\mc{Q} $-continuous  with bound
$$ | \mc{Q}_{\rm v}(F,G)|\leq C \|v\|_{\eps=0,3}\mc{Q}(F,F)^{1/2}\mc{Q}(G,G)^{1/2}.$$
In particular,  the mapping  $v\in{\rm Hol}_\eps(\D)  \mapsto \mc{Q}_{\rm v}\in \mathcal{L}(\mc{D}(\mc{Q} )\times \mc{D}(\mc{Q} ))$  is continuous if ${\rm v}=v(z)\pl_z$.
Set $\omega:={\rm Re}(v_0)$. We say that $v$ satisfies the coercivity condition if
 \begin{equation}\label{condQnice}
D(v):=\omega -C\big(|{\rm Im}(v_0)|^2+\sum_{n\geq  -1,n\not=0}(1+|n|)^{6} |v_n|^2\big)^{1/2}>0.  
\end{equation}
 The coercivity condition
implies that the quadratic form $\mc{Q}_{\rm v}$ is a coercive closed form, namely
\[
{\rm Re}\,\mc{Q}_{\rm v}(F,F)\ge D(v)\mc{Q}(F,F),\qquad F\in\mc{D}(\mc{Q}),
\]
with coercivity constant $D(v)$. In that case, there exists a unique contraction semigroup $e^{-t \mathbf{H}_{\rm v}}$ associated to $\mc{Q}_{\rm v}$. 
\end{corollary}

\begin{proof}
Recall that we denote by $V$ the quadratic form $V_{\rm v}$ with ${\rm v}=-z\pl_z$. Since $v$ is continuous we have
\begin{equation}\label{inegQ}
 |\langle V_{\rm v} F,G\rangle_{\mc{H}}|\leq C_v \langle V  F,F\rangle_{\mc{H}}^{1/2} \langle V  G,G\rangle_{\mc{H}}^{1/2}
\end{equation}
with $C_v=\sup_{\theta\in[0,2\pi]}|{\rm Re}(e^{-{\rm i}\theta}v(e^{{\rm i}\theta}))|$. Coupling this bound with Corollary \ref{corfq1}, we get the $\mc{Q} $-continuity of the quadratic form $\mc{Q}_{\rm v}$ with norm $C\|v\|_{\eps=0,3}$ for some $C>0$. Hence  the mapping   $v\in  {\rm Hol}_\eps(\D) \mapsto \mc{Q}_{\rm v}\in \mc{D}'(\mc{Q})\otimes \mc{D}'(\mc{Q})$ is continuous.

Set $\omega:={\rm Re}(v_0)$ and $\tilde{{\rm v}}:={\rm v}+\omega z\partial_z=\tilde{v}(z)\pl_z$. On $\mc{C}_{\rm exp}(\R\times \Omega_\T)\times \mc{C}_{\rm exp}(\R\times \Omega_\T)$, we have
\[ \mc{Q}_{\rm v}=\omega  \mc{Q} + \mc{Q}_{\tilde {\rm v} } .\]
Using the bound for $ \mc{Q}_{\tilde{\rm v}}$   we get that 
\[
{\rm Re}\,\mc{Q}_{\rm v}(F,F)
=
\omega\mc{Q}(F,F)+{\rm Re}\,\mc{Q}_{\tilde{\rm v}}(F,F)
\geq
(\omega-C\|\tilde v\|_{\eps=0,3})\mc{Q}(F,F).
\]
 Hence the quadratic form is coercive and $\mc{Q}$-continuous if $\omega-C\|\tilde v\|_{\eps=0,3}>0$, and this is equivalent to the bound \eqref{condQnice}.
 This means that $\mc{Q}_{\rm v}$ is strictly m-accretive \cite[Chapter VIII.6]{Reed-Simon}. By Theorem VIII.16 and the following lemma in \cite{Reed-Simon}, there is a unique closed operator extending $\bH_{\rm v}$ defined in a dense domain $\mc{D}(\bH_{\rm v})\subset\mc{D}(\mc{Q})$, with $(\bH_{\rm v}-\lambda)^{-1}$ invertible  if ${\rm Re}(\la)<0$, and resolvent 
 bound $\|(\bH_{\rm v}-\lambda)^{-1}\|_{\mc{H}\to \mc{H}}\leq (-{\rm Re}(\la))^{-1}$. By the Hille-Yosida theorem, $\bH_{\rm v}$ is the generator of a contraction semigroup denoted $e^{-t\bH_{\rm v}}$.
\end{proof}

Let us  now state a result about the action of the flow related to some vector field $v$ on $\mc{S}_>$,  the proof of which is essentially contained in \cite[Lemma 2.2]{BGKRV}.
\begin{lemma}\label{flowS}
Let ${\rm v}=v(z)\partial_z$, with $v\in {\rm Hol}_\eps(\D)$, be a holomorphic vector field. If $v$ satisfies the coercivity condition \eqref{condQnice} then $v$ is Markovian. If $v$ is Markovian then the flow $\partial_t f_t=v(f_t)$ preserves $\mathcal{S}$, namely $f_0\in \mc{S}_>$ implies $f_t\in \mc{S}_>$ for $t>0$. If $v\in {\rm Hol}^\bullet_\eps(\D)$ then the condition $f_0\in \mathcal{S}_\eps$ implies $f_t\in \mathcal{S}_\eps$ for all $t>0$. 
\end{lemma}
Finally we have the following:
\begin{lemma}\label{lemmaweird}
 If $f\in \mc{S}_\eps$ and $v\in{\rm Hol}^\bullet_\eps(\D)$, we set $f_t(z)=f^{-1}(f(z)+tv(z))$. Then, for some $\eps'\in (0,\eps)$ and $\delta>0$ sufficiently small, the map 
 \[h: (s,t)\in [0,\delta]^2\mapsto \frac{v\circ f_t^{-1}(e^s\cdot)}{f'(e^s\cdot)}\in {\rm Hol}^\bullet_{\eps'}(\D)\] 
 is continuous.
If $v/f'$ obeys the coercivity condition   \eqref{condQnice}, then $h(s,t)$ satisfies this condition as well for $s,t\in [0,\delta]$ with $\delta>0$ small enough, for some uniform coercivity constant $D$.
\end{lemma}
\begin{proof}
Since $f:\D^\circ \to f(\D^\circ)$ is a biholomorphism and extends biholomorphically over a neighborhood of the closed unit disk $\D$, we deduce that for $\eps^{(2)}>0$ small enough, the restriction of $f$ to $(1+\eps^{(2)})\D^\circ$ is a biholomorphism onto its image $f((1+\eps^{(2)})\D^\circ)$. If $\eps^{(1)}<\eps^{(2)}$ and $\delta^{(1)}>0$ is small enough, then, for all $t\in [0,\delta^{(1)}]$,  $f+tv$ is a biholomorphism on $(1+\eps^{(1)})\D^\circ$ with image contained in $f((1+\eps^{(2)})\D^\circ)$, and the map $t\mapsto f+tv\in {\rm Hol}^\bullet_{\eps^{(1)}}(\D)$ is continuous. Therefore, the map $f_t$ is well defined and the map $t\in [0,\delta^{(1)}]\mapsto f_t\in {\rm Hol}^\bullet_{\eps^{(1)}}(\D)$ is continuous.

Now we claim that there exists $r\in (0,\eps^{(1)})$ such that $(1+r)\D\subset f_t((1+\eps^{(1)})\D)$ for all $t\leq \delta^{(2)}$, for some $\delta^{(2)}>0$. We give a proof by contradiction. If not true, there are two sequences $(t_n)_n$ and $(x_n)_n$ such that $t_n\to 0$ as $n\to \infty$, $x_n\in\D$ and $(1+r)x_n\not\in f_{t_n}((1+\eps^{(1)})\D)$.  This means that $\forall z\in (1+\eps^{(1)})\D$, 
\begin{equation}\label{notequality}
f((1+r)x_n)\not =f(z)+t_nv(z).
\end{equation}
Since $\D$ is compact, we may assume that $x_n\to x\in \D$. Let us set $g_n(z)=f(z)-f((1+r)x_n)$ and $h_n(z)=t_nv(z)$ defined on $(1+\eps^{(1)})\D$. For $|z|=1+\eps^{(1)}$ and because $v$ is bounded, we have $|h_n(z)|\leq t_nM$ for some $M>0$ uniformly in $z$. Because $f$ is injective and continuous, we have $|g_n(z)|\geq \min_{|z'|=1+\eps^{(1)}}|f(z')-f((1+r)x)|-\alpha_n$ with $\alpha_n\to 0$ as $n\to\infty$ and $\min_{|z'|=1+\eps^{(1)}}|f(z')-f((1+r)x)|>0$. Therefore, for $n$ large enough so as to make $t_n$ small enough, we have 
\[\forall z \textrm{ such that }|z|=1+\eps^{(1)},\quad |(g_n+h_n)(z)-g_n(z)|=|h_n(z)|<|g_n(z)|\]
and by Rouch\'e's theorem, this implies that $g_n$ and $g_n+h_n$ have the same number of zeros inside the disk $(1+\eps^{(1)})\D$. Since $g_n$ admits one unique such a zero, so does $g_n+h_n$. This is a contradiction with \eqref{notequality}.

Finally,  we set $\delta=\min(\delta^{(1)},\delta^{(2)})$. The map $t\in [0,\delta]\mapsto f_t^{-1}\in {\rm Hol}^\bullet_{r}(\D)$ is thus continuous, and so is the map $h$ if we set $\eps'=e^{-\delta}(1+r)-1$. Since the map $v\mapsto D(v)$ (where $D$ is the constant appearing in \eqref{condQnice}) is continuous in ${\rm Hol}^\bullet_{\eps'}(\D)$, the coercivity of $h(s,t)$ for $s,t$ small enough follows by continuity.
\end{proof}

This setup for quadratic forms will be useful in the next subsection to express the differential of $\mathbf{T}_f$. Actually the remaining part of the proof follows a general argument so that we will formulate a general setup in order for this result to be used in similar contexts. From now on, we make the following assumptions, which are valid (and already proven) in our setup. 
\begin{assumption}\label{assu:gen}
We assume that:
\begin{enumerate}
\item\label{itemH}  There exists a measured space $(M,\mu_0)$ such that, if  $\mc{H}=L^2(M,\mu_0)$ is the Hilbert space of square integrable functions, we are given a symmetric quadratic form $\mc{Q}$ with domain  $\mc{D}(\mc{Q})\subset \mc{H}$, generating a continuous contraction semigroup of selfadjoint operators $(e^{-t\mathbf{H}})_{t\geq 0}$ on $\mc{H}$,  with selfadjoint generator $\mathbf{H}$. The semigroup also maps $ \mc{D}'(\mc{Q})\to \mc{D}(\mc{Q})$ continuously.
\item\label{itemDef} We are given a closed subset $\mathcal{AS}$     of $\mc{S}$, stable by composition, and then we consider the spaces $\mc{AS}_\eps=  \mathcal{AS}\cap\mc{S}_\eps$, $\mc{AS}_>:=  \mathcal{AS}\cap\mc{S}_>$, all of them equipped with the induced topology. Let $\mc{A}{\rm Hol}^\bullet(\D)\subset {\rm Hol}^\bullet(\D)$ be the tangent space of $\mathcal{AS}$ and $\mc{A}{\rm Hol}_\eps^\bullet(\D):=\mc{A}{\rm Hol}^\bullet(\D)\cap {\rm Hol}_\eps^\bullet(\D)$.
\item\label{itemFlow} The subspace  $\mathcal{AS}$   is preserved by the flow. More precisely, consider  $v\in  \mc{A}{\rm Hol}_\eps^\bullet(\D)$ Markovian, then the solution $f_t$ to the flow $\partial_tf_t=v(f_t)$ with $f_0\in \mathcal{AS}_\eps$ (resp. $f_0(z)=z$) belongs to  $\mathcal{AS}_\eps$ for all $t\geq 0$ (resp. for all $t>0$). 
\item\label{itemDila} If $f\in  \mathcal{AS}_\eps$ then for $t_0>0$ small enough, $e^{t_0}f(e^{t_0}\cdot)\in  \mathcal{AS}_{\eps'}$, where $(1+\eps')e^{t_0}=1+\eps $.
\item\label{itemCompo}  We are given a  mapping $f\in \mc{AS}\mapsto \big( \mathbf{T}_f:L^\infty(M,\mu_0)\to L^\infty(M,\mu_0)\big)$, obeying the composition rule $\mathbf{T}_{f_1}\circ\mathbf{T}_{f_2}=\mathbf{T}_{f_1\circ f_2}$.
\item\label{itemCont} For $f\in \mc{AS}_>$, the operator $\mathbf{T}_f$ extends to a bounded operator on $\mc{H}$, i.e. $\mathbf{T}_f\in \mc{L}(\mc{H},\mc{H})$ and the mapping $f\in  \mc{AS}_>\mapsto \mathbf{T}_f\in \mc{L}(\mc{H},\mc{H})$ is continuous.
\item\label{itemQcont} There exists a  linear map $v\in \bigcup_{\eps>0}\mc{A}{\rm Hol}_\eps^\bullet(\D)\mapsto \mc{Q}_{\rm v}$ where ${\rm v}=v(z)\pl_z$ and $\mc{Q}_{\rm v}$ is a continuous bilinear form on $\mc{D}(\mc{Q})^2$, equipped with the $\mc{Q}$-norm. For all  $\eps>0$, the restriction of this  map $v\in \mc{A}{\rm Hol}_\eps^\bullet(\D)\mapsto \mc{Q}_{\rm v}$ is  $\|\cdot\|_{\eps,k}$-continuous (for some $k\geq 0$).
\item\label{itemCoer} Let ${\rm v}=v(z)\pl_z$ with $v(z)=-\sum_{n\geq 0}v_nz^{n+1}\in  \mc{A}{\rm Hol}_\eps(\D)$. If $\omega:={\rm Re}(v_0)>0$ is large enough, in the sense that there exists $C>0$ and $s>0$ such that $v$ satisfies
\[
D(v):=\omega-C\big(|{\rm Im}(v_0)|^2+\sum_{n\geq -1,n\not=0}(1+|n|)^s|v_n|^2\big)^{1/2}>0
\]
 then  the bilinear form $\mc{Q}_{\rm v}$ generates a contraction semigroup $(e^{-t\mathbf{H}_{\rm v}})_{t>0}$ on $\mc{H}$, for some bounded operator $\mathbf{H}_{\rm v}:\mc{D}(\mc{Q})\to \mc{D}'(\mc{Q})$. Furthermore, we have in this case
\begin{equation*} 
\forall F\in \mc{D}(\mc{Q}),\quad {\rm Re}(\mc{Q}_{\rm v}(F,F))\geq D(v)\mc{Q} (F,F).
\end{equation*}
Finally, for $v(z)=-z$, we have $\mc{Q}_{\rm v}=\mc{Q}$.
\item\label{itemTf=} If $v\in  \mc{A}{\rm Hol}_\eps^\bullet(\D)$ satisfies the coercivity condition of item \eqref{itemCoer}, and $f_t$ denotes its flow with $f_0(z)=z$, then $e^{-t\mathbf{H}_{\rm v}}=\mathbf{T}_{f_t}$.
\item\label{itemWeird1} If $f\in \mc{AS}_\eps$ and $v\in\mc{A}{\rm Hol}^\bullet_\eps(\D)$ such that $v/f'$ satisfies the coercivity condition of item \eqref{itemCoer}, we set $f_t(z)=f^{-1}(f(z)+tv(z))$. We assume that for $\delta>0$ small enough and for some fixed $\eps'\in (0,\eps)$,  the map 
\[h: s,t\in [0,\delta]\mapsto \frac{v\circ f_t^{-1}(e^sz)}{f'(e^sz)}\in\mc{A}{\rm Hol}^\bullet_{\eps'}(\D)\] 
is continuous. 
\item\label{itemWeird2} If $f\in \mc{AS}_\eps$ and $v\in\mc{A}{\rm Hol}^\bullet_\eps(\D)$, then $-\omega zf'(z)+v(z)\in \mc{A}{\rm Hol}^\bullet_\eps(\D)$ for all $\omega>0$ large enough.
\end{enumerate}
\end{assumption}

All the above items are valid in our context with $\mathcal{AS}=\mathcal{S}$. We have slightly enlarged the setup in order to derive a similar result for the amplitudes of surfaces with a Neumann boundary, used to prove the conformal bootstrap for the boundary Liouville CFT \cite{GRW2026}.   We prove the following result. 
\begin{theorem}\label{th:gendiff}
Under Assumption \ref{assu:gen}, the mapping $f\in \mc{AS}_\eps\mapsto \bT_f  \in \mc{L}(\mc{D}(\mc{Q}),\mc{D}(\mc{Q}))$   is differentiable with differential  
\[v\in  \mc{A}{\rm Hol}^\bullet_\eps(\D)\mapsto D_{\rm v}\bT_f \in \mc{L}(\mc{D}(\mc{Q}),\mc{D}(\mc{Q}))\]
 characterised, for ${\rm v}=v(z)\pl_z$, by 
\[\forall F\in \mc{D} (\mc{Q}),G\in \mc{D}'(\mc{Q}),\quad \cjg D_{\rm v}\bT_fF,G\cjd_{\mc{H}} := -\mc{Q}_{{\rm w}}(F,\bT^*_f G) \]
where ${\rm w}=w(z)\pl_z$ with $w(z)\pl_z=v(z)/f'(z)\partial_z$.  Equivalently,  $D_v\bT_f=-\bT_f\mathbf{H}_{{\rm w}}:\mc{D}(\mc{Q})\to \mc{D}(\mc{Q})$.
\end{theorem}
 
Here, to prove Theorem \ref{th:differentiabiliteTf}, we apply the above Theorem with $\mathcal{AS}=\mathcal{S}$.  In a forthcoming work related to the study of amplitudes with mixed Neumann/Dirichlet boundaries, we will apply this result with  $\mathcal{AS}$ given by those $f\in\mc{S}$ such that $f(\bar{z})=\bar{f}(z)$. It is not clear  to us how to interpret the technical conditions of items \eqref{itemWeird1} and  \eqref{itemWeird2} but they are satisfied for all the applications we have in mind. From now on, we focus on the proof of Theorem \ref{th:gendiff}.
 
\subsubsection{Non-autonomous flows}\label{naf}
Note that in case we have a flow $\partial_tf_t=v(f_t)$ for some Markovian vector field $v(z)\partial_z\in \mc{A}{\rm Hol}^\bullet_\eps(\D)$ (see subsection \ref{sec:Markovian_operator}) then, by item  \eqref{itemTf=} of Assumption \ref{assu:gen}, the operator $\mathbf{T}_{f_t}F$ coincides with the semigroup $e^{-t\mathbf{H}_{\rm v}}F$. It is then easy to differentiate $\mathbf{T}_{f_t}F$ at $t>0$ with respect to $t$ using quadratic forms. Since  $f_t(z)=z+tv(z)+o(t)$, this amounts to differentiating $\bf T$ at $f(z)=z$ in the direction $v$.

So, if one aims to differentiate $\mathbf{T}_{f+tv}F$ with respect to $t$ where  $f\in \mc{AS}_\eps$ and $v\in  \mc{A}{\rm Hol}^\bullet_\eps(\D)$, it is natural to rewrite $\mathbf{T}_{f+tv}F=\mathbf{T}_f\circ \mathbf{T}_{f_t}F$ with $f_t(z):=f^{-1}(f(z)+tv(z))$  for $t\in [0, \delta)$. The ideal situation would be if we could express $f_t$ as the solution of some autonomous flow to differentiate this operator. However, this is too much to ask, as there may be some $f$ for which there is no such a flow. To overcome this issue, we observe that $f_t$ can be seen as the holomorphic solution of some explicit non-autonomous flow
\begin{equation}\label{naflow}
\partial_t f_t=w_t(f_t)\text{ with initial condition } f_0(z)=z,\quad \text{where }w_t(z)=\frac{v(f_t^{-1}(z))}{f'(z)}.
\end{equation}

Therefore $w_t$ can be seen as the infinitesimal vector field driving our flow. We have to make sure it is Markovian and generates (infinitesimally) a semigroup. This is the reason why we make the following assumptions: we consider the case    $v\in \mc{A}{\rm Hol}^\bullet_\eps(\D)$  such that $v/f' $ belongs to $ \mc{A}{\rm Hol}_\eps^\bullet(\D)$ and satisfies the coercivity condition of item  \eqref{itemCoer}.   Under this assumption and using item \eqref{itemWeird1}, for some $\eps'<\eps$, we can then find $\delta>0$ and $\eta>0$ such that 
\begin{align}\label{wtcoercif}
\forall t\leq \delta,\forall s\in [0,\eta]\quad  z\mapsto w_t(e^s z) \in \mc{A}{\rm Hol}_{\eps'}^\bullet(\D) \text{ satisfies the coercivity condition of item  \eqref{itemCoer} }\\
\text{for some uniform }  D>0\nonumber
\\
\text{ for some }\eps''\in (\eps',\eps),\,\, 
\text{ the mapping }  t\in [0, \delta]\mapsto  f_t^{-1}\in\mc{S}_{\eps''}\text{ is continuous}.\label{condmapc0} 
\end{align}
The argument for \eqref{condmapc0}  is the same as in the proof of Lemma \ref{lemmaweird}. Then  one may check that $(f_t(\D))_{t\in [0,\delta]}$ is decreasing for the inclusion and the restriction of $f_t$ to the unit circle $\T$ is an analytic diffeomorphism.

The strategy is then to approximate the non-autonomous flow by a piecewise autonomous flow in order to differentiate over each piece.  For this, fix   $n\geq 0$ and  set $t_k:=k\delta/n$ for $0\leq k\leq n$. We construct  recursively the flow $f_t^n$ for $t\in [0,\delta]$ as follows: we set $f^n_t=f^{n,k}_t$ for $t\in [t_k,t_{k+1}]$ where  $f^{n,k}$ is constructed inductively on $k$ as the solution of the autonomous flow
$$\forall t\in [t_k,t_{k+1}],\quad \partial_t f^{n,k}_t=w_{t_k}(f^{n,k}_t)\quad \text{ with initial condition } f^{n,k}_{t_k}(z)= f^{n,k-1}_{t_k}(z).$$
 
 \begin{lemma} 
For each $F\in\mc{D}(\mc{Q})$, the mapping $t\in [0,\delta]\mapsto \bT_{f^n_t}F$  converges   in $C([0,\delta],L^2)\cap L^2([0,\delta],\mc{D}(\mc{Q}))$, as $n\to \infty$, to the mapping $t\in [0,\delta]\mapsto \bT_{f_t}F $ satisfying, for ${\rm w}_t=w_t(z)\pl_z$,
\begin{equation}\label{integrated1}
\forall G\in\mc{D}(\mc{Q}), \quad \cjg \bT_{f_t}F,G\cjd_{\mc{H}}-\cjg F,G\cjd_{\mc{H}}=-\int_0^t\mc{Q}_{{\rm w}_r}( \bT_{f_r}F,G)\, \dd r.
\end{equation}
\end{lemma}

\begin{proof} To begin with, observe that for all $t\in [0,\delta]$ and $F\in \mc{D}(\mc{Q})$, $\bT_{f^n_t}F$ converges toward $\bT_{f_t}F$ in $L^2$ as $n\to \infty$. Indeed, it is direct  to show that for each $k\in\N$, $\sup_{t\in [0,\delta]}\|f^n_{t}-f_{t}\|_{\eps',k}\to 0$ as $n\to\infty$ (write the flow for $\partial_z^k f^n_t$ and use the Gronwall lemma as well as  \eqref{condmapc0}). 
From this, the convergence of $\bT_{f^n_t} F$  toward $\bT_{f_t} F$ in $L^2$ follows from item (6).

On each interval $[t_k,t_{k+1}]$, the flow is autonomous so that, using the items (7) and (8), we have
\[\forall G\in\mc{D}(\mc{Q}), \quad   \cjg \bT_{f^n_t}F,G\cjd_{\mc{H}}-\cjg\bT_{f^n_{t_k}}F,G\cjd_{\mc{H}}=
-\int_{t_k}^t\mc{Q}_{{\rm w}_{t_k}}( \bT_{f^n_r}F,G)\,\dd r.\]
Differentiating $\|{\bf T}_{f_t^n}F\|^2_{\mc{H}}$ with respect to $t$ and using item (8), one obtains the a priori estimate 
\[ \|\bT_{f^n_t}F\|_{\mc{H}}^2+2\int_{t_k}^t{\rm Re}\mc{Q}_{{\rm w}_{t_k}}(\bT_{f^n_r}F,\bT_{f^n_r}F)\,dr\leq \|\bT_{f^n_{t_k}}F\|_{\mc{H}}^2.\]
 Bootstrapping this estimate on each consecutive interval $[t_k,t_{k+1}]$ produces the relation
\begin{equation}\label{estT1}
\sup_{[0,\delta]}\|\bT_{f^n_t}F\|_{\mc{H}}^2+ 2\int_{0}^{\delta}{\rm Re}\mc{Q}_{{\rm w}^n_r}(\bT_{f^n_r}F,\bT_{f^n_r}F)\,dr\leq \|F\|_{\mc{H}}^2
\end{equation}
 where we have set $w^n_t(z):=w_{t_k}(z) $ for $t\in  [t_k,t_{k+1}]$.   Also, using the Chasles relation,  we have  
\begin{equation}\label{chasles}
\forall G\in\mc{D}(\mc{Q}), \quad   \cjg\bT_{f^n_t}F,G\cjd_{\mc{H}}-\cjg  F,G\cjd_{\mc{H}}= 
-\int_0^t\mc{Q}_{{\rm w}^n_{r}}( \bT_{f^n_r}F,G)\, {\rm d}r.
\end{equation}
The following formula holds for $n,m$ (differentiate in $t$ and use \eqref{chasles})
\[ \cjg {\bf T}_{f^n_t}F,{\bf T}_{f^m_t}F\cjd_{\mc{H}}=\|F\|_{\mc{H}}^2-\int_0^t \mc{Q}_{{\rm w}^n_r}({\bf T}_{f^n_r}F,{\bf T}_{f_r^m}F) {\rm d}r-\int_0^t \mc{Q}_{{\rm w}^m_r}({\bf T}_{f^m_r}F, \bT_{f^n_r}F) {\rm d}r.\] 
This implies, setting  $\Delta^{n,m}_t=\bT_{f^n_t}F-\bT_{f^m_t}F$,
\begin{align*}
\|\Delta^{n,m}_t\|_{\mc{H}}^2&+2\int_0^t {\rm Re}\mc{Q}_{{\rm w}^n_{r}}( \Delta^{n,m}_r,\Delta^{n,m}_r)\,{\rm d}r\\
=&-2\int_0^t{\rm Re}(\mc{Q}_{{\rm w}^n_{r}}-\mc{Q}_{{\rm w}^m_{r}})( \bT_{f^m_r}F,\Delta^{n,m}_r) \,{\rm d}r\\
\leq & 2C\sup_{t\in [0,\delta]}\|w^n_{t}-w^m_{t}\|_{\eps',k}\int_0^t\mc{Q}(\bT_{f^m_r}F,\bT_{f^m_r}F)^{1/2}\mc{Q}(\Delta^{n,m}_r,\Delta^{n,m}_r)^{1/2}\, {\rm d}r
\end{align*} 
 where we have used \eqref{itemQcont}  in the last line. Now we can use the coercivity of the quadratic form $\mc{Q}_{{\rm w}^n_{t}}$, in particular there is a constant $C'>0$ such that $C'\mc{Q}\leq {\rm Re}\mc{Q}_{{\rm w}^n_{t}}$ for all $t\in[0,\delta]$. We deduce that
\begin{align*}
\|\Delta^{n,m}_t\|_{\mc{H}}^2&+2C'\int_0^t \mc{Q} ( \Delta^{n,m}_r,\Delta^{n,m}_r)\, \dd r\\
\leq & \frac{1}{\lambda}C^2\sup_{t\in [0,\delta]}\|w^n_{t}-w^m_{t}\|_{\eps',k}^2 \int_0^t\mc{Q}(\bT_{f^m_r}F,\bT_{f^m_r}F) \, \dd r+\lambda\int_0^t \mc{Q} ( \Delta^{n,m}_r,\Delta^{n,m}_r)\, \dd r
\end{align*}
for arbitrary $\lambda>0$. Choosing $\lambda$ small enough so as to make $2C'-\lambda>0$ and using \eqref{estT1}, we deduce that there is another constant $C>0$ (independent on $n,t$) such that
\begin{equation}
\sup_{t\in [0,\delta]}\|\Delta^{n,m}_t\|_{\mc{H}}^2 +\int_0^\delta \mc{Q}( \Delta^{n,m}_r,\Delta^{n,m}_r)\, \dd r\leq C\sup_{t\in [0,\delta]}\|w^n_{t}-w^m_{t}\|_{\eps',k}^2\|F\|_{\mc{H}}^2.
\end{equation}
Observe that $\sup_{t\in [0,\delta]}\|w^n_{t}-w^m_{t}\|_{\eps',k}$ converges to $0$ as $n,m\to\infty$. Indeed,
observe that for $n \leq m$ ($C$ is a constant which may change from line to line)  
 \begin{align*}
 \sup_{t\in [0,\delta]}\|w^n_{t}-w^m_{t}\|_{\eps',k} & \leq C  \sup_{[0,\delta]}\|w^n_{t}-w^m_{t}\|_{\eps'',0}. \\
 & \leq C  \sup_{t\in [0,\delta]} \sup_{|z| = 1+\eps''} | w^n_{t}(z)-w^m_{t}(z)   |. \\
 & \leq C  \sup_{s,t \in [0,\delta], |t-s| \leq \frac{\delta}{n} } \sup_{|z| = 1+\eps''} | f_t^{-1}(z)-f_s^{-1}(z)    | 
 \end{align*} 
 where in the last line we have used the fact that $v$ is Lipschitz and that there exists $\eta >0$ such that $|f'(z)|> \eta$ on $|z| \leq 1+ \eps''$. By assumption on $\eps''$ (see \eqref{condmapc0}), the last quantity goes to $0$ as $n,m\to\infty$. This shows that the mapping $t\mapsto \bT_{f^n_t}F$ is Cauchy in $C([0,\delta], L^2)$  and    in $L^2([0,\delta],\mc{D}(\mc{Q}))$, from which our claim follows.
\end{proof}

\subsubsection{Differentiability}
Now we need to introduce some further material to differentiate the flow. For $s\in \R$ we introduce the (Sobolev like) norm
$$|F|_{\mathbf{H},s}^2:= \cjg\mathbf{H}^sF,F\cjd_{\mc{H}}.$$
Note that this quantity perfectly makes sense via the spectral theorem (for appropriate $F$). We consider the space $W^s$ as the closure of the space of $F\in \mc{H}$ such that  $|F|_{\mathbf{H},s}<\infty$
 with respect to the norm $|\cdot|_{\mathbf{H},s}$. Note that for $s=1$ we have $|F|_{\mathbf{H},1}^2=\mc{Q}(F,F)$ and $W^1=\mc{D}(\mc{Q})$. Observe the obvious inequality for $s\in \R$
\begin{equation}\label{ineqfq}
\forall F\in W^{2-s},G\in W^s,\quad |\mc{Q}(F,G)| \leq |F|_{\mathbf{H},2-s}|G|_{\mathbf{H},s}.
\end{equation}
 
Now we prepare the proof with a few preliminary results. Recall that we have assumed  \eqref{wtcoercif}  and   \eqref{condmapc0}. Furthermore, by  item \eqref{itemDila}, we may assume that $e^{t_0}f(e^{t_0}\cdot)\in \mc{AS}_{\eps'}$ with $e^{t_0}(1+\eps')=(1+\eps)$ for some  $t_0>0$ small enough. Without loss of generality, we may also assume that $\eta$ given by \eqref{wtcoercif} satisfies $\eta< t_0$. 
\begin{lemma}\label{estTfstar}
For $f\in \mc{AS}_\eps$, there   is some operator $\bT_{f}^\eta$ such that $\bT_{f}=\bT_{f}^\eta   e^{-\eta\mathbf{H}}$ with  the operator norm estimates for $s,s'\in\R$ 
\begin{align*}
\|\bT_{f}\|_{W^s\to W^{s'}}<+\infty  \quad \text{ and }\quad  \|\bT_{f}^\eta\|_{W^s\to W^{s'}}<+\infty.
\end{align*}
The same property and operator norm estimates hold for the adjoints $\bT_{ f}^*$ and $(\bT_{ f}^\eta)^*$.
\end{lemma}

\begin{proof} Consider the mapping $q_t(z):=e^{-t}z$. With $t_0$ and $\eps'$ defined as above, then $f=q_{t_0}\circ f_{t_0} \circ q_{t_0}$ where the function $f_{t_0}(z)=e^{t_0}f(e^{t_0}z)$ belongs to $\mc{A}\mc{S}_{\eps'}$. It follows, using item \eqref{itemCompo}, that $\bT_f=e^{-t_0\mathbf{H}}\bT_{f_{t_0}}e^{-t_0\mathbf{H}}$ since $\bT_{q_{t_0}}=e^{-t_0\mathbf{H}}$ by item \eqref{itemTf=}. Also, since $ f_{t_0}\in\mc{AS}_{\eps'}$, $f_{t_0}|_{|\T}$ is an analytic diffeomorphism and $f_{t_0}(\D)\subset \D^\circ$.  By item \eqref{itemH}, the operator $\bT_{f_{t_0}}$ maps continuously $W^0\to W^0$ (recall $W^0=L^2=\mc{H}$).

On the other hand and by the spectral theorem, it is readily seen that $e^{-t_0\mathbf{H}}:W^s\to W^{s'}$ continuously for any $s,s'\in\R$. We deduce that  $\bT_f=e^{-t_0\mathbf{H}}\bT_{f_{t_0}}e^{-t_0\mathbf{H}}$ maps continuously $W^s\to W^{s'}$.  The same property holds for the adjoint  $\bT_{ f}^*$.
Finally we can proceed similarly with the operator $\bT_{ f}^\eta:=  e^{-t_0\mathbf{H}}\bT_{f_{t_0}}e^{-(t_0-\eta)\mathbf{H}}  $, which satisfies $\bT_{f}= \bT_{f}^\eta e^{-\eta\mathbf{H}}$.
\end{proof}

Now we consider $\delta>0$ given by  \eqref{wtcoercif}. We claim:
\begin{lemma}\label{estTft}
There exists a constant $C>0$ such that for $t\in (0,\delta)$
\[\|\bT_{f_t}-{\rm Id}\|_{W^{1}\to W^{0} }\leq C t^{1/2}.\]
\end{lemma}

\begin{proof}
From \eqref{integrated1}, we have 
\[\partial_t\|\bT_{f_t}F-F\|_{\mc{H}}^2=2{\rm Re}\cjg\partial_t\bT_{f_t}F,\bT_{f_t}F-F\cjd_{\mc{H}}=-2{\rm Re}\, \mc{Q}_{{\rm w}_t}(\bT_{f_t}F,\bT_{f_t}F-F ),\] 
producing
\[ \|\bT_{f_t}F-F\|_{\mc{H}}^2=-2\int_0^t{\rm Re}\, \mc{Q}_{{\rm w}_r}(\bT_{f_r}F,\bT_{f_r}F-F )\,\dd r.\]
Now we consider a constant $C$ such that, for $t\in [0,\delta]$,  $C^{-1}\mc{Q}\leq {\rm Re}\mc{Q}_{{\rm w}_t}$ (by item \eqref{itemCoer} and \eqref{wtcoercif}) and  $\forall F,G\in\mc{D}(\mc{Q})$,    $|\mc{Q}_{{\rm w}_t}(F,G)|\leq C |F|_{\mathbf{H},1}|G|_{\mathbf{H},1}$ (by item  \eqref{itemWeird1} and item \eqref{itemQcont}). Therefore
\begin{align*}
\|\bT_{f_t}F-F\|_{\mc{H}}^2+2\int_0^t{\rm Re\, }\mc{Q}_{{\rm w}_r}(\bT_{f_r}F-F,\bT_{f_r}F-F )\,\dd r=&-2\int_0^t{\rm Re}\, \mc{Q}_{{\rm w}_r}(F,\bT_{f_r}F-F )\,\dd r\\
\leq & \frac{t}{ \lambda }|F|_{\mathbf{H},1}^2+ C^2\lambda \int_0^t  |\bT_{f_r}F-F|_{\mathbf{H},1}^2  \, \dd r
\end{align*}
for arbitrary $\lambda>0$.
We deduce
\[ \|\bT_{f_t}F-F\|_{\mc{H}}^2+(2C^{-1}- C^2\lambda )\int_0^t |\bT_{f_r}F-F|_{\mathbf{H},1}^2 \,dr\leq \frac{t}{ \lambda }|F|_{\mathbf{H},1}^2.\]
We complete the proof by choosing $\lambda$ small enough so as to make $(2C^{-1}- C^2\lambda)>0$.
\end{proof}

\begin{lemma}\label{commut}
 Let $t>0$ and $v\in \mc{A}{\rm Hol}^\bullet_\eps(\D)$, with $t<\ln(1+\epsilon)$, such that $v (e^s \cdot)$ satisfies the coercivity condition \eqref{condQnice} for $s\in [0,t]$.  Then the following holds true with ${\rm v}_t=v_t(z)\pl_z$ for $v_t(z)=e^{-t}v(e^{t}z)$
\[\forall F,G\in \mc{D}(\mc{Q}),\quad  \mc{Q}_{\rm v}(F,e^{-t\mathbf{H}}G)= \mc{Q}_{{\rm v}_t}(e^{-t\mathbf{H}}F, G).\]
\end{lemma}

\begin{proof}
For such a $t$, and $s>0$ small enough, we have
$$e^{-t\mathbf{H}} e^{-s\mathbf{H}_{\rm v}}=\bT_{q_t}\circ \bT_{h_s}$$ with $q_t(z)=e^{-t}z$ and $h_s$ is the flow defined by $\partial_s h_s(z)=v(h_s(z))$ with $h_0(z)=z$. By the composition rule, namely item \eqref{itemCompo}, we have 
$\bT_{q_t}\circ \bT_{h_s}=\bT_{q_t \circ  h_s}=\bT_{e^{-t} h_s}$. Consider now $h^t_s(z):=e^{-t} h_s(e^tz)$, which is the solution of the flow $\partial_s h^t_s(z)=v_t(h^t_s(z))$ with $\bar h^t_0(z)=z$ and $v_t(z):=e^{-t}v(e^tz)$. Therefore we have
$$\bT_{q_t}\circ \bT_{h_s}=\bT_{e^{-t} h_s}=\bT_{h^t_s\circ q_t }=  \bT_{h^t_s} \circ\bT_{q_t}=  e^{-s\mathbf{H}_{{\rm v}_t}}e^{-t\mathbf{H}}.$$
Now, by item \eqref{itemCoer} and for $v$ satisfying the coercivity condition, for all  $F,G\in \mc{D}(\mc{Q})$
\[\partial_s|_{s=0}\langle e^{-s\mathbf{H}_{\rm v}}F,G\rangle_{\mc{H}}=-\mc{Q}_{\rm v}(F,G).\]
Differentiating  in this way  the relation $e^{-t\mathbf{H}} e^{-s\mathbf{H}_{\rm v}}=e^{-s\mathbf{H}_{{\rm v}_t}}e^{-t\mathbf{H}}$, we prove the claim.
\end{proof}

\begin{remark}
The above statement is a particular case of a more general (beautiful) property of this  non-commutative semigroup: for Markovian $v, w\in {\rm Hol}_\eps(\D)$
\[e^{-t\mathbf{H}_{\rm v}}e^{-s\mathbf{H}_{\rm w}}= e^{-s\mathbf{H}_{{\rm w}_t}}e^{-t\mathbf{H}_{\rm v}}\] 
where $w_t(z):=\frac{w(f_t^{-1}(z))}{(f_t^{-1})'(z)}$ and $f_t$ is the flow $\partial_tf_t(z)=v(f_t(z))$ and $f_0(z)=z$. This relation uses the fact that, if $g_s$ is  the solution  of the flow $\partial_sg_s(z)=w(g_s(z))$ and $g_0(z)=z$ then $g^t_s:=f_t\circ g_s\circ f_t^{-1}$ is the solution of the flow $\partial_sg^t_s(z)=w_t(g^t_s(z))$ and $g^t_0(z)=z$ (notice that in the case $v=w$ then $w_t(z)=w(z)$, which makes our relation consistent with the commutative case).  
\end{remark}

Using the first two above lemmas we deduce $\bT_f \bT_{f_t}:W^1\to W^1$  for $t\geq 0$ small enough, therefore $\bT_{f+tv}\in \mc{L}(\mc{D}(\mc{Q}),\mc{D}(\mc{Q}))$ using $f\circ f_t=f+tv$ and \eqref{itemCompo}.
Also, using  \eqref{integrated1}, we can write for all $F\in\mc{D}(\mc{Q})$ and $G\in\mc{D}(\mc{Q})$
\begin{align*}
\cjg \bT_{f_t}F,G\cjd_{\mc{H}}-\cjg F,G\cjd_{\mc{H}}=&-\int_0^t\mc{Q}_{{\rm w}_r}( \bT_{f_r}F,G)\,\dd r\\
=&-t\mc{Q}_{{\rm w}_0}( F,G)-\int_0^t(\mc{Q}_{{\rm w}_r}(  F,G)-\mc{Q}_{{\rm w}_0}( F,G))\,dr -\int_0^t\mc{Q}_{{\rm w}_r}( \bT_{f_r}F-F,G)\,\dd r \\
=:&-t\mc{Q}_{{\rm w}_0}( F,G)-R^1_t(G)-R^2_t(G).
\end{align*}
In particular, for $G\in\mc{D}'(\mc{Q})$ (hence $\bT_f ^*G\in \mc{D}(\mc{Q})$ by Lemma \ref{estTfstar})
\begin{align*}
\cjg\bT_f \bT_{f_t}F,G\cjd_{\mc{H}}-\cjg\bT_fF,G\cjd_{\mc{H}}=&\cjg \bT_{f_t}F,\bT_f ^*G\cjd_{\mc{H}}-\cjg F,\bT_f^*G\cjd_{\mc{H}} \\
=& -t\mc{Q}_{{\rm w}_0}( F,\bT_f^*G)-R^1_t(\bT_f^*G)-R^2_t(\bT_f^*G).
\end{align*}
The term $\mc{Q}_{{\rm w}_0}( F,\bT_f^*G)$ expresses the differential: indeed, using the $\mc{Q}$-continuity of $\mc{Q}_{{\rm w}_0}$ (see item \eqref{itemQcont}) and Lemma \ref{estTfstar}, we have 
\[|\mc{Q}_{{\rm w}_0}( F,\bT_f^*G)|\leq C |F|_{\mathbf{H},1}|\bT_f^*G|_{\mathbf{H},1}\leq   C |F|_{\mathbf{H},1}|G|_{\mathbf{H},-1}.\]

By Riesz representation theorem, there exists an element denoted $D_{\rm v}\mathbf{T}_fF\in W^1$ such that $\cjg D_{\rm v}\mathbf{T}_fF,G\cjd_{\mc{H}}=-\mc{Q}_{{\rm w}_0}( F,\bT_f^*G)$ for all $G\in W^{-1}$.
We have to show that the remainder terms $R_t^1,R_t^2$ are $o(t)$ in operator norm $W^1\to W^{1}$. 
We have, using  Lemma \ref{commut} and Lemma \ref{estTfstar},
\begin{align*}
|R^2_t(\bT_f^*G)|& =\, |\int_0^t\mc{Q}_{{\rm w}_r}( \bT_{f_r}F-F,e^{-\eta\mathbf{H}}(\bT_f^\eta)^*G)\,\dd r|
\\
&= \, |\int_0^t\mc{Q}_{e^{-\eta} w_r(e^\eta\cdot)\pl_z}( e^{-\eta\mathbf{H}}(\bT_{f_r}F-F), (\bT_f^\eta)^*G)\,\dd r|
\\
& \leq  C\int_0^t | e^{-\eta\mathbf{H}}(\bT_{f_r}F-F)|_{\mathbf{H},1}| (\bT_f^\eta)^*G|_{\mathbf{H},1}\,\dd r
\end{align*}
 where we have used item \eqref{itemQcont} in the last line. Then, using Lemmas \ref{estTfstar} and \ref{estTft}, we deduce
\[\begin{split}
|R^2_t(\bT_f^*G) |
& \leq  C\int_0^t | \bT_{f_r}F-F|_{\mathbf{H},0}|  G|_{\mathbf{H},-1}\,\dd r\\
& \leq C\int_0^t r^{1/2}\,\dd r  | F|_{\mathbf{H},1}|G|_{\mathbf{H},-1}\\
& =Ct^{3/2}| F|_{\mathbf{H},1}|G|_{\mathbf{H},-1}.
\end{split}\]
Concerning $R^1_t$, we use the following lemma which shows that 
$$|R^1_t(\bT_f^*G)|\leq  C t^2| F|_{\mathbf{H},1}|\bT_f^*G|_{\mathbf{H},1}\leq  C t^2| F|_{\mathbf{H},1}|G|_{\mathbf{H},-1},$$ where we have used Lemma \ref{estTfstar} in the last inequality.  
  
\begin{lemma}
There exists  $C>0$ such that for all   $t\in [0,\delta)$ 
\[\forall F,G \in  \mc{D}(\mc{Q}) ,\quad |\mc{Q}_{{\rm w}_{t}}(F,G) -\mc{Q}_{{\rm w}_{0}}(F,G)|\leq Ct\mc{Q}(F,F)^{1/2}\mc{Q}(G,G)^{1/2}.\]
\end{lemma}
\begin{proof}
Let ${\rm w}'_t:=(\pl_tw_t)(z)\pl_z$. Item \eqref{itemQcont} above implies that  for $t\in [0,\delta)$, $|\mc{Q}_{{\rm w}'_t}(F,G)|\leq  C\mc{Q}(F,F)^{1/2}\mc{Q}(G,G)^{1/2}$. 
It follows that
\[|\mc{Q}_{{\rm w}_t}(F,G) -\mc{Q}_{{\rm w}_{0}}(F,G)|\leq \int_0^t |\mc{Q}_{{\rm w}'_s}(F,G)|ds\leq Ct\mc{Q}(F,F)^{1/2}\mc{Q}(G,G)^{1/2}.\qedhere\]
 \end{proof}
So far, we have shown that $f\in \mc{AS}_\eps\mapsto \bT_{f}$ has a right-directional derivative  in operator norm $W^1\to W^{1}$, denoted $D_r\bT_f(v)$, for all $v\in \mc{A}{\rm Hol}^\bullet_\eps(\D)$ such that  $v/f' $ satisfies the coercivity condition \eqref{condQnice}. Furthermore, it is characterized by  $\cjg D_r\bT_f(v)F,G\cjd_{\mc{H}}=-\mc{Q}_{{\rm w}_{0}}( F,\bT_f^*G)$ for $F\in W^1$ and $G\in W^{-1}$. 

The next step is to promote this right-directional derivative to a directional derivative.
For this, we use a corollary of the first form of the mean value theorem: it states that if a function $f:[a,b]\to E$, with $E$ normed vector space,   is continuous on $[a,b] $  with a right derivative on $(a,b)$ and, if the right derivative is continuous at $x_0\in (a,b)$, then $f$ is differentiable at $x_0$. So, we notice that the mapping $(f,v)\in \mc{AS}_\eps\times \mc{A}{\rm Hol}^\bullet_\eps(\D)\mapsto    \mc{Q}_{{\rm w}_{0}}( \cdot,\bT_f^*\cdot) \in \mc{L}(W^1\times W^{-1})$ is continuous.  
  Indeed,  for all $f$ in a neighborhood of $f_0\in\mc{AS}_\eps$,  
 $$\bT_f=e^{-t_0\mathbf{H}}\bT_{f_{t_0}}e^{-t_0\mathbf{H}}$$ 
 with $f_{t_0}(z) =e^{t_0}f(e^{t_0}z)\in \mathcal{AS}_{\eps'} $ with $1+\eps'=(1+\eps)e^{-t_0}$.  Since the mapping $f\in\mc{AS}_\eps \mapsto f_{t_0}\in \mc{AS}_{\eps'} $ is continuous, we get the continuity of $f\in\mc{AS}_\eps\mapsto \bT_{f_{t_0}}\in \mc{L}(L^2,L^2)$ by item \eqref{itemCont}. Now recall that $e^{-t_0\mathbf{H}}:W^{-1}\rightarrow L^2$ as well as $e^{-t_0\mathbf{H}}: L^2\rightarrow W^{1} $ so that we get the continuity of the mapping $f\in\mc{AS}_\eps\mapsto \bT_f\in \mc{L}(W^{-1},W^{1})$, hence of the mapping  $f\in\mc{AS}_\eps\mapsto \bT^*_f\in \mc{L}(W^{-1},W^{1})$. 
 Consequently, using item \eqref{itemQcont}, the mapping $(f,v)\in \mc{AS}_\eps\times \mc{A}{\rm Hol}^\bullet_\eps(\D)\mapsto    \mc{Q}_{{\rm w}_0}( \cdot,\bT_f^*\cdot) \in W^{-1}\otimes W^1$ is continuous. And so is the mapping $ (f,v)\in \mc{AS}_\eps\times \mc{A}{\rm Hol}^\bullet_\eps(\D)\mapsto  D_r\bT_f(v)\in\mc{L}(W^1,W^1)$.
Hence, we deduce that the mapping 
 $f\in \mc{AS}_\eps\mapsto \bT_{f}$ has a  directional derivative  in operator norm $W^1\to W^{1}$ for any direction $v\in  \mc{A}{\rm Hol}^\bullet_{\eps}(\D)$  such that $v/f'$ satisfies the coercivity condition \eqref{condQnice}. 

Now we extend the directional differentiability to arbitrary $v\in  \mc{A}{\rm Hol}^\bullet_{\eps}(\D)$. For such a $v$ we set $w(z):=-\omega zf'(z)+v(z)$ with $\omega>0$ chosen so as to make $w$ satisfy the coercivity condition \eqref{condQnice}  (note  that $w-v$  satisfies this condition too). By item \eqref{itemWeird2}, both $w$ and $w-v$ belong to $\mc{A}{\rm Hol}^\bullet_{\eps}(\D)$.
Recall that  $D\bT_f(v)\in \mc{L}(W^1,W^1)$  is characterized by  $\cjg D_r\bT_f(v)F,G\cjd_{\mc{H}}=-\mc{Q}_{{\rm w}_{0}}( F,\bT_f^*G)$ for $F\in W^1$ and $G\in W^{-1}$. We then get the following in $\mc{L}(W^1,W^1)$,  for some $s\in (0,1)$, 
 \begin{align*}
 \bT_{f+tv}-\bT_f-tD\bT_f(v)=& -(\bT_{f+tw}-\bT_{f+tv}-tD\bT_f(w-v))+\bT_{f+tw}-\bT_f-tD\bT_f(w)
 \\
=& -t\int_0^1 D\bT_{f+tw+st(v-w)}(w-v)\dd s+tD\bT_f(w-v)+o(t)
 \end{align*}
 where we have used that $\bT_f$ is differentiable in the direction $w$  and $w-v$. Also, by continuity of $f\mapsto D\bT_f(v-w)$, we deduce that $-t\int_0^1D\bT_{f+tw+st(v-w)}(w-v)\dd s+tD\bT_f(w-v)=o(t)$  in $\mc{L}(W^1,W^1)$ and therefore $\bT_f$ is differentiable in the direction $v$ for $v\in  \mc{A}{\rm Hol}^\bullet_{\eps}(\D)$. Since the directional differential is continuous on $\mc{AS}_\eps$ then $\bT$ is differentiable on $\mc{AS}_\eps$. This completes the proof of Theorem \ref{th:gendiff}, hence of Theorem \ref{th:differentiabiliteTf}.

 \subsection{Adjoint and differentiability of  annuli amplitudes }
 
We have seen above that the holomorphic annuli $\A_f$ for $f\in \mc{S}$ generate the non-negative Virasoro elements 
${\bf L}_n,\tilde{\bf L}_n$ for $n\geq 0$ under differentiation of their amplitudes. One may ask if the negative elements 
${\bf L}_n,\tilde{\bf L}_n$ for $n\leq 0$ can also appear as variations of annuli amplitudes and this is indeed the case if one computes the variations of the adjoint of the operator associated to $\A_f$ for $f\in \mc{S}$. Mixing both positive modes and negative modes is also possible, up to adding a central element (a constant times the identity) coming from the choice of metric on the annulus.

First, we are going to compute the adjoint of the operator $\bT_f$ for   $f\in \mc{S}_>$. Let $\iota:\hat{\C}\to\hat{\C}$ be the anticonformal involution $z\mapsto1/\bar{z}$. Let $\D^*=\iota(\D)$. Given $f$ holomorphic on $\D$ (resp. $\D^*$), we define $f^*:=\iota\circ f\circ\iota$ holomorphic on $\D^*$ (resp. $\D$). If $f$ fixes 0, then $f^*$ fixes $\infty$, and we have $(f^*)'(\infty)=1/\overline{f'(0)}$.

We claim

\begin{proposition}\label{propadjoint}
For $f\in \mc{S}_>$ and $g_f$ any admissible metric on $\A_f$, the operator $\bT_f^*:\mc{H}\to \mc{H}$ has the following expression
\begin{equation}
\sqrt{2}\pi e^{\frac{c_\mathrm{L}}{12}W(f,g_f)}  \bT_f^*G(\tilde{\varphi}_2)= \int \mc{A}_{\mathbb{A}_{f^*}    ,g_{f^*}, \bzeta_{f^*}}(\tilde{\varphi}_1,\tilde{\varphi}_2)G(\tilde{\varphi}_1) \,\dd\mu_0(\tilde{\varphi}_1)
\end{equation}
with the constant $W(f,g_f)$ given by \eqref{def_of_W}
and  $\mathbb{A}_{f^*} =\D^*\setminus (f^*(\D^*))^\circ$, $g_{f^*}=e^{-2{\chi_f}_*}g_{\mathbb{A}} $ where  $h_*(z)=\log \Big|\frac{z(f^*)'(z)}{f^*(z)}\Big|$ and ${\chi_f}_*=\chi_f\circ\iota$ (recall \eqref{defg_tadmissible}) and the parametrisation $\bzeta_{f^*}$ is given by $\zeta_1(e^{{\rm i}\theta})=e^{{\rm i}\theta}$ and $\zeta_2(e^{{\rm i}\theta})=f^*(e^{{\rm i}\theta})$.
\end{proposition}

\begin{proof}
Since $\bT_f$ has an integral kernel with respect to the measure $\mu_0$, then $\bT^*_f$ has an integral kernel too. More explicitly, the kernel of $\bT_f$ (for $f\in \mc{S}_>$) is given by Corollary \ref{cor:annulus_propagator}
\[
\sqrt{2}\pi e^{\frac{c_\mathrm{L}}{12}W(f,g_f)}{\bf T}_fF(\tilde{\varphi}_1)= 
\int \mc{A}_{\mathbb{A}_f,g_f,\bzeta_f}(\tilde{\varphi}_1,\tilde{\varphi}_2)F(\tilde{\varphi}_2) \,\dd\mu_0(\tilde{\varphi}_2)
.\]
 Therefore
\[\sqrt{2}\pi e^{\frac{c_\mathrm{L}}{12}W(f,g_f)}{\bf T}_f^*G(\tilde{\varphi}_2)= 
\int \mc{A}_{\mathbb{A}_f,g_f,\bzeta_f}(\tilde{\varphi}_1,\tilde{\varphi}_2)G(\tilde{\varphi}_1) \,\dd\mu_0(\tilde{\varphi}_1).
\]
Now we use the diffeomorphism invariance of amplitude (Proposition \ref{Weyl}) with the diffeomorphism $\iota$
$$ \mc{A}_{\mathbb{A}_f,g_f,\bzeta_f}(\tilde{\varphi}_1,\tilde{\varphi}_2)=\mc{A}_{\mathbb{A}_{f^*}    ,\iota_*(g_f),\iota\circ \bzeta_f}(\tilde{\varphi}_1,\tilde{\varphi}_2)$$
where   $\iota_*(g_f)$ is the pushforward of the metric $g_f$ defined by \eqref{defg_tadmissible}. A straightforward computation gives $\iota_*(g_f)=  g_{f^*}$.
\end{proof}

\begin{definition}[\textbf{Annulus}]\label{def:annulus}
An annulus with analytic parametrised boundary, one incoming and one outgoing,
 can be seen as a pair $\bs{f}:=(f_{\rm out},f_{\rm in})$ where $f_{\rm out},f_{\rm in}$ belong to ${\rm Hol}_\eps(\A)$ (recall the definition in Section \ref{sec:funct_space}) for some $\eps\in (0,1)$ such that $f_{\rm out}^*\in {\rm Hol}_\eps(\A)$ and the curves $f_{\rm in}(\T)$ and $f_{\rm out}^*(\T)$ are not  intersecting and form two disjoint Jordan curves.  We call $\A_{\bs{f}}$ the subset of $\hat \C$ corresponding to the connected component of   $\hat\C\setminus (f_{\rm in}(\T)\cup f^*_{\rm out}(\T))$ admitting both $f_{\rm in}(\T)$ and  $f^*_{\rm out}(\T)$ as boundary components.  The parametrisation  $\bzeta_{\bs{f}}$ of the boundary components is then given by  $\zeta_1(e^{{\rm i}\theta})=f_{\rm out}^*(e^{{\rm i}\theta})$ and $\zeta_2(e^{{\rm i}\theta})=f_{\rm in}(e^{{\rm i}\theta})$, which are respectively incoming and outgoing, and we put on $\A_{\bs{f}}$ the complex structure induced by $\C$. 
\end{definition}

\begin{definition}[\textbf{Equivalent annuli}]
  Two annuli with analytic parametrised boundary $\bs{f}^1$ and $\bs{f}^2$ are said to be equivalent  if there is a holomorphic map $\Phi$ defined on a neighborhood of $\A_{\bs{f}^1}$ such that $\Phi(\A_{\bs{f}^1})=\A_{\bs{f}^2}$ and such that $\Phi\circ f^1_{\rm in}=f^2_{\rm in}$ and $\Phi\circ (f^1_{\rm out})^*=(f^2_{\rm out})^*$.  
\end{definition}

\begin{definition}[\textbf{Model form}]
 We call 
  {\it annulus in model form}  an annulus with analytic parametrised boundary $\bs{f}:=(f_\mathrm{out},f_\mathrm{in})\in {\rm Hol}^\bullet_\eps(\D)\times  {\rm Hol}_\eps^\bullet (\D)$ which are biholomorphisms 
  $f_{\rm in}:(1+\eps)\D\to f_{\rm in}((1+\eps)\D)$ and $f_{\rm out}:(1+\eps)\D\to f_{\rm out}((1+\eps)\D)$  for some $\eps>0$. If in addition $f_\mathrm{in}$,  $f_\mathrm{out}$ belong to $\mc{S}_\eps$, i.e. $f_{\rm in}(\D)\subset \D^\circ$ and $f_{\rm out}(\D)\subset \D^\circ$, 
  we say that the model form splits, in the sense that the unit circle splits the annulus in two annuli $\A_{f_{\rm in}}$ and $\iota(\A_{f_{\rm out}})$.
\end{definition}

Notice that the normalization $f_\mathrm{in}(0)=0$ and $f^*_\mathrm{out}(\infty)=\infty$  holds true for model forms. 
\begin{lemma}\label{lem_model_form}
Let $\bs{f}:=(f_{\rm out},f_{\rm in})$ be an annulus with analytic parametrisation. Then there exists a biholomorphism $\Phi_{\bs{f}}: \A_{\bs{f}}\to \A_{\bs{\psi}}$ such that 
$\bs{\psi}=(\psi_{\rm out},\psi_{\rm in})=((\Phi_{\bs{f}} \circ f^*_{\rm out})^*,\Phi_{\bs{f}} \circ f_{\rm in})$ is an equivalent annulus in model form. The map $\Phi_{\bs{f}}$ and $\bs{\psi}$ are uniquely defined if one requires that $(\psi^*_{\rm out})'(\infty)=\lambda$ for $\lambda>0$ fixed. For such fixed  $\lambda$, the map 
\[ \bs{f}=(f_{\rm out},f_{\rm in})\mapsto \bs{\psi}=((\Phi_{\bs{f}} \circ f^*_{\rm out})^*,\Phi_{\bs{f}} \circ f_{\rm in})\]
is differentiable. If in addition $\bar{f}_{\rm in}(z)=f_{\rm in}(\bar{z})$ and  $\bar{f}_{\rm out}(z)=f_{\rm out}(\bar{z})$ for all $z$, then $\bbar{\Phi}_{\bs{f}}(\bar{z})=\Phi_{\bs{f}}(z)$ for all $z$ and $\bs{\psi}$ also satisfies $\bbar{\psi}_{\rm out}(\bar{z})=\psi_{\rm out}(z)$ and $\bbar{\psi}_{\rm in}(\bar{z})=\psi_{\rm in}(z)$.
\end{lemma}
\begin{proof} 
For $\eps>0$ small, let $\bs{f}^0=(f^0_{\rm out},f^0_{\rm in})\in {\rm  Hol}_\eps(\A)\times {\rm  Hol}_\eps(\A)$ be an annulus with analytic parametrisation and we will consider $\bs{f}=(f_{\rm out},f_{\rm in})$ in a small neighborhood of $\bs{f}^0$. First, consider an extension 
of $f_{\rm in}$, $f_{\rm out}$ in $(1+\eps)\D$ as smooth diffeomorphisms to their image. In a small neighborhood of $\bs{f}^0$, this can be done in a way that it depends in a $C^1$ fashion on $\bs{f}$ in the Fr\'echet topology of ${\rm  Hol}_\eps(\A)$. 
Notice that if in addition $\bar{f}_{\rm in}(z)=f_{\rm in}(\bar{z})$ and  $\bar{f}_{\rm out}(z)=f_{\rm out}(\bar{z})$ for all $z$ (we say they are symmetric by complex conjugation), 
we can take their extension to also satisfy this property.
There is a Beltrami differential $\mu_{\bs{f}}\in C^\infty(\hat{\C})$ such that 
\[ (f_{\rm in})_*|\dd z|^2=e^{w_{\bs{f}}}|\dd z+\mu_{\bs{f}} \dd\bar{z}|^2 \textrm{ in }f_{\rm in}(\D) , \quad 
(f^*_{\rm out})_*|\dd z|^2=e^{w_{\bs{f}}}|\dd z+\mu_{\bs{f}} \dd\bar{z}|^2 \textrm{ in }f^*_{\rm out}(\D^*) \]
where $w_{\bs{f}}\in C^\infty(f_{\rm in}(\D) \cup f_{\rm out}^*(\D^*) )$ and $\mu_{\bs{f}}$ has support in 
$f_{\rm in}(\delta \D)\cup f_{\rm out}^*(\delta^{-1}\D^*)\subset \hat{\C}\setminus \A_{\bs{f}}$. If in addition $f_{\rm in}, f_{\rm out}$ are symmetric by complex conjugation, then $\mu_{\bs{f}}$ is as well, i.e. $\bar{\mu}_{\bs{f}}(z)=\mu_{\bs{f}}(\bar{z})$.
The Beltrami coefficient $\mu_{\bs{f}}$ encodes a complex structure $J_{{\bs f}}$ on $\hat{\C}$ obtained by gluing the disk $\D$ to $\A_{\bs{f}}$ by identifying the boundary $\T$ to $f_{\rm in}(\T)$ via $f_{\rm in}$ and then glue another disk $\D^*$ to $\A_{\bs{f}}$ by identifying the boundary $\T$ to $f^*_{\rm out}(\T)$ via $f_{\rm out}$. The resulting Riemann surface is the sphere with two marked points $0,\infty$ and the complex structure $J_{{\bs f}}$. There is a uniformising biholomorphism
map $\Phi_{\bs{f}}:(\hat{\C},J_{\bs{f}})\to (\hat{\C},J_{\C})$ such that $\Phi_{\bs{f}}(0)=0$, $\Phi_{\bs{f}}(\infty)=\infty$, which is quasiconformal for the usual complex structure of $\hat{\C}$ 
\[ \pl_{\bar{z}}\Phi_{\bs{f}}=\mu_{\bs{f}} \pl_z \Phi_{\bs{f}}\]
and two such maps can only differ by a M\"obius transform fixing $0,\infty$, i.e. a map $z\mapsto \alpha z$ for $\alpha\in \C$. Notice that 
$\psi_{\rm in}=\Phi_{\bs{f}}\circ f_{\rm in}$ and $\psi_{\rm out}=(\Phi_{\bs{f}} \circ f^*_{\rm out})^*$ are in $\mc{S}_{\eps}$.
Post composing by a M\"obius transform depending smoothly on $\bs{f}$, we can thus fix a last condition on $\Phi_{\bs{f}}$ to make it uniquely defined, by asking that $\psi_{\rm out}^*:=\Phi_{\bs{f}} \circ f^*_{\rm out}$ satisfies $(\psi^*)'_{\rm out}(\infty)=\lim_{z\to \infty}(\psi_{\rm out}^*)'(z)=\lambda$ if $\lambda>0$ is a fixed number. 
The map $\Phi_{\bs{f}}$ solves a Beltrami equation with coefficients depending smoothly on $\bs{f}$ (in fact analytically), thus the solution $\Phi_{\bs{f}}$ is differentiable with respect to $\bs{f}$ by \cite{Ahlfors_Bers,NagBook}. If in addition $\mu_{\bs{f}}$ is symmetric by conjugation, then by uniqueness of the solution we must have $\bbar{\Phi}_{\bs{f}}(\bar{z})=\Phi_{\bs{f}}(z)$ for all $z\in \C$.
\end{proof}

For an annulus in model form $\bs{f}$, we consider the operator $\bT_{\bs{f}}$ defined by 
 \begin{equation}\label{def_Tf_general}
\bT_{\bs{f}}:=\bT_{f_{\rm out}}^*\circ \bT_{f_{\rm in}}.
\end{equation}
The operator $\bT_{\bs{f}}$ possesses an integral kernel that can be expressed in terms of annuli amplitudes:

\begin{lemma}\label{lemmamero}
Let  $\bs{f}=(f_{\rm out},f_{\rm in})$ be an annulus in model form which splits. The operator $\bT_{\bs{f}}$ of \eqref{def_Tf_general} has an integral kernel, which can be expressed in terms of annuli amplitudes as 
\begin{equation}\label{kernel*}
\sqrt{2}\pi e^{\frac{c_\mathrm{L}}{12}W({\bs{f}},g_{\bs{f}})}\bT_{\bs{f}}F(\tilde{\varphi}_1)=  \int \mc{A}_{\A_{\bs{f}},g_{\bs{f}},\bzeta_{\bs{f}}}(\tilde{\varphi}_1,\tilde{\varphi}_2)F(\tilde{\varphi}_2) \,\dd\mu_0(\tilde{\varphi}_2)
\end{equation}
if $g_{\bs{f}}$ is any admissible metric on the annulus $(\A_{\bs{f}},J_\C,\bs{\zeta}_{\bs{f}})$ and
\begin{equation}
W({\bs{f}},g_{\bs{f}}) := W(f_{\rm in},g_{\bs{f}})+W(f_{\rm out},\iota^*g_{\bs{f}})
\end{equation}
with $W(f,g)$ defined by \eqref{def_of_W}.
\end{lemma} 

\begin{proof} Since $(\A_{\bs{f}},g_{\bs{f}},\bs{\zeta}_{\bs{f}})$ is the gluing of the annulus $(\A_{f_{\rm in}},g_{\bs{f}}, ({\rm Id},f_{\rm in}))$ with $(\iota(\A_{f_{\rm out}}),g_{\bs{f}}, ({\rm Id},f_{\rm out}^*))$ along the unit circle $\T$,
 the result is a direct consequence of  Proposition \ref{glue1}
 and Proposition \ref{propadjoint}.
\end{proof}

Next we want to differentiate the operator $\bT_{\bs{f}}$ and amplitudes of  annuli with analytic parametrised boundaries. For this, we first extend  the definition of the operator $\mathbf{H}_{\rm v}$ in \eqref{defopHv} to the case of a holomorphic vector field in the annulus $\mathbb{A}_\delta$.
\begin{lemma}\label{lemHwmero}
Let $v\in {\rm Hol}_\eps (\A)$ of the form $v(z)=-\sum_{n\in\Z}v_nz^{n+1}$ with $\eps>0$ small. Then the series
\[\mathbf{H}_{\rm v}:=\sum_{n\in\Z}  v_n\bL_n+\sum_{n\in\Z}  \bar{v}_n\tilde{\bL}_n\]
converges in $\mathcal{L}(\mc{D}(\mc{Q}),\mc{D}'(\mc{Q}))$ if ${\rm v}=v(z)\pl_z$.
\end{lemma}

\begin{proof} As in the proof of Corollary \ref{fqcont}, we can use the bounds \eqref{bilinearboundLn} and \eqref{inegQ} 
to get the estimate
\[|\cjg \mathbf{H}_{\rm v} F,G\cjd_{\mc{H}}|\leq C\|{ v}\|_{\eps ,3}\mc{Q}(F)^{1/2}\mc{Q}(G)^{1/2}\]
for some constant $C>0$, which amounts to claiming the convergence of $\mathbf{H}_{\rm v} $ as a linear map $\mc{D}(\mc{Q})\to\mc{D}'(\mc{Q})$.
 \end{proof}
For $v(z)=-\sum_{n=0}^\infty v_nz^{n+1}$, we can write for ${\rm v}^*=v^*(z)\pl_z$ with 
 $v^*(z):=-z^2\bar{v}(1/\bar{z})=-\sum_{n=-\infty}^0\bar{v}_{-n}z^{n+1}$  
 \begin{equation}\label{H_adjoint}
 {\bf H}^*_{\rm v}=\sum_{n=0}^\infty (\bar{v}_n{\bf L}_{-n}+v_n\tilde{{\bf L}}_{-n})={\bf H}_{{\rm v}^*}.
 \end{equation}
As an easy consequence of Theorem \ref{th:differentiabiliteTf}, we obtain the following:
  \begin{proposition}\label{p:differentiabiliteTform}
The mapping $\bs{f}\in (\mc{S}_\eps)^2\mapsto \bT_{\bs{f}} \in \mc{L}(\mc{D}(\mc{Q}),\mc{D}(\mc{Q}))$   is differentiable with differential  
\[(v_{\rm out},v_{\rm in})\in  {\rm Hol}^\bullet_\eps(\D)^2\mapsto D_{\bf v}\bT_{\bs{f}} \in \mc{L}(\mc{D}(\mc{Q}),\mc{D}(\mc{Q}))\]
 characterised, for ${\bf v}=({\rm v}_{\rm out},{\rm v}_{\rm in})=(v_{\rm out}\pl_z,v_{\rm in}\pl_z)$, by 
\[\forall F\in \mc{D} (\mc{Q}),G\in \mc{D}'(\mc{Q}),\quad \cjg D_{\bf v}\bT_{\bs{f}}F,G\cjd_{\mc{H}} := -\mc{Q}_{{\rm w}_{\rm in}}(F,\bT^*_{\bs{f}} G)  -\overline{\mc{Q}_{{\rm w}_{\rm out}}(G,\bT_{\bs{f}}F) }\]
with $w_{\rm in}(z)=v_{\rm in}(z)/f_{\rm in}'(z) $ and $w_{\rm out}(z)=v_{\rm out}(z)/f_{\rm out}'(z)$ and, for ${\rm w}=w(z)\pl_z$ with $w(z)=-\sum_{n=0}^\infty w_nz^{n+1}$, 
$\mc{Q}_{{\rm w}}$ is the closure of the $\mc{Q}$-continuous bilinear form defined on $\mc{C}_{\rm exp}(\R\times \Omega_\T)\times\mc{C}_{\rm exp}(\R\times \Omega_\T)$ by
\begin{equation} 
 \mc{Q}_{{\rm w}}(F,G):=\sum_nw_n\cjg\mathbf{L}_nF,G\cjd_{\mc{H}} +\sum_n\bbar{w}_n\cjg\widetilde{\mathbf{L}}_nF,G\cjd_{\mc{H}}. 
\end{equation}
Equivalently,  $D_{\bf v}\bT_{\bs{f}}=-\bT_{\bs{f}}\mathbf{H}_{{\rm w}_{\rm in}}- \mathbf{H}_{{\rm w}_{\rm out}}^*\bT_{\bs{f}}:\mc{D}(\mc{Q})\to \mc{D}(\mc{Q})$  with the operator  $\mathbf{H}_{{\rm w}}:\mc{D}(\mc{Q})\to \mc{D}'(\mc{Q})$ defined as in  Lemma \ref{lemHwmero}.
\end{proposition}

Now we focus on the differentiability of annulus amplitudes when we parametrise one boundary with a biholomorphic map on an annulus (instead of a disk). The setup is similar to Subsection \ref{sub:intk}.  Let us   denote by ${\rm BiHol}_\eps (\A)$  the set of biholomorphisms on the annulus $\A_{\delta,\delta^{-1}}$ with $\delta=(1+\eps)^{-1}$ and $\eps>0$, which is an open subset of ${\rm Hol}_\eps(\A)$ equipped with the seminorms \eqref{seminorm}. 
Then we consider the open subset 
\begin{equation}\label{defSA}
\mc{S}_\eps(\A):=\{f\in {\rm BiHol}_\eps (\A)\,\,\, |\,f(\T)\subset\D^\circ\}
\end{equation}
 For $f\in \mc{S}_\eps(\A)$, we  define the curves $\mc{C}:=f(\T)\subset \D^\circ$,  the simply connected domain $\D_f$  as the bounded connected component of $\C\setminus \mc{C}$   and the annulus $\A_f$ by  $\mathbb{A}_f:=\D\setminus \D_f$.
 We equip the annulus $\mathbb{A}_f$ with the metric $g_{\mathbb{A}}$ (recall \eqref{gA})
and we consider the boundary parametrisations $\zeta_1(e^{{\rm i}\theta})=e^{{\rm i}\theta}$ for $\pl_1\A_f=\T$ 
and $\zeta_2(e^{{\rm i}\theta})=f(e^{{\rm i}\theta})$ for 
$\pl_2\mathbb{A}_f=\mc{C}$. 
The metric $g_{\mathbb{A}}$ has vanishing scalar curvature, it is not admissible for $\mathbb{A}_f$ near $\mc{C}$ 
but it is admissible near $\T$. Let us consider the harmonic function $h$ defined on a neighborhood of $ \mc{C}$ by \eqref{defh_t}
so that the metric $g_\mathbb{A}$ near $\mc{C}$ pulls-back through $f$ to 
$f^*g_\mathbb{A}= e^{2h\circ f} g_{\mathbb{A}}.$
Any admissible metric $g$ on $\mathbb{A}_f$ with parametrisation $\bzeta_f=(\zeta_1,\zeta_2)$ of the boundary as above can be written  in the form
\begin{equation}\label{def_gf}
 g_f=e^{-2 \chi_f}g_{\mathbb{A}}
 \end{equation}
where $\chi_f \in C^\infty(\mathbb{A}_f)$ is equal to $h$ near $\mc{C}$ and $0$ near $\T$. A possible choice, 
that depends smoothly on $f$ near a given $f_0\in \mc{S}_\eps(\A)$, is to take $\chi_f=h\chi$ where $\chi \in C^\infty(\mathbb{A}_f)$ is equal to $1$ near $\mc{C}$ and $0$ near $\T$ and does not depend on $f$ but only on $f_0$.

Given $f\in\mc{S}_\eps(\A) $,  we consider the operator $ \mc{A}_{\A_{f},g_f,\bs{\zeta}_{f}} : \mc{H}\to \mc{H}$  associated to the amplitude of the annulus $\A_f$, namely 
\[ (\mc{A}_{\A_{f},g_f,\bs{\zeta}_{f}}F)(\tilde{\varphi}_1):= \int \mc{A}_{\A_{f},g_f,\bs{\zeta}_{f}}(\tilde{\varphi}_1,\tilde{\varphi}_2)F(\tilde{\varphi}_2) \,\dd\mu_0(\tilde{\varphi}_2).\]
We study the differentiability of this operator at the specific point $f_0(z)=rz$ where  $r<1$ is fixed. Notice that $g_{f_0}=g_\A$ since $h(z)=0$ in that case (recall \eqref{defh_t}).

This result will be instrumental to differentiate the amplitudes with respect to holomorphic vector fields on the annulus since amplitudes always possess such regular annuli $\A_{f_0}=\A_{r}$ near their boundary circles.

\begin{theorem}\label{derivative_meromorphic_vf}
The operator $\mc{A}_{\A_{f},g_f,\bs{\zeta}_{f}}$ is bounded as  a map $\mc{D}(\mc{Q})\to \mc{D}(\mc{Q})$ and
the mapping $f\in\mc{S}_\eps(\A) \mapsto  \mc{A}_{\A_{f},g_f,\bs{\zeta}_{f}}\in\mc{L}( \mc{D}(\mc{Q}), \mc{D}(\mc{Q}))$ is differentiable at $f=f_0$ with $f_0(z)=rz$ for $0<r<1$, with differential given by 
\[ D_{\rm v}\mc{A}_{\A_{f_0},g_{\A},\bs{\zeta}_{f_0}}(F)=-c_{\rm L}\Big(\frac{{\rm Re}(v_0)}{12 r}+
D_{\rm v}S_{\rm L}^0(\A_{f_0},g_{f_0},g_{\A})\Big)\mc{A}_{\A_{f_0},g_{\A},\bs{\zeta}_{f_0}}(F)-\mc{A}_{\A_{f_0},g_{\A},\bs{\zeta}_{f_0}} \big( {\bf H}_{{\rm v}/r} F\big)\]
for $v=\sum_{n\in \Z}v_nz^{n+1}\in {\rm Hol}_\eps(\A)$ with ${\rm v}=v(z)\pl_z$.
\end{theorem} 

\begin{proof}  Let $ v\in{\rm Hol}_\eps (\A)$ for $\eps>0$ small and define, for $v$ small enough (for the distance induced by the seminorms \eqref{seminorm}), $f_v:=f_0+v$, which belongs to $\mc{S}_\eps(\A) $.
Let us consider the annulus $\A_{f_v}$ with boundaries parametrised by $\bs{\zeta}_{f_v}:=({\rm Id},f_v|_{\T})$, and its model form $\bs{f}_v=(f_{{\rm out},v},f_{{\rm in},v})$, normalised by requiring $(f_{{\rm out},v}^*)'(\infty)=a/r$ for some $a>0$ such that $a<1$ and $r/a<1$. Note that $f_{{\rm in},0}(z)=az$ and $f_{{\rm out},0}(z)=(r/a)z$ due to our normalisation and both of them belong to $\mathcal{S}_{\eps}$ for $\eps>0$ so that the model form splits for $v$ small enough. By Lemma \ref{lem_model_form}, there is a biholomorphism $\Phi$ defined over a neighborhood of $\A_{f_v}$ such that $f_{{\rm in},v}=\Phi\circ f_v$ and $f_{{\rm out},v}^*=\Phi$ on $\T$. Note that $f_{{\rm out},v}^*$ is holomorphic on $\D^*$ and the last equality entails that it extends holomorphically  over a neighborhood of the complement of $\D_{f_v}$. Furthermore, by Lemma \ref{lem_model_form},  the map  $v\mapsto (f_{{\rm out},v},f_{{\rm in},v})$ is differentiable near $v=0$ as  a map  ${\rm Hol}_\eps (\A)\to \mathrm{Hol}^\bullet_{\eps}(\D)^2$. We can apply the diffeomorphism invariance  in Proposition \ref{Weyl} to the diffeomorphism $f_{{\rm out},v}^*$ and we get
\[ \mc{A}_{\A_{f_v},g_{f_v},\bs{\zeta}_{f_v}}=  \mc{A}_{\A_{\bs{f}_v},(f_{{\rm out},v}^*)_*g_{f_v},\bs{\zeta}_{\bs{f}_v}}.\]
 Now using   Lemma \ref{lemmamero} we get
 \begin{equation} \label{Amp_Tf}
  \mc{A}_{\A_{f_v},g_{f_v},\bs{\zeta}_{f_v}}(F)=C(v)\bT_{\bs{f}_v} F
  \end{equation}
with 
\[C(v):= \sqrt{2}\pi e^{\frac{c_\mathrm{L}}{12}W({\bs{f}_v},(f_{{\rm out},v}^*)_*g_{f_v})}.\]
The function $C(v)$ is  differentiable  at $v=0$ and, by Theorem \ref{p:differentiabiliteTform}, $ \bT_{\bs{f}_v}$ is differentiable in $\mc{L}(\mc{D}(\mc{Q}),\mc{D}(\mc{Q}))$. This shows that the operator $F\mapsto \mc{A}_{\A_{f},g_{f},\bs{\zeta}_{f}}(F)$ is differentiable  in $\mc{L}(\mc{D}(\mc{Q}),\mc{D}(\mc{Q}))$. 

Now we express its directional derivative in the direction $v$. We shortcut the expression $f_0+tv$ as 
\[f_t:=f_0+tv\] 
and we consider the model form $(f_{{\rm out},t},f_{{\rm in},t})$ of the annulus $({\rm Id},f_0+tv|_{\T})$ (which splits for small $t$), normalized by $(f_{{\rm out},t}^*)'(\infty)=a/r$. Let us denote 
\[v_{\rm in}:=\lim_{t\to 0}(f_{{\rm in},t}-f_{{\rm in},0})/t\quad \text{ and }\quad v_{\rm out}:=\lim_{t\to 0}(f_{{\rm out},t}-f_{{\rm out},0})/t,\]
which satisfy both $v_{\rm in}(0)=v_{\rm out}(0)=0$, and let ${\rm v}_{\rm in}=v_{\rm in}(z)\pl_z$, ${\rm v}_{\rm out}=v_{\rm out}(z)\pl_z$.
   Then, differentiating the relation $f_t=(f_{{\rm out},t}^*)^{-1}\circ f_{{\rm in},t}$ at $t=0$ we get
\begin{equation}\label{vinout}
 v(z)=\frac{r}{a}v_{\rm in}(z)+raz^2\overline{v_{\rm out}(1/(r\bar{z}))}.
\end{equation}
Also, we write $C(t)$ for $C(tv)$. By differentiating at $t=0$ the relation 
\[  \mc{A}_{\A_{f_t},g_{f_t},\bs{\zeta}_{f_t}}(F)=C(t)\bT_{\bs{f}_t} F,\]
 we obtain, using Proposition \ref{p:differentiabiliteTform} and the relation $e^{\log(r){\bf H}}{\bf H}_{\rm v}= {\bf H}_{r{v}(\cdot/r)\pl_z}e^{\log(r){\bf H}}$,
\begin{align*}
 \pl_t   \mc{A}_{\A_{f_t},g_{f_t},\bs{\zeta}_{f_t}}(F)|_{t=0}= 
 &
C'(0)e^{\log(r){\bf H}}F- C(0)\big(\frac{1}{a}e^{\log(r){\bf H}}{\bf H}_{{\rm v}_{\rm in}}+ {\bf H}^*_{(a/r){\rm v}_{\rm out}}e^{\log(r){\bf H}}\big)F
 \\
= &
 C'(0)e^{\log(r){\bf H}}F-C(0)e^{\log(r){\bf H}} \big(\frac{1}{a} {\bf H}_{{\rm v}_{\rm in}}+ {\bf H}^*_{a{ v}_{\rm out}(\cdot/r)\pl_z} \big)F
   \\
= &
 C'(0)e^{\log(r){\bf H}}F-C(0)e^{\log(r){\bf H}} \big( \frac{1}{a}{\bf H}_{{\rm v}_{\rm in}}+ {\bf H}_{a z^2\overline{v_{\rm out}(\cdot/r\bar{z})}\pl_z} \big)F.
\end{align*}
Now by \eqref{Amp_Tf} we have $C(0)e^{\log(r){\bf H}}= \mc{A}_{\A_{f_0},g_{\A},\bs{\zeta}_{f_0}}$ and we compute using \eqref{def_of_W}, $g_{\bs{f}_t}:=(f_{{\rm out},t}^*)_*g_{f_t}$  (in particular $g_{\bs{f}_0}=g_\A$)
\[\begin{split} 
\frac{C'(0)}{C(0)}=&-\frac{c_{\rm L}}{12}{\rm Re}(\frac{v_{\rm in}'(0)}{a}+\frac{a}{r}v_{\rm out}'(0))-c_{\rm L}
\pl_t(S_{\rm L}^0(\A_{f_{\rm in,t}},g_{\bs{f}_t},g_\A)+S_{\rm L}^0(\iota(\A_{f_{\rm out,t}}),g_{\bs{f}_t},g_\A))|_{t=0}\\
=&-\frac{c_{\rm L}}{12}\frac{{\rm Re}(v_0)}{r}-c_{\rm L}
\pl_tS_{\rm L}^0(\A_{\bs{f}_t},g_{\bs{f}_t},g_{\bs{f}_0})|_{t=0}\\
=& -\frac{c_{\rm L}}{12}\frac{{\rm Re}(v_0)}{r}-c_{\rm L}
\pl_tS_{\rm L}^0(\A_{f_t},g_{f_t},g_{\A})|_{t=0}
\end{split}\]
where $v(z)=\sum_{n\in \Z}v_nz^{n+1}$.
We obtain the result using \eqref{H_adjoint} and \eqref{vinout}.
\end{proof} 

\subsection{Differentiability of surface amplitudes}
In this section, we explain how the differentiability theorem for annuli implies differentiability for general surfaces amplitudes with respect to the boundary parametrisation. We shall only focus on the case with incoming boundary for simplicity, the case with outgoing boundary can be deduced from this by composing the parametrisation $\zeta_j$ by the reflection of orientation map $z\in \T\mapsto 1/z\in \T$ and 
analysing its action on the Hilbert space $\mc{H}$. This will be explained in detail in \cite{BGKR2}.

We consider an admissible Riemannian surface $S_g=(\Sigma,g,{\bf x},\bs{\zeta})$ o
 with $b$ incoming boundary components, and $\bs{\alpha}$ some weights attached to ${\bf x}$. In order to give a sense to the differentiability of amplitudes of such Riemannian surfaces with respect to the boundary parametrisation $\bs{\zeta}$, 
we introduce a closed Riemann surface $(\hat{\Sigma},\hat{J})$ with marked points $\hat{{\bf x}}=(\bf{x}',\bf{x})$ called the \textbf{filling} of $\Sigma$ as follows: for $\delta<1$ close enough to $1$ so that $\zeta_j$ extends holomorphically in $\A_{1,\delta^{-1}}$,
 we glue $b$ disks $\mc{D}_1=\delta^{-1}\D,\dots, \mc{D}_b=\delta^{-1}\D$ to $\Sigma$ by setting  
\[ \hat{\Sigma}=(\Sigma \sqcup \bigsqcup_{j=1}^b \mc{D}_j)/\sim \]
where the equivalence relation is $u \sim \zeta_j(z)$ if $u=\zeta_j(z)$ for $z\in \A_{1,\delta^{-1}}$ and $u\in \zeta_j(\A_{1,\delta^{-1}})$.
 The complex structure $\hat{J}$ is  the one induced by the complex structure of $\Sigma$ and the canonical structure $|\dd z|^2$ on $\mc{D}_j=\delta^{-1}\D$, and the new marked points ${\bf x}'=(x'_1,\dots,x'_b)$ are the centers (i.e. the point $z=0$) of $\mc{D}_1,\dots,\mc{D}_b$. 

We define the Riemannian 
metric on $\hat{\Sigma}\setminus \{{\bf x}'\}$ \[ \hat{g}= g \textrm{ in }\Sigma, \quad \hat{g}=g_\A=\frac{|\dd z|^2}{|z|^2} \textrm{ in each }\mc{D}_j\setminus\{x_j'\}=\D\setminus \{0\}.\] 
Since $g$ is admissible on $\Sigma$, $\hat{g}$ is well-defined and smooth on $\hat{\Sigma}\setminus \{{\bf x}'\}$,  and it is compatible with the complex structure $\hat{J}$.
Up to taking $|\delta-1|$ possibly smaller, we can assume that $\zeta_j$ has a holomorphic extension  
$\zeta_j:\A_{\delta,\delta^{-1}}\to \hat{\Sigma}$ (it is biholomorphic onto its image, which is a collar neighborhood of $\pl \Sigma$)
 for all $j=1,\dots,b$.
For $\bs{f}\in {\rm Hol}_\eps(\A)^b$ close to ${\rm Id}=({\rm Id},\dots,{\rm Id})$ in ${\rm Hol}_\eps(\A)^b$ for some small $\eps>0$, 
we define the surface 
\[ \Sigma^{\bs f}=\hat{\Sigma}\setminus( \cup_{j=1}^b \mc{D}_j(f_j))\] 
 for $\mc{D}_j(f_j)\subset \mc{D}_j$ the open region containing $0$ 
and bounded by $f_j(\T)$,  and use the boundary parametrisation $\bs{\zeta}^{\bs{f}}=(\zeta_1\circ f_1,\dots,\zeta_b\circ f_b)$.  Consider 
$g^{\bs{f}}$  and admissible metrics on $\Sigma^{\bs{f}}$ with $g^{\bs{f}}=g$ outside a small neighborhood of $\pl \Sigma$,
 We define 
\[S_g\cdot \bs{f}:=(\Sigma^{\bs{f}},g^{\bs{f}},{\bf x},\bs{\zeta}^{\bs{f}})\] 
and view the family $\bs{f}$ of biholomorphisms as acting on the Riemannian surface $S_g$, although only the Riemann surface $S\cdot\bs{f}:=(\Sigma^{\bs{f}},\hat{J},{\bf x},\bs{\zeta}^{\bs{f}})$ is canonically defined if $S:=(\Sigma,J,{\bf x},\bs{\zeta})$. The action of $\bs{f}$ on the Riemannian surface $S_g$ involves to make a choice of admissible metric $g^{\bs{f}}$ for each $\bs{f}$. 

\begin{theorem}\label{diff_amplitudes}
Let $S_g=(\Sigma,g,{\bf x},\bs{\zeta})$ be an admissible Riemannian surface  with $b$ incoming boundary circles, $\bs{\alpha}$ some weights attached to ${\bf x}$, and let 
$(\hat{\Sigma},\hat{J})$ be the filling of $\Sigma$.
With the notations introduced just above, there is a neighborhood $\mc{U}$ of  ${\rm Id}=({\rm Id},\dots,{\rm Id})$ in ${\rm Hol}_\eps(\A)^b$ such that, if $\bs{f}\in \mc{U}\mapsto g^{\bs{f}}$ is a $C^1$ family of admissible metrics on $\Sigma^{\bs{f}}$ equal to $g$ outside a small neighborhood of $\pl \Sigma$, then 
\[   \mc{A}_{S_g\cdot ,\bs{\alpha}}:\bs{f} \in \mc{U}\mapsto \mc{A}_{S_g\cdot \bs{f},\bs{\alpha}}\in \mc{L}(\mc{D}(\mc{Q})^{\otimes b},\C)\]
is differentiable with differential at ${\rm Id}$ given by 
\[\begin{split} 
D_{\bv} \mc{A}_{S_g\cdot ,\bs{\alpha}}({\rm Id})F=&- \mc{A}_{S_g,\bs{\alpha}} \big( {\bf H}_{{\rm v}_1}\otimes {\rm Id}\otimes \dots \otimes {\rm Id}\big)F-\dots-\mc{A}_{S_g,\bs{\alpha}}\big( {\rm Id}\otimes\dots \otimes {\rm Id}\otimes {\bf H}_{{\rm v}_b}\big)F\\
& -c_{\rm L}\Big(\sum_{j=1}^b\frac{{\rm Re}(v_{j0})}{12}+
D_{\bv} \mc{S}_{S_g\cdot }({\rm Id}) \Big)\mc{A}_{S_g,\bs{\alpha}}F
\end{split}\]
for $\bv=({\rm v}_1,\dots,{\rm v}_b)$ with ${\rm v}_j=v_j(z)\pl_z$, $v_j\in {\rm Hol}_\eps(\A)$ and $v_j(z)=\sum_{n\in \Z}v_{jn}z^{n+1}$, and where $\mc{S}_{S_g\cdot }:\mc{U}\to \R$ is the map  defined by $\mc{S}_{S_g\cdot \bs{f}}:=S_{\rm L}^0(\Sigma^{\bs{f}},g^{\bs{f}},\hat{g})$.
\end{theorem}
\begin{proof}
Let $r\in (\delta,1)$ and denote $\lambda_r: z\mapsto rz$ and $\lambda_{r^{-1}}=\la_r^{-1}:z\mapsto r^{-1}z$ its inverse. 
Up to applying the Weyl anomaly formula, we can assume that $g^{\bs{f}}=g$ outside $\zeta_j(\A_{r,r^{-1}})$ with
 $g=(\zeta_j)_*g_\A$ near $\zeta_{j}(r^{-1}\T)$ and that $g^{\bs{f}}=(\zeta_j)_*(g_{f_j})$ only depends on $f_j$ on the annulus $\zeta_j(\A_{r,r^{-1}})\cap \Sigma^{\bs{f}}$ 
near $\pl_j\Sigma^{\bs{f}}$. 
We decompose $\Sigma^{\bs{f}}=\Sigma^0\cup \bigcup_{j=1}^b(\mc{D}_j(\la_r^{-1})\setminus \mc{D}_j(f_j))$ where $\Sigma^0=\Sigma
\setminus \cup_{j=1}^b\zeta_j^{-1}(\A_{1,r^{-1}})=\hat{\Sigma}\setminus \cup_{j=1}^b\mc{D}_j(\la_r^{-1})$ is the surface $\Sigma$ with $b$  annuli $\zeta_j(\A_{1,r^{-1}})$ removed 
at each boundary, and $\bs{\zeta}^0=(\zeta^0_{1},\dots,\zeta^0_{b})$ is the parametrisation of the boundary of $\Sigma^0$ given by 
$\zeta^0_{j}(e^{{\rm i}\theta}):=\zeta_j(r^{-1}e^{{\rm i}\theta})$.
With the parametrisation $\bs{\zeta}_{f_j}=(\la_r^{-1},f_j)$, we observe that the surface 
$(\mc{D}_j(\la_r^{-1})\setminus \mc{D}_j(f_j),g_{f_j},\bs{\zeta}_{f_j})$  is equivalent (by $\lambda_r:z\mapsto rz$) 
to the annulus $(\A_{rf_j}, (\lambda_r)_*g_{f_j}, r\bs{\zeta}_{f_j})$ and $(\la_r)_*\hat{g}=g_\A$. By Proposition \ref{glue1}, we get 
\[ \mc{A}_{S_g\cdot \bs{f},\bs{\alpha}}=\mc{A}_{\Sigma^0,g,{\bf x},\bs{\alpha},\bs{\zeta}^0}\circ \bigotimes_{j=1}^b\mc{A}_{\A_{rf_j},(\la_r)_*g_{f_j}, r\bs{\zeta}_{f_j}}.\]
Now we can apply Theorem \ref{derivative_meromorphic_vf} to $\mc{A}_{\A_{rf_j},(\la_r)_*g_{f_j}, r\bs{\zeta}_{f_j}}$: 
this gives the desired result since 
\[ \sum_{j=1}^bS_{\rm L}^0(\A_{rf_j},(\la_r)_*g_{f_j}, g_\A)=S_{\rm L}^0(\Sigma^{\bs{f}},g^{\bs{f}},\hat{g})\qedhere.\]
\end{proof}


\appendix

 \section{Sequence of Green functions}

\begin{lemma}\label{limit_Green}
Let $\Sigma$ be a closed surface or a surface with boundary and $g_n$ be a sequence of smooth Riemannian 
metrics  converging  to $g$ in $C^\infty(\Sigma)$ as $n\to \infty$. Consider the Green functions $G_{g_n}$ and $G_g$ of $\Delta_{g_n}$ and $\Delta_g$, with either Neumann or Dirichlet boundary condition on $\pl \Sigma$ if $\Sigma$ has a non-empty boundary and $ {\rm dv}_{g_n},{\rm dv}_g$ the volume densities associated to $g_n,g$. For each $N\in \N$, if $F_n\in H^N(\Sigma)$ is a sequence converging to $F$ in $H^N(\Sigma)$, let 
\[u_n(x)= \int_{\Sigma}G_{g_n}(x,x')F_n(x'){\rm dv}_{g_n}(x'), \quad u(x):=\int_{\Sigma}G_g(x,x')F(x'){\rm dv}_{g}(x')\]
then we have as $n\to \infty$
\[  \|u_n-u\|_{H^N(\Sigma)}\to 0.\]
\end{lemma}
\begin{proof}
We consider the case without boundary or when  boundary condition is Neumann, as the Dirichlet case is similar but even simpler since the Laplacian has no kernel in that case.
Since $g_n\to g$ in $C^\infty(\Sigma)$, we have $\|\Delta_{g_n}-\Delta_{g}\|_{H^2(\Sigma)\to L^2(\Sigma)}\to 0$. Let 
$\Pi_0$ be the $L^2$-projection on constants with respect to ${\rm dv}_g$, $\Pi_0f=\int_\Sigma f{\rm dv}_g$ and similarly $\Pi_0^{(n)}f=\frac{1}{{\rm v}_g(\Sigma)}\int_\Sigma f{\rm dv}_{g_n}$. Let $R_g=\Delta_{g}^{-1}(1-\Pi_0)$ be the operator on $L^2(\Sigma,{\rm dv}_g)$ with integral kernel given by the Green function $G_g$ of $\Delta_g$ on $\Sigma$, satisfying $\Delta_gR_g=1-\Pi_0$. We have 
\[ \Delta_{g_n}R_g=1-\Pi_0+(\Delta_{g_n}-\Delta_{g})R_g=1-\Pi_0^{(n)}+(\Pi_0^{(n)}-\Pi_0)+(\Delta_{g_n}-\Delta_{g})R_g\]
Multiplying on the left by $1-\Pi_0^{(n)}$ and using $\Pi_0^{(n)}\Delta_{g_n}=0$, this gives 
\[\Delta_{g_n}R_g=(1-\Pi_0^{(n)})(1+(\Pi_0^{(n)}-\Pi_0)+(\Delta_{g_n}-\Delta_{g})R_g)=:(1-\Pi_0^{(n)})(1+S_n).\]
We observe that for each $N\in \N$, $\|S_n\|_{H^N(\Sigma)\to H^N(\Sigma)}\to 0$ as $n\to \infty$ thus $(1+S_n)$ is invertible on 
$H^N(\Sigma,{\rm dv}_g)$ for $n$ large enough and we deduce that the inverse of $\Delta_{g_n}$ on the range of $1-\Pi_0^{(n)}$ is 
\[ R_{g_n}=(1-\Pi_0^{(n)})R_g(1+S_n)^{-1}(1-\Pi_0^{(n)})\]
and this shows that $R_{g_n}\to R_{g}$ in the $\mc{L}(H^N(\Sigma),H^{N+2}(\Sigma))$ operator norm as $n\to \infty$. The desired 
result follows since $R_{g_n}$ has $G_{g_n}(x,x')$ as integral kernel with respect to the ${\rm dv}_{g_n}$ measure. 
\end{proof}

\bibliographystyle{alpha}
\bibliography{virasoro}

@article{GRW2026,
	author = {{Guillarmou}, C. and {Rhodes}, R. and {Wu}, B.},
	journal = {Preprint.},
	title = {Conformal Bootstrap for surfaces with boundary in Liouville CFT. Part II: spectral resolution and bootstrap},
	year = {2026}}

@book{Berestycki_lqggff,
	author = {Berestycki, N. and Powell, E.},
	doi = {10.1017/9781009405492},
	fseries = {Cambridge Studies in Advanced Mathematics},
	isbn = {978-1-00-940550-8; 978-1-00-940549-2},
	language = {English},
	publisher = {Cambridge: Cambridge University Press},
	series = {Camb. Stud. Adv. Math.},
	title = {Gaussian free field and {Liouville} quantum gravity},
	volume = {220},
	year = {2026},
	zbmath = {8097122}}

@inbook{Teschner_Teich,
	author = {Teschner, J.},
	editor = {Papadopoulos, A.},
	pages = {685--760.},
	publisher = {EMS Publishing House, Berlin},
	title = {An analog of a modular functor from quantized Teichm{\"u}ller theory, Handbook of Teichmuller theory Vol 1},
	volume = {Handbook of Teichmuller theory Vol 1},
	year = {2007}}

@article{Teschner04,
	author = {Teschner, J.},
	journal = {International Journal of Modern Physics A},
	number = {No. supp 02},
	pages = {459--477},
	title = {On the relation between quantum {L}iouville Theory and the Quantized {T}eichm\"uller space},
	volume = {19},
	year = {2004}}

@article{BGKR2,
	author = {Baverez, G. and Guillarmou, C. and Kupiainen, A. and Rhodes, R. and Xie, Y.},
	journal = {Preprint.},
	title = {Virasoro conformal blocks and modular functor from {L}iouville {CFT}},
	year = {2026}}

@article{Ahlfors_Bers,
	author = {Ahlfors, L. and Bers, L.},
	journal = {Annals of Math.},
	number = {2},
	pages = {385--404},
	title = {Riemann's Mapping Theorem for Variable Metrics},
	volume = {72},
	year = {1960}}

@book{NagBook,
	author = {Nag, S.},
	publisher = {Wiley Interscience},
	title = {The complex analytic theory of Teichm\"uller space},
	year = {1988}}

@book{Reed-Simon,
	author = {Reed, M. and Simon, B.},
	publisher = {Academic Press},
	title = {Methods of Modern Mathematical Physics},
	volume = {1},
	year = {1980}}

@article{carron,
	author = {Carron, G.},
	journal = {Amer. J. Math.},
	number = {2},
	pages = {307--352},
	title = {D\'{e}terminant relatif et la fonction {X}i},
	volume = {124},
	year = {2002}}

@article{BGKRV,
	author = {Baverez, G. and Guillarmou, C. and Kupiainen, Ai and Rhodes, R. and Vargas, V.},
	doi = {10.2140/pmp.2024.5.269},
	fjournal = {Probability and Mathematical Physics},
	issn = {2690-0998},
	journal = {Probab. Math. Phys.},
	language = {English},
	number = {2},
	pages = {269--320},
	title = {The {Virasoro} structure and the scattering matrix for {Liouville} conformal field theory},
	volume = {5},
	year = {2024},
	zbl = {1563.81111},
	zbmath = {7870782}}

@article{DKRV16,
	author = {David, F. and Kupiainen, A. and Rhodes, R. and Vargas, V.},
	doi = {10.1007/s00220-016-2572-4},
	fjournal = {Communications in Mathematical Physics},
	issn = {0010-3616},
	journal = {Comm. Math. Phys.},
	mrclass = {81T20},
	mrnumber = {3465434},
	number = {3},
	pages = {869--907},
	title = {Liouville quantum gravity on the {R}iemann sphere},
	url = {https://doi.org/10.1007/s00220-016-2572-4},
	volume = {342},
	year = {2016}}

@article{DRV16_tori,
	author = {David, F. and Rhodes, R. and Vargas, V.},
	journal = {J. Math. Phys.},
	number = {2},
	pages = {022302, 25},
	title = {Liouville quantum gravity on complex tori},
	volume = {57},
	year = {2016}}

@article{DornOtto94,
	author = {Dorn, H. and Otto, H.-J.},
	doi = {10.1016/0550-3213(94)00352-1},
	fjournal = {Nuclear Physics. B. Theoretical, Phenomenological, and Experimental High Energy Physics. Quantum Field Theory and Statistical Systems},
	issn = {0550-3213},
	journal = {Nuclear Phys. B},
	mrclass = {81T30 (81T40)},
	mrnumber = {1299071},
	number = {2},
	pages = {375--388},
	title = {Two- and three-point functions in {L}iouville theory},
	url = {https://doi.org/10.1016/0550-3213(94)00352-1},
	volume = {429},
	year = {1994}}

@article{GKRV20_bootstrap,
	author = {Guillarmou, C. and Kupiainen, A. and Rhodes, R. and Vargas, V.},
	doi = {10.4310/ACTA.2024.v233.n1.a2},
	fjournal = {Acta Mathematica},
	issn = {0001-5962},
	journal = {Acta Math.},
	language = {English},
	number = {1},
	pages = {33--194},
	title = {Conformal bootstrap in {Liouville} theory},
	volume = {233},
	year = {2024},
	zbmath = {7983478}}

@article{GKRV21_Segal,
	author = {{Guillarmou}, C. and {Kupiainen}, A. and {Rhodes}, R. and {Vargas}, V.},
	journal = {Annals of Math, to appear},
	title = {{Segal's axioms and bootstrap for Liouville Theory}}}

@article{Guillarmou2019,
	author = {Guillarmou, C. and Rhodes, R. and Vargas, V.},
	fjournal = {Publications Math\'{e}matiques. Institut de Hautes \'{E}tudes Scientifiques},
	journal = {Publ. Math. Inst. Hautes \'{E}tudes Sci.},
	pages = {111--185},
	title = {Polyakov's formulation of {$2d$} bosonic string theory},
	volume = {130},
	year = {2019}}

@article{Kahane85,
	author = {Kahane, J-P.},
	journal = {Ann. Sci. Math. Qu\'{e}bec},
	number = {2},
	pages = {105--150},
	title = {Sur le chaos multiplicatif},
	volume = {9},
	year = {1985}}

@article{Ner87,
	author = {Neretin, Yu. A.},
	date = {1987/04/01},
	doi = {10.1007/BF01078036},
	id = {Neretin1987},
	isbn = {1573-8485},
	journal = {Functional Analysis and Its Applications},
	number = {2},
	pages = {160--161},
	title = {A complex semigroup that contains the group of diffeomorphisms of the circle},
	url = {https://doi.org/10.1007/BF01078036},
	volume = {21},
	year = {1987}}

@article{Ner1990,
	author = {Yu A Neretin},
	doi = {10.1070/SM1990v067n01ABEH001321},
	journal = {Mathematics of the USSR-Sbornik},
	month = {feb},
	number = {1},
	pages = {75},
	title = {HOLOMORPHIC EXTENSIONS OF REPRESENTATIONS OF THE GROUP OF DIFFEOMORPHISMS OF THE CIRCLE},
	url = {https://dx.doi.org/10.1070/SM1990v067n01ABEH001321},
	volume = {67},
	year = {1990}}

@article{OsgoodPS88,
	author = {B Osgood and R Phillips and P Sarnak},
	doi = {https://doi.org/10.1016/0022-1236(88)90070-5},
	issn = {0022-1236},
	journal = {Journal of Functional Analysis},
	number = {1},
	pages = {148-211},
	title = {Extremals of determinants of {L}aplacians},
	url = {https://www.sciencedirect.com/science/article/pii/0022123688900705},
	volume = {80},
	year = {1988}}

@article{Polyakov81,
	author = {Polyakov, A. M.},
	doi = {10.1016/0370-2693(81)90743-7},
	fjournal = {Physics Letters. B. Particle Physics, Nuclear Physics and Cosmology},
	issn = {0370-2693},
	journal = {Phys. Lett. B},
	mrclass = {81E99 (58D30 81G05 82A68)},
	mrnumber = {623209},
	number = {3},
	pages = {207--210},
	title = {Quantum geometry of bosonic strings},
	url = {https://doi.org/10.1016/0370-2693(81)90743-7},
	volume = {103},
	year = {1981}}

@article{Ray-Singer,
	author = {Ray, D. and Singer, I.},
	journal = {Advances in Mathematics},
	number = {2},
	pages = {145--210.},
	title = {R-torsion and the {L}aplacian on {R}iemannian manifolds.},
	volume = {7},
	year = {1971}}

@article{rhodes2014_gmcReview,
	author = {Rhodes, R. and Vargas, V.},
	journal = {Probab. Surv.},
	pages = {315--392},
	title = {Gaussian multiplicative chaos and applications: a review},
	volume = {11},
	year = {2014}}

@incollection{Segal87,
	author = {Segal, G. B.},
	booktitle = {Differential geometrical methods in theoretical physics ({C}omo, 1987)},
	mrclass = {58D30 (14H15 32G15 81E05 81E40)},
	mrnumber = {981378},
	mrreviewer = {Yukihiko Namikawa},
	pages = {165--171},
	publisher = {Kluwer Acad. Publ., Dordrecht},
	series = {NATO Adv. Sci. Inst. Ser. C Math. Phys. Sci.},
	title = {The definition of conformal field theory},
	volume = {250},
	year = {1988}}

@inproceedings{Teschner03,
	author = {Teschner, J.},
	booktitle = {Proceedings of the 35th {I}nternational {S}ymposium {A}hrenshoop on the {T}heory of {E}lementary {P}articles ({B}erlin-{S}chm\"{o}ckwitz, 2002)},
	doi = {10.1002/prop.200310109},
	fjournal = {Fortschritte der Physik. Progress of Physics},
	issn = {0015-8208},
	journal = {Fortschr. Phys.},
	mrclass = {32G81 (32G15 81T40)},
	mrnumber = {2008006},
	mrreviewer = {David Radnell},
	number = {7-8},
	pages = {865--872},
	title = {Quantum {L}iouville theory versus quantized {T}eichm\"{u}ller spaces},
	url = {https://doi.org/10.1002/prop.200310109},
	volume = {51},
	year = {2003}}

@article{Teschner_revisited,
	author = {J. Teschner},
	doi = {10.1088/0264-9381/18/23/201},
	journal = {Classical and Quantum Gravity},
	month = {nov},
	number = {23},
	pages = {R153--R222},
	publisher = {{IOP} Publishing},
	title = {Liouville theory revisited},
	url = {https://doi.org/10.1088/0264-9381/18/23/201},
	volume = {18},
	year = 2001}

@article{Zamolodchikov96,
	author = {Zamolodchikov, A. and Zamolodchikov, Al.},
	doi = {10.1016/0550-3213(96)00351-3},
	fjournal = {Nuclear Physics. B. Theoretical, Phenomenological, and Experimental High Energy Physics. Quantum Field Theory and Statistical Systems},
	issn = {0550-3213},
	journal = {Nuclear Phys. B},
	mrclass = {81T40},
	mrnumber = {1413469},
	mrreviewer = {Kazuto Oshima},
	number = {2},
	pages = {577--605},
	title = {Conformal bootstrap in {L}iouville field theory},
	url = {https://doi.org/10.1016/0550-3213(96)00351-3},
	volume = {477},
	year = {1996}}
\end{document}